\documentclass
{siamltex}
\usepackage{amsfonts}
\usepackage{amsmath}
\usepackage{amssymb}
\usepackage{array}
\usepackage{stmaryrd}
\usepackage{graphicx}
\usepackage{color}

\usepackage{mathabx} 

\usepackage{bbm}

\usepackage{graphicx}
\usepackage{ifpdf}
\usepackage{tikz}

\newtheorem{remark}[theorem]{\it Remark}

\numberwithin{equation}{section}

    \usepackage{color}  
     \definecolor{red}{rgb}{0.9,0,0}
     \definecolor{green}{rgb}{0,0.6,0}
     \definecolor{rb}{rgb}{0.6,0,0.2}     
     \definecolor{blue}{rgb}{0,0,0}
\renewcommand{\preccurlyeq}{\lesssim}
\renewcommand{\sim}{\simeq}
\renewcommand{\epsilon}{\varepsilon}
\newcommand{\dt}{\noindent\dotfill}

\newcommand{\pt}{\partial}
\newcommand {\eps} {\varepsilon}

\newcommand{\R}{\mathbb{R}}

\newcommand{\btau}{\boldsymbol{\tau}}
\newcommand{\btauz}{\btau_{\!z}}
\newcommand{\btauT}{\btau_{\!T}}
\newcommand{\btauS}{\btau_{\!S}}
\newcommand{\bmu}{\boldsymbol{\mu}}
\newcommand{\bnu}{{\boldsymbol{\nu}}}
\newcommand{\Tf}{{\mathcal T}^{\! f}}
\newcommand {\beq} {\begin{equation}}
\newcommand {\eeq} {\end{equation}}
\newcommand {\beqa} {\begin{eqnarray}}
\newcommand {\eeqa} {\end{eqnarray}}
\newcommand {\beqann} {\begin{eqnarray*}}
\newcommand {\eeqann} {\end{eqnarray*}}

\title{Fully computable a posteriori error estimator using anisotropic flux equilibration\\ on
anisotropic meshes%
\thanks{{The author was partially supported by  Science Foundation Ireland grant SFI/12/IA/1683}
}}

\author{Natalia Kopteva\thanks{Department of Mathematics and
        Statistics, University of Limerick, Limerick, Ireland
        ({\tt natalia.kopteva@ul.ie}).%
 }}

\date{}
\begin{document}

\maketitle

\begin{abstract}
Fully computable  a posteriori error estimates in the energy norm are given
for singularly perturbed semilinear reaction-diffusion equations
posed in polygonal domains.
Linear finite elements are considered
on anisotropic triangulations.
To deal with the latter, we employ anisotropic quadrature and explicit anisotropic flux reconstruction.
Prior to the flux equilibration, divergence-free corrections are introduced for pairs of anisotropic triangles sharing a short edge.
We also give an upper bound for the resulting estimator, in which
the error constants are
independent  of the diameters and the aspect ratios of mesh elements,
and of the small perturbation parameter.
\end{abstract}




\pagestyle{myheadings} \thispagestyle{plain}



\section{Introduction}
We consider linear finite element approximations to
singularly perturbed semilinear reaction-diffusion equations of the form
\begin{align}
Lu:= -\epsilon^2 \triangle u + f(x,y;u) = 0\quad\mbox{for}\;\;(x,y)\in \Omega, \qquad u=0\quad \mbox{on}\;\; \partial \Omega,
\label{eq1-1}
\end{align}
posed in a,
possibly non-Lipschitz,
polygonal domain $\Omega\subset\mathbb{R}^2$.
Here $0<\epsilon \le 1$. We also assume that
$f$ is continuous {on $\Omega \times \mathbb{R}$ and satisfies $f(\cdot; s) \in L_\infty(\Omega)$ for all $s \in \mathbb{R}$},
and
the one-sided Lipschitz condition $f(x,y; u)-f(x,y;v) \ge C_f [u-v]$ whenever $u\ge v$,
with some constant $C_f \ge 0$.
Then there is a unique 
$u\in W_\ell^2(\Omega)\subseteq W_q^1{\color{blue}(\Omega)}\subset C(\bar\Omega)$ for some  $\ell>1$ and $q>2$
\cite[Lemma~1]{DK14}.
We additionally assume that $C_f+\varepsilon^2\ge 1$
(as a division by $C_f+\varepsilon^2$ immediately reduces \eqref{eq1-1} to this case).

Our goal is to give explicitly and fully computable a posteriori error estimates
on reasonably general anisotropic meshes (such as on Fig.\,\ref{fig_gen_mesh} and Fig.\,\ref{fig_node_types})
in the energy norm
defined by
\vspace{-0.2cm}
$$
\vvvert v \vvvert_{\eps\,;{\Omega}}:=
\Bigl\{\eps^2\|\nabla v \|^2_{{\Omega}}
+C_f\|v \|^2_{{\Omega}} \Bigr\}^{1/2},
$$
where $\color{blue}\|\cdot\|_{\mathcal D}:=\|\cdot\|_{L_2({\mathcal D})}\;\forall {\mathcal D}\subseteq\Omega$.
This goal is achieved by
a certain combination of explicit flux reconstruction and  flux equilibration.

Flux equilibration for equations of type \eqref{eq1-1} was considered in \cite{AinsBab_1999,AinsVej_NM2011,AinsVej_2014,Ched_Vohr_2009}
on shape-regular meshes
(see also
\cite[Chap.\,6]{AinsOd_2000} for the case $\eps=1$),
and in \cite{Grosman_2006} on anisotropic meshes.
The estimators in \cite{AinsVej_NM2011,AinsVej_2014,Ched_Vohr_2009} are based on flux reconstructions,
while \cite{AinsBab_1999,Grosman_2006} employ solutions of certain local problems.

Our approach in this paper differs from the previous work in a few ways.

\begin{itemize}

\item The fluxes are equilibrated within a local patch using anisotropic weights depending on the local, possibly anisotropic, mesh geometry
(see \eqref{beta_choice_case1}).

\item	Prior to the flux equilibration, divergence-free corrections are introduced for pairs of anisotropic triangles sharing a short edge
(see \S\ref{sec_A1_star}, in particular, \eqref{btauS}).

\item	A certain anisotropic quadrature is used on anisotropic elements (see \S\ref{sec_Q}).
This is motivated by some observations made in \cite{Kopt_mc14}, and
also enables us to drop some mesh assumptions  made in
recent papers \cite{Kopt15,Kopt_NM_17}.

\item	Our estimator is explicitly and fully computable in the sense that it involves no unknown error constants
(unlike other estimators on anisotropic meshes, such as in \cite{Grosman_2006,Kopt15,Kopt_NM_17}).

\item In contrast to \cite{Grosman_2006}, an upper bound  for our estimator 
involves no matching functions (which we discuss below).
In fact, the error constant $C$ in the upper bound \eqref{upper_bound} is
independent  not only of the diameters and the aspect ratios of mesh elements,
but also of the small perturbation parameter $\eps$.

\item\color{blue}
Unlike \cite{AinsBab_1999,AinsVej_NM2011,AinsVej_2014,Ched_Vohr_2009}, and also \cite{Ver98c,Kunert2000,KunVer00,Kun01},
we consider the semilinear case, which mostly simplifies the presentation
(as $f$ may include a few linear terms).

\item
\color{blue}
By contrast, dealing with anisotropic elements requires some non-incemental changes in the flux construction and also a more intricate analysis compared to the isotropic-mesh case.

\item
\color{blue}
The efficiency of error estimators on anisotropic meshes was addressed in \cite{Kunert2000,KunVer00,Kun01} using the standard bubble function approach.
However,  a numerical example will be given in \S\ref{sec_kunert} that clearly demonstrates that  short-edge jump residual terms in such bounds are not sharp.
So, under additional restrictions on the anisotropic mesh, we shall give a new bound for the short-edge jump residual terms, and thus show
that at least for some  anisotropic meshes the error estimator constructed in the paper is efficient.

\end{itemize}

The robustness of our estimator, denoted by $\mathcal E$, with respect to the mesh aspect ratios, as well as the small perturbation parameter $\eps$, is demonstrated  by the following upper bound
 (which follows from  Theorems~\ref{theo_flux} and~\ref{theo_main_bounds}):
 \begin{align}
\notag
{\mathcal E}\le
C\,\Bigl\{\sum_{z\in\mathcal N}
{\color{blue}\min\{1,\,\eps h_z^{-1}\}} \,
\bigl\|
\eps J_z\bigr\|^2_{\omega_z}&{}+\sum_{z\in\mathcal N}
\bigl\| \min\{1,\,h_z\eps^{-1}\}\,f_h^I\bigr\|^2_{\omega_z}+\bigl\|f_h-f_h^I\bigr\|^2_{\Omega}
\Bigr.
 \\\label{upper_bound}
 \Bigr.
 +\sum_{T\in\mathcal T} \bigl\|\lambda_T\,{\rm osc}(f_h^I;T)\bigr\|^2_{T}
&{}+\!\!\sum_{z\in{\mathcal N}^*_{\pt\Omega}}\!\bigl\|\lambda_T f_h(z)\bigr\|^2_{\omega_z}\Big\}^{1/2},
\end{align}
where $C$ is 
{independent   of the diameters and the aspect ratios} of elements in the triangulation $\mathcal T$,
and of $\eps$.
Here
$\mathcal N$ is the set of nodes 
in $\mathcal T$,
and $\omega_z$ is the patch of elements surrounding any $z\in\mathcal N$, while
$J_z$ is the maximum within $\omega_z$ of the standard jump in the normal derivative of the computed solution $u_h$ across an element edge,
$f_h=f(\cdot;u_h)$ and $f_h^I$ is its standard piecewise-linear Lagrange interpolant.
We also use $\lambda_T=\min\{1,\,H_T\eps^{-1}\}$,
$H_T\sim {\rm diam}(T)$, $h_T\sim H_T^{-1}|T|$,
and $h_z\sim |\omega_z| / {\rm diam}(\omega_z)$
{\color{blue}(and some notation defined in the final paragraph of this section).}
The boundary subset ${\mathcal N}^*_{\pt\Omega}$ of $\mathcal N$ is defined in \eqref{N_boundary}.

To relate \eqref{upper_bound}  to interpolation error bounds (as well as to possible adaptive-mesh construction strategies), note that
$|J_z|$ may be interpreted as approximating the diameter of $\omega_z$ under the  metric
induced by the squared Hessian matrix of the exact solution
(while $f_h^I$ approximates $\eps^2\triangle u$).
Note also that the right-hand side in \eqref{upper_bound} is similar to the estimator
in the recent paper \cite{Kopt_NM_17}, and reduces, in the case of shape-regular meshes, to
a version of the estimator given by
\cite{Ver98c}.

Explicit residual-type a posteriori error estimates for problems of type \eqref{eq1-1}
were also given in \cite{Ver98c,DK14} on shape-regular meshes,
 \cite{Siebert96,Kopt08,ChKopt} on anisotropic tensor-product meshes,
and  \cite{Kunert2000,KunVer00,Kun01,Kopt15,Kopt_NM_17,Kopt_bail16} on more general anisotropic meshes
(for $\eps=1$ in \cite{Siebert96,Kunert2000}).
All these estimates are not fully guaranteed in the sense that they involve unknown error constants.
(The cited papers deal with the energy norm,
except for \cite{ChKopt,DK14,Kopt08,Kopt15} addressing the maximum norm.)

Note that the error constants in the  estimators of \cite{Kunert2000,Kun01,KunVer00} (as well as the upper bound for the estimator \cite{Grosman_2006}
that we already mentioned) involve the so-called matching functions. The latter depend on the unknown error
and take moderate values only when the grid is either isotropic, or, being anisotropic, is aligned correctly to the solution,
while, in general, they may be as large as mesh aspect ratios.
The presence of such matching functions in the estimator is clearly undesirable, and is entirely avoided in recent papers \cite{Kopt15,Kopt_NM_17,Kopt_bail16},
as well as in our upper bound \eqref{upper_bound}.

Finally, note that a posteriori error estimation
on anisotropic meshes presents a more serious challenge not only compared to the shape-regular-mesh case, but also to the a priori error estimation.
Indeed, there is a vast number of papers showing that \mbox{a-priori}-chosen anisotropic meshes
offer an efficient way of computing
reliable numerical approximations 
of solutions that exhibit sharp boundary and interior layers.
In the context of singularly perturbed differential equations, such as \eqref{eq1-1} with $\eps\ll 1$,
see, for example,
\cite{ClGrEOR_mc05,Kopt_mc07,Kopt_EOR,RStTob} and references therein.

The paper is organized as follows.
In \S\ref{sec_nodes}, we list all triangulation assumptions. 
Next, \S\ref{sec_Q} describes the considered finite element discretization with quadrature.
The structure of the reconstructed flux and the main results are presented in \S\ref{sec_main_res}.
The case $h_z\lesssim \eps$ is addressed in \S\S\ref{sec_A1}--\ref{sec_A1_star}, while \S\ref{sec_A2} deals with the case
$h_z\gtrsim \eps$.
{\color{blue}The efficiency of the constructed estimator is illustrated by some numerical results in \S\ref{ssec_numer}.
We conclude the paper by discussing lower error bounds on anisotropic meshes in \S\ref{sec_kunert}.}
\smallskip

{\it Notation.}
We write
 $a\sim b$ when $a \lesssim b$ and $a \gtrsim b$,
 $a = {\mathcal O}(b)$ when $|a|\lesssim b$,
 and
$a \lesssim b$ when $a \le Cb$ with a generic constant $C$ depending on $\Omega$ and
$f$,
but 
%
 $C$
 does not depend
 on either $\eps$ or
 the diameters and the aspect ratios of elements in~$\mathcal T$.
 Also, we write $a \ll b$ when $a<c_0b$
 with a fixed 
small constant $c_0$ (used to distinguish between anisotropic and isotropic elements).
The indicator function $\mathbbm{1}_{A}$ takes value $1$ if condition $A$ is satisfied, and vanishes otherwise.
For any $\mathcal{D}\subset\Omega $,
we let $\|\cdot\|_{\mathcal{D}}=\|\cdot\|_{L_2(\mathcal{D})}$,
and $\color{blue}{\rm osc}(v;\mathcal{D})=\sup_{\mathcal{D}}v-\inf_{\mathcal{D}}v\;\forall v\in L_\infty(\mathcal{D})$, while
$\bnu$ and $\bmu$, possibly subscripted, denote the unit vectors on $\pt \mathcal{D}$ in the outward normal and counterclockwise tangential direction, respectively.
For any triangles $T$ and $T'$ sharing an edge, a standard notation is used:
  $$
  \bigl[\btau\cdot \bnu\bigr]_{\pt T\cap\pt T'}:=\btau\cdot \bnu\bigr|_{T}+\btau\cdot \bnu\bigr|_{T'}\,,
  \qquad
  \bigl[\pt_\bnu u_h\bigr]_{\pt T\cap\pt T'}:=\bigl[\nabla u_h\cdot \bnu\bigr]_{\pt T\cap\pt T'}\,.
  $$

%
\section{Triangulation assumptions}\label{sec_nodes}
We shall use $z$, 
$S$ and $T$ to respectively denote particular mesh nodes, edges and {\color{blue}triangular} elements,
while $\mathcal N$, $\mathcal S$ and $\mathcal T$ will respectively denote their sets.
For each $T\in\mathcal T$,
let
$H_T$ be the maximum edge length 
and $h_T:= 2 H_T^{-1}|T|$ be the minimum {\color{blue}altitude} in $T$.
For each $z\in\mathcal N$, let
$\omega_z$ be the patch of elements surrounding any $z\in\mathcal N$,
${\mathcal S}_z$ the set of edges originating at $z$,
and
\beq\label{ring_gamma}
H_z:={\rm diam}(\omega_z),\quad h_z:=\max_{T\subset\omega_z}h_T,
\quad
\mathring{h}_z:=\min_{T\subset\omega_z}h_T,
\quad
\gamma_z:={\mathcal S}_z\setminus\pt\Omega.\vspace{-0.1cm}
\eeq
{\color{blue}(With slight abuse of notation, such as  in the latter formula, we occasionally treat subsets of $\mathcal N$, $\mathcal S$ and $\mathcal T$
as sets of points.)}


\begin{figure}[!b]
%
~\hfill
\includegraphics[width=0.5\textwidth]{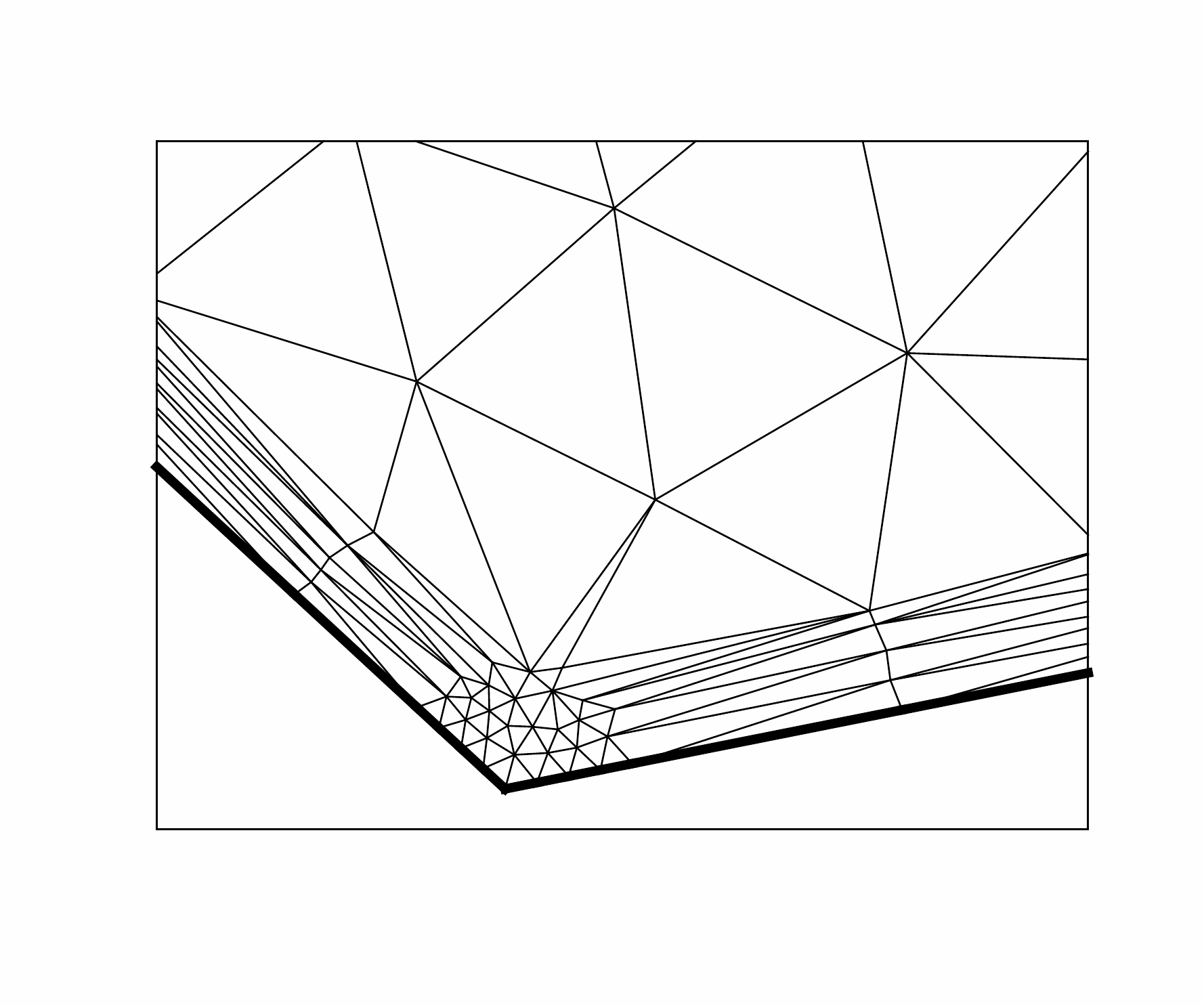}
\hfill~
\vspace{-0.9cm}
\caption{Example of a mesh  that satisfies
all assumptions made in \S\ref{sec_nodes}
(including ${\mathcal A}1$ and ${\mathcal A}2$).%
}
\label{fig_gen_mesh}
\end{figure}

Throughout the paper we make some {triangulation assumptions}.
All of them are automatically satisfied by
shape-regular triangulations.
\smallskip
\begin{itemize}

\item
{\it Maximum Angle condition.} Let the maximum interior angle in any triangle $T\in\mathcal T$
be uniformly bounded by some positive $\lambda_0<\pi$.
\smallskip

\item
Let the number of triangles containing any node be uniformly bounded.

\smallskip

\item
For any $z\in\mathcal N$, one has
\beq\label{mixed_node}
h_T\sim \mathring{h}_z\;\;\mbox{and}\;\;H_T\sim H_z\qquad\mbox{or}\qquad h_T\sim H_T \qquad\forall \ T\subset\omega_z.
\eeq
\end{itemize}

\begin{figure}[t!]
~\hfill
\begin{tikzpicture}[scale=0.22]
\draw[fill] (9.5,7) circle [radius=0.4];
\path[draw]  (8,5.1)--(18,5)--(19,6.7)--(18.5,8.5)--(9.1,9)--(1,7.7)--cycle;
\path[draw] (8,5.1)--(9.5,7)--(18,5);
\path[draw] (19,6.7)--(9.5,7)--(18.5,8.5);
\path[draw] (9.1,9)--(9.5,7)--(1,7.7);
%
\draw[fill] (4,2.4) circle [radius=0.4];
\path[draw]  (4,0.7)--(16,0)--(17,1.7)--(15.8,3.5)--(4,4.5)--(4,2.4)--cycle;
\path[draw] (4,2.4)--(16,0);
\path[draw] (17,1.7)--(4,2.4)--(15.8,3.5);
\path[draw,ultra thick
] (4,0)--(4,5);
\end{tikzpicture}
\hfill
\begin{tikzpicture}[scale=0.22]
\draw[fill] (9.5,7) circle [radius=0.4];
\path[draw]  (10,4.5)--(20,5)--(21,6.7)--(20.5,8.5)--(9.1,9)--(7,7.7)--(7.1,5)--cycle;
\path[draw] (10,4.5)--(9.5,7)--(20,5);
\path[draw] (21,6.7)--(9.5,7)--(20.5,8.5);
\path[draw] (9.1,9)--(9.5,7)--(7,7.7);
\path[draw] (7.1,5)--(9.5,7);
%
\draw[fill] (9.5,0) circle [radius=0.4];
\path[draw]  (7,0)--(7.5,2.4)--(10,2.9)--(20.5,1.8)--(21,0)--cycle;
\path[draw] (7.5,2.4)--(9.5,0)--(20.5,1.8);
\path[draw] (10,2.9)--(9.5,0);
\path[draw,ultra thick
] (6,0)--(22,0);
\end{tikzpicture}
\hfill
\begin{tikzpicture}[scale=0.21]
\draw[fill] (9.5,2.2) circle [radius=0.4];
\path[draw]  (10,0)--(20.2,0.5)--(21,2.7)--(20.5,4.5)--(9.1,10)--(7.1,3.5)--(7.3,0.7)--cycle;
\path[draw] (10,0)--(9.5,2.2)--(20.2,0.5);
\path[draw] (21,2.7)--(9.5,2.2)--(20.5,4.5);
\path[draw] (9.1,10)--(9.5,2.2)--(7.1,3.5);
\path[draw] (7.3,0.7)--(9.5,2.2);
\end{tikzpicture}
\hfill~
\caption{Examples of anisotropic nodes $z\in{\mathcal N}_{\rm ani}$ (left),
 nodes $z\in{\mathcal N}\backslash({\mathcal N}_{\rm ani}\cup {\mathcal N}_{\rm iso})$ (centre), an isotropic node $z\in{\mathcal N}_{\rm iso}$ (right),
and a node
$z\in{\mathcal N}_{\pt\Omega}^*$ (bottom left).
Examples of nodes that satisfy ${\mathcal A}1_{\rm ani}^*$ (top left), ${\mathcal A}1_{\rm ani}$ (bottom left), and ${\mathcal A}1_{\rm mix}$ (centre and right).
}
\label{fig_node_types}
\end{figure}
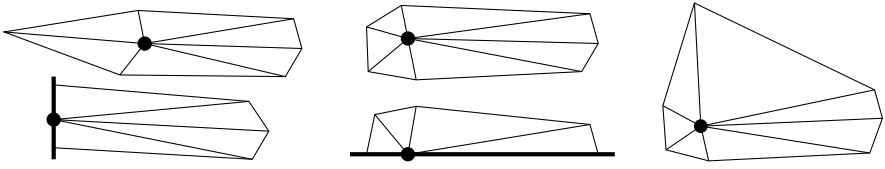

We also distinguish 
subsets ${\mathcal N}_{\rm ani}$, 
${\mathcal N}_{\rm iso}$
and ${\mathcal N}^*_{\pt\Omega}$
of ${\mathcal N}$ (see Fig.\,\ref{fig_node_types}).
Note that ${\mathcal N}_{\rm ani}\cap {\mathcal N}_{\rm iso}=\emptyset$,
while ${\mathcal N}\backslash({\mathcal N}_{\rm ani}\cup {\mathcal N}_{\rm iso})$ is not necessarily empty.

\smallskip

(1) {\it Anisotropic nodes}, whose set is denoted by ${\mathcal N}_{\rm ani}$, are such that
\beq\label{ani_node}
h_z\ll H_z,
\qquad\quad
h_T\sim h_z\;\;\mbox{and}\;\;H_T\sim H_z \quad\;\forall \ T\subset\omega_z.
\eeq
\noindent
Note that \eqref{ani_node} implies \eqref{mixed_node},
while $z\in{\mathcal N}_{\rm ani}$ implies $\mathring{h}_z\sim h_z$.
\smallskip

(2) {\it Isotropic nodes}, to whose set we shall refer as ${\mathcal N}_{\rm iso}$, are such that $h_z\sim H_z$.
%
\smallskip

(3) One may expect anisotropic elements near the boundary to be aligned along it. To distinguish some boundary nodes for which it is not the case,
we introduce
\beq\label{N_boundary}
{\mathcal N}^*_{\pt\Omega}:=\bigl\{z\in{\mathcal N}_{\rm ani}\cap\pt\Omega\backslash\{\mbox{corners of~}\Omega\}\mbox{~and~}|{\mathcal S}_z\cap\pt\Omega|\sim h_z\lesssim \eps\bigr\}.
\eeq

Occasionally, we shall make additional assumptions that we describe below.
\medskip

\begin{enumerate}

\setlength\labelwidth{5cm}

\item[$\mathcal A1\,$~~~]
Each $z\in\mathcal N$ with $h_z\lesssim \eps$
 satisfies $z\in {\mathcal N}_{\rm ani}\backslash\{\mbox{corners of~}\Omega\}$ and condition~${\mathcal A}1_{\rm ani}$, or it satisfies  condition ${\mathcal A}1_{\rm mix}$; see below.
 \medskip

\item[${\mathcal A}1_{\rm ani}$] {\it Quasi-non-obtuse anisotropic elements.}
Let the maximum triangle angle at 
$z\in {\mathcal N}_{\rm ani}$
be bounded by $\frac\pi2+\lambda_1 \frac{h_z}{H_z}$
for some positive constant~$\lambda_1$.
\medskip

\item[${\mathcal A}1_{\rm mix}\!$] With  $\mathring{\mathcal S}_z:=\{S\subset{\mathcal S}_z: |S|\sim \mathring{h}_z\}$ and
 $\mathring{\omega}_z:=\{T\subset \omega_z: h_T\sim H_T\sim \mathring{h}_z\}$ respectively denoting the sets of edges and isotropic triangles of diameter${}\sim\mathring{h}_z$
 within $\omega_z$,
 let
 $(\mathring{\omega}_z\cup\mathring{\mathcal S}_z)\backslash\{z\}$
 be connected.
 %
\medskip

\item[${\mathcal A}2\,$~~~]
Each $z\in\mathcal N$ with $h_z\gtrsim \eps$
 satisfies $\mathring{h}_z\ge \sqrt{6}\eps$.
\medskip
\end{enumerate}

\noindent
Note that ${\mathcal A}1_{\rm ani}$ is always satisfied by isotropic elements, so it requires only some of the anisotropic part of the mesh to be close to a non-obtuse triangulation.
${\mathcal A}1_{\rm mix}$ is also always satisfied on shape-regular meshes (as then $\mathring{\omega}_z=\omega_z$).
For anisotropic nodes, ${\mathcal A}1_{\rm mix}$ may be satisfied if $z\in\pt\Omega$
(in this case,
$\mathring{\omega}_z=\emptyset$, while
$\mathring{\mathcal S}_z\backslash\{z\}$ is connected only if $\mathring{\mathcal S}$ contains a single edge).
Note also that  $\mathcal A2$ is satisfied for any $z\not\in {\mathcal N}_{\rm iso}$, while for isotropic nodes
 it does impose a mild restriction
(as for the latter, $h_z\sim H_z$, so whenever $H_z\gtrsim \eps$, within $\omega_z$ we impose $h_T\ge \sqrt{6}\eps$).

\medskip

We shall also consider a weaker version of ${\mathcal A}1$.
\smallskip

\begin{enumerate}
\item[$\mathcal A1^*\!$~~]
Each $z\in\mathcal N$ with $h_z\lesssim \eps$
 satisfies $z\in {\mathcal N}_{\rm ani}\backslash\pt\Omega$ and condition~${\mathcal A}1_{\rm ani}^*$,
 or $z\in {\mathcal N}^*_{\pt\Omega}$ and satisfies~${\mathcal A}1_{\rm ani}$,
 or it satisfies  condition ${\mathcal A}1_{\rm mix}$.
 \smallskip

\item[${\mathcal A}1_{\rm ani}^*$]
{\it Local Element Orientation condition.}
For $z\in\mathcal N_{\rm ani}$, there exists a rectangle $R_z\supset \omega_z$ such that $|R_z|\sim |\omega_z|$.
\smallskip
\end{enumerate}


\section{Finite element method with quadrature}\label{sec_Q}

We discretize \eqref{eq1-1} using  linear finite elements.
Let $S_h \subset H_0^1(\Omega)\cap C(\bar\Omega)$ be
a piecewise-linear finite element space  relative to a triangulation $\mathcal T$, and let
the computed solution $u_h \in S_h$ satisfy
\begin{align}
\eps^2 \langle\nabla u_h,\nabla v_h\rangle+
\langle f_h,v_h\rangle_h
=0\quad
\forall\;v_h \in S_h,
\qquad\quad
f_h(\cdot):=f(\cdot;u_h).
\label{eq1-2}
\end{align}
Here $\langle\cdot,\cdot\rangle$ is the $L_2(\Omega)$ inner product,
and $\langle\cdot,\cdot\rangle_h$ is its quadrature approximation.

We now describe
$\langle f_h,v_h\rangle_h$
used in 
\eqref{eq1-2}.
For the integral over $T\in \mathcal T$, a quadrature formula $Q_T$ is employed, which
is anisotropic on a certain subset ${\mathcal T}^*$ of anisotropic elements:
%
\beq\label{Q_T}
Q_T(w)=|T|\sum_{j=1}^3 \theta_{T;z_j} w(z_j):=
\left\{\!
\begin{array}{cl}
\frac13|T|\bigl(w(z_1)+w(z_2)+w(z_3)\bigr)&\mbox{for~}T\in{\mathcal T}\backslash {\mathcal T}^*,\\[0.2cm]
\frac12|T|\bigl(w(z_1)+w(z_2)\bigr)&\mbox{for~}T\in {\mathcal T}^*.
\end{array}
\right.
\eeq
Here $\{z_j\}_{j=1}^3$ are the vertices of $T$, with $z_3=:z^*$ opposite the shortest edge, while
\beq\label{Tcal_star}
{\mathcal T}^*:=\{T\in {\mathcal T} : h_T\ll H_T \mbox{~and~}h_{T}\lesssim \eps\}
{}\,\backslash\, {\mathcal T}_0\,,
\eeq
with ${\mathcal T}_0\subset\{T\in {\mathcal T} : z^*\in{\mathcal N}_{\rm iso}\mbox{~and~}z_1,z_2\not\in {\mathcal N}_{\rm ani}\backslash\pt\Omega\}$
(so, unless one wants to minimize ${\mathcal T}^*$, the simplest option is ${\mathcal T}_0:=\emptyset$).
%
Now, let
\beq\label{quadr}
\langle f_h,v_h\rangle_h:=\sum_{T\in{\mathcal T}\backslash {\mathcal T}^*}Q_T( f_h v_h)
+\sum_{T\in{\mathcal T}^*}Q_T(\widebar f_h v_h),
\eeq
where
\beq\label{f_bar}
\widebar f_h=\bar f_{h; T}:={\textstyle\frac13}[f_h(z_1)+f_h(z_2)+f_h(z_3)]\quad\forall\ T\in{\mathcal T}.
\eeq
As $\widebar f_h$ is an elementwise constant approximation of $f_h$, so $Q_T(\bar f_h v_h)=\bar f_{h,T} Q_T( v_h)$.

Note that the discretization \eqref{eq1-2},\,\eqref{Q_T},\,\eqref{Tcal_star},\,\eqref{quadr} can be written as
\beq\label{FEM}
\sum_{S\subset\gamma_z}{\textstyle\frac12}\eps^2|S| \bigl[\partial_\bnu u_h\bigr]_{S}+\sum_{T\subset\omega_z}\theta_{T;z}|T|\,F_{T;z}=0\qquad
%
%
\forall\ z\in{\mathcal N}\backslash\pt\Omega,
\eeq
where $F_{T;z}:=f_h(z)$ for $T\in{\mathcal T}\backslash {\mathcal T}^*$ and $F_{T;z}:=\bar f_{h; T} $ for $T\in {\mathcal T}^*$.
It will be sometimes convenient to replace the second sum here using an average $\widebar F_z$ of $F_{T;z}$, associated with $z$, defined by
\beq\label{f_bar_z}
\widebar F_z\sum_{T\subset\omega_z}\!\theta_{T;z}|T|:=\!\sum_{T\subset\omega_z}\!\theta_{T;z}|T|\,F_{T;z}
\;\;\forall\ z\not\in\pt\Omega,
\quad
\widebar F_z
%
:=\color{blue}\left\{\!\!\begin{array}{cl}
0&\!\!\mbox{for}\; z\in{\mathcal N}^*_{\pt\Omega} : H_z\gtrsim \eps,\\
f_h(z)&\!\! \mbox{otherwise for}\;z\in\pt\Omega.
\end{array}\right.
\vspace{0.1cm}
\eeq

\begin{remark}
The above quadrature yields the standard linear lumped-mass finite element discretization
on ${\mathcal T}\backslash {\mathcal T}^*$.
On $ {\mathcal T}^*\!$,
a special anisotropic quadrature is employed (designed to address certain convergence issues reported in \cite{Kopt_mc14}).
The resulting method may be also interpreted as
the
vertex-centered finite volume method (or the box method) with a special choice of control volumes,
 applied to the approximation of our equation $-\eps^2\triangle u+ \widebar{ f(\cdot;u)}=0$.
A related interpretation as a
Petrov-Galerkin method  is also possible.%
\end{remark}%
\smallskip

\begin{remark}
Our results remain valid, if $Q_T( f_h v_h)$ is replaced by $Q_T(\widebar f_h v_h)$ in the first sum in \eqref{quadr}.
However, using
$Q_T( f_h v_h)$ for $h_T\gg \eps$ yields a superior discretization
(as in this case the local stiffness matrices become negligible so diagonal mass matrices are preferable).
On the other hand, replacing $Q_T(\widebar f_h v_h)$  by $Q_T(f_h v_h)$ in the second sum in \eqref{quadr} yields a less standard lumped-mass discretization on ${\mathcal T}^*$.
For the latter choice, our estimator will enjoy a version of the upper bound~\eqref{upper_bound}
with $\lambda_T$ replaced by $1$ whenever $h_T\ll H_T \lesssim\eps$.
Furthermore, all our results remain valid without any changes if
$Q_T(\widebar f_h v_h)$  is used  only for $T\in {\mathcal T}^* : H_T \lesssim\eps$.
\end{remark}
\smallskip

\section{A posteriori error estimator. Main Results}\label{sec_main_res}

We start with a relatively standard auxiliary result,  a version of which can be found, for example, in \cite[Lemma~1]{AinsVej_NM2011}
and \cite[Theorem~3.1]{Ched_Vohr_2009}.
\smallskip

\begin{theorem}\label{theo_flux}
For any $u_h\in S_h$,
let $\btau\in H^1({\rm div},T)\ \forall T\in\mathcal T$ also satisfy
\beq\label{tau_jump}
[\btau\cdot \bnu]=
[\partial_\bnu u_h]
\qquad\mbox{on all}\;\;S\in{\mathcal S}\backslash\pt\Omega.
\eeq
Then, for a solution $u$ of \eqref{eq1-1}, with $C_f>0$, one has
\beq\label{error_lemma}
\vvvert u_h-u \vvvert_{\eps\,;\Omega}\le
{\mathcal E}:=
\Bigl\{
\|\eps \btau\|_\Omega^2+C_f^{-1} \|\eps^2{\rm div}\btau+f_h\|_\Omega^2
\Bigr\}^{1/2}.
\eeq
\end{theorem}

\begin{proof}
With $v:=u_h-u$, one has
\begin{align*}
\vvvert u_h-u \vvvert^2_{\eps\,;\Omega}&\le
\eps^2 \langle\nabla (u_h-u),\nabla v\rangle+\langle f(\cdot;u_h)-f(\cdot;u),v\rangle\\
&=
\eps^2 \langle\nabla u_h,\nabla v\rangle+\langle f(\cdot;u_h),v\rangle\,,\\
&=
\eps^2 \langle\btau ,\nabla v\rangle+\langle \eps^2{\rm div}\btau+f_h,v\rangle,
\end{align*}
which immediately implies \eqref{error_lemma}.
Here we employed the observation (obtained using $\triangle u_h=0$ in any $T$, and \eqref{tau_jump}) that
$$
\langle\nabla u_h,\nabla v\rangle+\langle\underbrace{\triangle  u_h}_{{}=0}, v\rangle
=\sum_{S\in\mathcal S}\int_S v\underbrace{[\partial_\bnu u_h]}_{{}=[\btau\cdot \bnu]}
=\langle\btau ,\nabla v\rangle+\langle{\rm div}\btau , v\rangle.
$$
Note that here (as well as in \eqref{error_lemma}),  with slight abuse of notation, we understood
$\triangle  u_h$ and ${\rm div}\btau$ as regular functions in $\Omega$ defined elementwise.
\end{proof}
\medskip

\begin{remark}[Cases $C_f\ge0$ and $C_f=C_f(x,y)$]
An inspection of the above proof shows that the estimator
${\mathcal E}$ in
\eqref{error_lemma} can be replaced by a more general
$$
{\mathcal E}:=
\Bigl\{
(1-\vartheta)^{-1}\|\eps \btau\|_\Omega^2+\bigl(C_f+\eps^2C_{\Omega}^{-2}\vartheta\bigr)^{-1} \|\eps^2{\rm div}\btau+f_h\|_\Omega^2
\Bigr\}^{1/2},
$$
where $C_\Omega$ is the Poincar\'{e} constant, and $\vartheta\in[0,1)$ is an arbitrary constant, with $\vartheta>0$ unless $C_f>0$.
Note also that if $C_f=C_f(x,y)$, the above result remains valid with an obvious modification of the energy norm to
$\vvvert v \vvvert_{\eps\,;\Omega}:=
\bigl\{\eps^2\|\nabla v \|^2_\Omega
+\|C_f^{1/2}v \|^2_\Omega \bigr\}^{1/2}$
and a similar modification of ${\mathcal E}$.
\end{remark}

\newpage
\subsection{Structure of $\btau$}
Clearly, $\btau$ that satisfies the conditions of Theorem~\ref{theo_flux} is not unique, and there are various choice available in the literature.

Our task in this paper is to explicitly define $\btau$ to be used in \eqref{error_lemma} in a way that is appropriate for anisotropic meshes.
We introduce a suitable $\btau$ in the form
\beq\label{tau_def}
\btau:=\sum_{z\in\mathcal N}\btauz + \sum_{T\in \Tf}\btauT^f+
\sum_{S\in{\mathcal S}^*}\btauS^J, \qquad\quad
\Tf:=\{ T\in{\mathcal T}: H_T\lesssim \eps\},
\eeq
where $\btauz$, $\btauT^f$, and $\btauS^J$ have support on $\omega_z$, $T$, and
$T\cup T'$ for
$S=\pt T\cap\pt T'$, respectively.
It is also convenient to set $\btauT^f$ and $\btauS^J$ to $0$
whenever respectively $T\not\in \Tf$ and $S\not \in{\mathcal S}^*$.
Under condition ${\mathcal A}1$, we set ${\mathcal S}^*:=\emptyset$, while otherwise ${\mathcal S}^*$ is essentially the set of short edges
shared by pairs of anisotropic triangles
(see \S\ref{sec_A1_star} for details).

To be more precise, in the case $ {\mathcal S}^*=\emptyset$, the function
$\btauz$, with support on $\omega_z$, is simply required to satisfy
\beq\label{tau_n}
[\btauz\cdot \bnu]
=
\phi_z[\partial_\bnu u_h]\;\;\mbox{on}\;\;\gamma_z,
\qquad\quad
\btauz\cdot \bnu=0\;\;\mbox{on}\;\;\pt\omega_z\backslash\pt\Omega,
\eeq
{\color{blue}where $\{\phi_z\}_{z\in\mathcal N}$ are the standard basis
hat functions.}
The function $\btauT$, with support in $T$, satisfies
\beq\label{tau_f}
\btauT\cdot \bnu=0\quad\mbox{on}\;\pt T \quad\mbox{and}\quad
\eps^2{\rm div}\btauT+(f_h^I-\widebar f_h)=0\quad\mbox{in}\; T\qquad\forall\ T\in\Tf,
\eeq
and is explicitly defined (see, e.g., \cite[(22)]{AinsVej_NM2011}) by
\beq\label{tau_f_def}
\btauT^f:=
 b_1\phi_2\phi_3|S_1|\bmu_1
+b_2\phi_3\phi_1|S_2|\bmu_2
+b_3\phi_1\phi_2|S_3|\bmu_3,
\qquad
{\color{blue}b_j}:={\textstyle\frac13}\eps^{-2}\nabla f_h^I\cdot |S_j|\bmu_j\,,
\eeq
where $\{z_j\}_{j=1}^3$ are the vertices of $T$ with the corresponding basis functions $\phi_{j}:=\phi_{z_j}$, and for $j=1,2,3$,
the edge $S_j$ is opposite to $z_j$, while the counterclockwise tangential unit vector $\bmu_j$ lies along $S_j$;
see Fig.\,\ref{fig_tau_f}.
\medskip

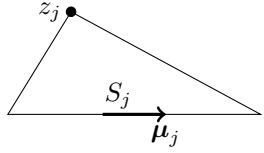
\begin{figure}[t!]
~\hfill
\begin{tikzpicture}[scale=0.21]
\draw[fill] (4,6.5) circle [radius=0.3];
\path[draw]  (4,6.5)node[left] {$z_j$}--(0,0)--(16,0)--cycle;
\node[above] at (7,-0.2) {$S_j$};
\draw[very thick, ->] (6,0) -- (10,0) node[below] {${\bmu}_j$};
%
\end{tikzpicture}
\hfill~\vspace{-0.4cm}
\caption{Notation used in \eqref{tau_f_def}:
the edge $S_j$ is opposite to the vertex $z_j$, the counterclockwise tangential unit vector $\bmu_j$ lies along $S_j$.
}
\label{fig_tau_f}
\end{figure}

\subsection{Upper bound for the estimator}
In this section, we present a theorem, which, combined with \eqref{error_lemma},
gives the upper bound \eqref{upper_bound} for the estimator $\mathcal E$.
At the same time, this theorem provides valuable information on the local properties
of the components of $\btau$ in \eqref{tau_def}.
These components (except for $\btauT^f$) are constructed and analyzed
in \S\S\ref{sec_A1}--\ref{sec_A1_star}
for the case $h_z\lesssim \eps$,  and in \S\ref{sec_A2}
for the case $h_z\gtrsim \eps$.
So, with the exception of  \eqref{tau_f_main}, all bounds in the following theorem
will be obtained in these forthcoming sections.
(To be more precise, here we summarize the results of
Lemmas~\ref{cor_case1}, \ref{lem_beta}, \ref{lem_beta_star} and~\ref{lem_theo_bound_case2}.)
\medskip

\begin{theorem}\label{theo_main_bounds}
Let $u_h$ solve \eqref{eq1-2} with
$\langle \cdot,\cdot\rangle_h$ defined in \S\ref{sec_Q}, and set
$$
J_z:=\max_{S\subset\gamma_z}\Bigl|[\pt_\bnu u_h]_S\Bigr|
\qquad\mbox{and}\qquad
\lambda_T:=\min\{1,\,H_T \eps^{-1}\}
\;\;
\;\;\forall T\in{\mathcal T}\!.
$$
(i) Under conditions ${\mathcal A}1$ and ${\mathcal A}2$,
one can construct $\btau$, subject to \eqref{tau_jump}, in the form \eqref{tau_def} with ${\mathcal S}^*=\emptyset$, where
$\btau_z$ and an associated function $g_z$, both with support in $\omega_z$,
satisfy, for any $z\in\mathcal N$,
\begin{align}\label{theo_div_bound}
%
\sum_{\textstyle{z\in{\mathcal N}:{}\above 0pt h_z\lesssim \eps}}
\eps^2{\rm div}\btauz+\widebar f_{h}={}&\sum_{z\in\mathcal N}g_z
\quad\mbox{in~}\Omega,
\\
\notag
\mathbbm{1}_{h_z \gtrsim \eps}\, \|\eps^2{\rm div}\btauz\|_{\omega_z}\!+\|\eps\btauz\|_{\omega_z}\!
+\|g_z\|_{\omega_z}\!\lesssim{}
&
{\color{blue}\min\{1,\eps h_z^{-1}\}^{1/2}}\|
\eps J_z\|_{\omega_z}\!+
 \min\{1,h_z\eps^{-1}\}\|f_h^I\|_{\omega_z}
 \\[0.2cm]
 +\!\!\sum_{T\subset\omega_z}\!\!\lambda_T
 &
 \|{\rm osc}(f_h^I;T)\|_{T}
{}+\mathbbm{1}_{z\in{\mathcal N}^*_{\pt\Omega}}\,\|\lambda_T f_h(z)\|_{\omega_z},
 \label{tau_z_main}
\end{align}
while $\btauT^f$ from \eqref{tau_f_def} satisfies \eqref{tau_f}, and, for any $T\in\mathcal T$,
\beq\label{tau_f_main}
\|\eps^2{\rm div}\btau^f_T+(f_h-\widebar f_h)\|_{T}+\|\eps\btau^f_T\|_{T}\lesssim \lambda_T\|{\rm osc}(f_h^I;T)\|_{T}+\|f_h-f_h^I\|_{T}.
\eeq
(ii) Under conditions ${\mathcal A}1^*$ and ${\mathcal A}2$,
one can construct $\btau$, subject to \eqref{tau_jump}, in the form \eqref{tau_def} with ${\mathcal S}^*\neq\emptyset$,
such that the above relations \eqref{theo_div_bound},\,\eqref{tau_z_main},\,\eqref{tau_f_main} hold true and, in addition,
for any edge $S=\pt T\cap\pt T'\in {\mathcal S}^*$
{\color{blue}with an endpoint $z$}, 
\beq\label{bound_theo_S}
{\rm div} \btauS^J=0\;\;\mbox{in~}T\cup T',\quad\;\;
\bigl\|\eps\btauS^J\bigr\|_{T\cup T'}\lesssim  \bigl\| \eps[\pt_\bnu u_h]_{S}\bigr\|_{T\cup T'}
\lesssim
{\color{blue}\min\{1,\,\eps h_z^{-1}\}^{1/2}\|\eps J_z\|_{\omega_z}\,.}
\eeq
\end{theorem}

{\it Proof of \eqref{tau_f_main}.}
If $T\in\Tf$, so, by \eqref{tau_def}, $ \lambda_T\sim H_T\eps^{-1}$,
a calculation \cite[\S3.3]{AinsVej_NM2011} shows that
\eqref{tau_f_def} implies \eqref{tau_f},
while
$|\gamma_j|\lesssim\eps^{-2} {\rm osc}(f_h^I;T)$ 
yields $|\eps \btauT^f|\lesssim H_T\eps^{-1}{\rm osc}(f_h^I;T)$.
The desired bound \eqref{tau_f_main} for $T\in\Tf$ follows.
Otherwise, i.e. if $T\not\in\Tf$, so $ \lambda_T\sim 1$, one has $\btauT^f=0$, while
$|f_h-\widebar f_h|\le |f_h-f_h^I|+ {\rm osc}(f_h^I;T)$, so we again get \eqref{tau_f_main}.
\endproof
\medskip

\begin{remark}
Note that for $z\in{\mathcal N}^*_{\pt\Omega}$, the bound \eqref{tau_z_main}
involves $\lambda_T |f_h(z)|\sim \min\{\eps^{-1},H_z^{-1}\}\,H_z|f_h(z)|$,
where $\eps^{-2} H_z|f_h(z)|$
may be interpreted as the diameter of $\omega_z$ under the  metric
induced by the squared Hessian matrix of the exact solution at $z\in\pt\Omega$.
Indeed, as $u=0$ on $\pt\Omega$, the Hessian matrix involves only the normal derivatives, while
$\eps^{-2} f_h=\eps^{-2} f(\cdot; 0)=\pt_\bnu^2 u$ on $\pt\Omega$;
see also the definition of ${\mathcal N}^*_{\pt\Omega}$.
\end{remark}
\smallskip

\section{Construction of  $\btau_z$ for $ h_z \lesssim \eps$ under condition ${\mathcal A}1$}\label{sec_A1}
Let  the patch $\omega_z$ be formed by $N_z$ triangles $\{T_i\}_{i=1}^{N_z}\subset\mathcal T$,
numbered counterclockwise so that $\gamma_z$ is formed by the edges
 $\pt T_{i-1}\cap \pt T_i$ for $i=1,\ldots,N_z$ if $z\not\in\pt\Omega$ (with the notation $T_0:=T_{N_z}$),
and for $i=2,\ldots, N_z$ if $z\in\pt\Omega$ (see Fig.\,\ref{fig_tau_z_case1} (left, centre)).
For each $T_i\subset\omega_z$, let $z$ be opposite to the edge denoted $S_i$, with the outward normal and the counterclockwise tangential unit vectors
denoted $\bnu_i$ and $\bmu_i$ (see Fig.\,\ref{fig_tau_z_case1} (right)).

Define $\btauz$ associated with $z$ by 
\begin{subequations}\label{tauz_h_small}
\beq\label{tau_z}
\btauz:=\phi_z\bigl(\alpha_i \bnu_i+\beta_i d_i^{-1}\bmu_i\bigr)\quad\forall\ T_i\subset\omega_z,
\qquad d_i:=2|T_i|\,|S_i|^{-1},
\eeq
where, using $F_{T;z}$ and $\widebar F_z$ from \eqref{FEM},\,\eqref{f_bar_z},
\beq\label{F_z_T}
{\color{blue}\alpha_i:=\eps^{-2}d_i\theta_{T_i;z} \widetilde F_{T_i;z}, \qquad}
\widetilde  F_{T;z}:=\left\{\begin{array}{cl}
\widebar F_z& \mbox{if}\;\;H_z\gtrsim \eps\mbox{~and~} z\in{\mathcal N}_{\rm ani}
,\\
F_{T;z}& \mbox{otherwise}.
\end{array}\right.
\eeq
Here we require $\{\beta_i\}_{i=1}^{N_z}$ to satisfy
\beq\label{tau_n_hat}
\beta_{i-1}-\beta_{i}+ \alpha_{i-1}\bnu_{i-1}\cdot|S_{i-1}^+|\bnu^+_{i-1}
+ \alpha_{i}\bnu_{i}\cdot|S_i^-|\bnu^-_{i}
=|S_{i}^-| \bigl[\partial_\bnu u_h\bigr]_{\pt T_{i-1}\cap \pt T_i}
\eeq
for $i=1,\ldots,N_z$ if $z\not\in\pt\Omega$,
and for $i=2,\ldots, N_z$ if $z\in\pt\Omega$.
We use the notation $S_i^\pm:=\pt T_i\cap \pt T_{i\pm 1}$, as well as
$\bnu^\pm_i$ and $\bmu^\pm_i$ for the outward normal and the counterclockwise tangential unit vectors of the edge $S_i^\pm$ in triangle $T_i$
(see Fig.\,\ref{fig_tau_z_case1} (right)).
\end{subequations}
\medskip

\begin{figure}[h!]
~\hfill%
\begin{tikzpicture}[scale=0.21]
\draw[fill] (0,0) circle [radius=0.5];
\path[draw]  (0,0)--(0,-6)--(6,-5)--cycle; \node[above] at (2.3,-5.5) {$T_1$};
\node[left] at (-0.2,0.5) {$z$};
\path[draw]  (8,2)--(0,0)--(0,6.5)--cycle;\node[below] at (2.9,4.1) {$T_3$};
\path[draw]  (8,2)--(6,-5);\node[left] at (7,-1.5) {$T_2$};
\path[draw, line width=1.5pt]  (0,-7)--(0,7.5);
\end{tikzpicture}
\hfill
\begin{tikzpicture}[scale=0.21]
\draw[fill] (0,0) circle [radius=0.5];
\path[draw]  (0,0)--(-2,-6)--(6,-5)--cycle; \node[above] at (2,-5.5) {$T_1$};
\node[left] at (-0.2,0.5) {$z$};
\path[draw]  (8,2)--(0,0)--(-2,6.5)--cycle;\node[below] at (2.9,4.25) {$T_3$};
\path[draw]  (-2,6.5)--(-7,-1)--(0,0);\node[right] at (-4.6,2.3) {$T_4$};
\path[draw]  (-7,-1)--(-2,-6);\node[right] at (-4.9,-3.0) {$T_5$};
\path[draw]  (8,2)--(6,-5);\node[left] at (7,-1.5) {$T_2$};
%
\path[draw, line width=1.5pt]  (2,-7)--(2,-7);
\end{tikzpicture}
\hfill~\hfill
\begin{tikzpicture}[scale=0.21]
\draw[fill] (4,8) circle [radius=0.5];
\path[draw]  (4,8)--(0,0)--(16,0)--cycle; \node[left] at (3.7,8.3) {$z$};
\node[above] at (6.3,-0.25) {$S_i$};
\node[left] at (10,3.6) {$S_i^+$};
\node[right] at (1.8,4) {$S_i^-$};
\draw[very thick, ->] (6.5,0) -- (10,0) node[below] {${\bmu}_i$};
\draw[very thick, ->] (6.5,0) -- (6.5,-3.5) node[left] {${\bnu}_i$};
\draw[line width=1.5pt, ->] (10.41602515, 3.722649902) -- (7.503849117, 5.664100589) ;  \node[right]at (7.4, 6.7) {${\bmu}^+_i$} ;
\draw[line width=1.5pt, ->] (10.41602515, 3.722649902) -- (12.35747584, 6.634825932)  ;\node[below] at (13.3, 7) {${\bnu}^+_i$} ;
\draw[line width=1.5pt, ->] (2.223606798, 4.447213595) -- (0.658359213, 1.316718427) ;  \node[left]at (1.1, 1.3) {${\bmu}^-_i$} ;
\draw[line width=1.5pt, ->] (2.223606798, 4.447213595) -- (-0.906888370, 6.012461179)  ;\node[left] at (-0.0, 6.3) {${\bnu}^-_i$} ;
\node[right] at (18,3) {$T_i$};
\path[draw,help lines]  (18,3)
--(9,1.5);
\end{tikzpicture}\vspace{-0.1cm}
\caption{Notation in \eqref{tauz_h_small}:
$\omega_z$ is formed by $\{T_i\}_{i=1}^{N_z}$ with $N_z=3$ (left) and $N_z=5$ (centre);
the edge $S_i$ in $T_i$ is opposite to $z$,  with the outward normal and the counterclockwise tangential unit vectors
$\bnu_i$ and $\bmu_i$;
the other edges
$S_i^\pm=\pt T_i\cap \pt T_{i\pm 1}$ have the outward normal and the counterclockwise tangential unit vectors $\bnu^\pm_i$ and $\bmu^\pm_i$ (right).%
}
\label{fig_tau_z_case1}
\end{figure}
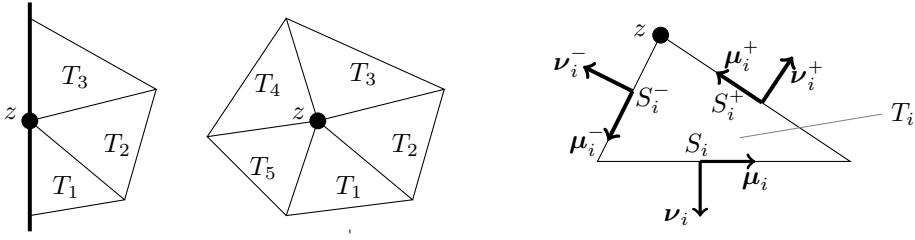

\newpage
\begin{lemma}\label{lem_tau_case1}
Let $ h_z \lesssim \eps$. Then
relations \eqref{tauz_h_small} for $\btauz$ imply \eqref{tau_n} and
\begin{subequations}\label{tauz_small_h_prop}
\begin{align}\label{tauz_small_h_prop_a}
\eps^2{\rm div}\btauz+\theta_{T;z}\widetilde F_{T;z}&=0\qquad \forall\ T\subset\omega_z,
\\\label{tauz_small_h_prop_b}
\|\eps \btauz\|_{\omega_z}&\lesssim
\|
h_z\eps^{-1}f^I_h
\|_{\omega_z}+\eps\,\Bigl\{\sum_{i=1}^{N_z} \beta^2_i d_i^{-2} |T_i|\Bigr\}^{1/2}.
\end{align}
\end{subequations}
The system \eqref{tau_n_hat} for $\{\beta_i\}_{i=1}^{N_z}$ is consistent and has infinitely many solutions.
\end{lemma}
\smallskip

\begin{proof}
Combining ${\rm div}(\phi_z \bmu_i)=\nabla\phi_z\cdot \bmu_i=0$  with ${\rm div}(\phi_z \bnu_i)=\nabla\phi_z\cdot \bnu_i=-d_i^{-1}$
one gets
${\rm div}\btauz+\alpha_i d_i^{-1}=0$ in $T_i\subset\omega_z$, which immediately implies \eqref{tauz_small_h_prop_a}.
For \eqref{tauz_small_h_prop_b}, note that
$d_i\theta_{T_i;z}\lesssim h_z\theta_{T_i;z}$, because, in view of \eqref{Q_T},\,\eqref{Tcal_star}, unless $d_i\lesssim h_z$, one has $T_i\in{\mathcal T}^*$ and $z=z^*_{T_i}$ so $\theta_{T_i;z}=0$.
Now, $|\eps\alpha_i|
\lesssim (h_z \eps^{-1})\, \theta_{T_i;z}|\widetilde F_{T_i;z}|$.
Combining this with \eqref{F_z_T} and \eqref{f_bar_z}, one gets \eqref{tauz_small_h_prop_b}

Next, note that \eqref{tau_n},
combined with \eqref{tau_z}, is equivalent to
$$
\bigl(\alpha_{i-1}\bnu_{i-1}+\beta_{i-1}d_{i-1}^{-1}\bmu_{i-1}\bigr)\cdot\bnu^+_{i-1}
+ \bigl(\alpha_{i}\bnu_{i}+\beta_{i}d_i^{-1}\bmu_{i}\bigr)\cdot\bnu^-_{i}
= \bigl[\partial_\bnu u_h\bigr]_{\pt T_{i-1}\cap \pt T_i}\,.
$$
Multiplying this  by $|S_{i-1}^+|=|S_i^-|$ and noting that
$d_i=\bmu_{i}\cdot|S_{i}^+|\bnu^+_{i}=-\bmu_{i}\cdot|S_i^-|\bnu^-_{i}$, one gets \eqref{tau_n_hat}.
So \eqref{tau_n_hat} is, indeed, equivalent to  \eqref{tau_n}.

Finally, consider the system \eqref{tau_n_hat} for
$\{\beta_i\}_{i=1}^{N_z}$.
{\color{blue}For this system to be consistent, it suffices to show that it is under-determined
(as then, taking any specific $\beta_1$, one can uniquely compute all other $\{\beta_i\}$).}
For $z\in\pt\Omega$,
there are $N_z-1$ equations, so
this system  is clearly under-determined.
For $z\not\in\pt\Omega$, this is also the case as an application of $\sum_{i=1}^{N_z}$ to \eqref{tau_n_hat} yields $0$.
To check the latter,
one first employs the observation that
$\bnu_i\cdot\bigl(|S_i^+|\bnu^+_i +|S_{i}^-|\bnu^-_{i}\bigr)+2|T_i|d_i^{-1}=0$, and then recalls \eqref{F_z_T}, as well as \eqref{FEM} and \eqref{f_bar_z}.
\end{proof}
\medskip

\begin{remark}[Anisotropic flux equilibration]%
\label{rem_beta}
The choice of a particular solution $\{\beta_i\}$  of \eqref{tau_n_hat} is crucial, as
 our estimator, roughly speaking, involves the component
  $\sum_{i=1}^{N_z} \beta^2_i d_i^{-2} |T_i|$ from \eqref{tauz_small_h_prop_b},
while, unless the mesh is shape-regular, $d_i^{-2} |T_i|=\frac14|S_i|^2|T_i|^{-1}$ may vary very significantly within $\omega_z$.
One simple and useful approach is to minimize this component, i.e.
given any particular solution $\{\hat\beta_i\}$ of \eqref{tau_n_hat}, let 
\beq\label{beta_choice_case1}
\beta_i:=\hat\beta_i-C_z,\qquad
\mbox{where}\quad
\sum_i(\hat\beta_i-C_z)d_i^{-2} |T_i|=0.
\vspace{-2pt}
\eeq
(Alternatively, one can set $\beta_i:=0$ for the element $T_i$ with the largest $d_i^{-2} |T_i|$ within~$\omega_z$,
or choose $\{\beta_i\}$ as in the proof of Lemma~\ref{lem_beta}.)
\end{remark}
\medskip

\begin{remark}[Computing $\btauz$ via optimization]\color{blue}
More generally, for $h_z\lesssim \eps$, one can construct
$\btauz$ using \eqref{tau_z} in which $\{\alpha_i\}$ and $\{\beta_i\}$
are chosen to minimize
\beq\label{optima}
\sum_{i=1}^{N_z} \Bigl({\textstyle\frac16}\eps^2\alpha_i^2+{\textstyle\frac16}\eps^2 \beta^2_i d_i^{-2} +[\theta_{T;z} F_{T;z}-\eps^2\alpha_i d_i^{-1}]^2\Bigr)|T_i|
\qquad\mbox{subject to \eqref{tau_n_hat}}.\vspace{-0.1cm}
\eeq
(Here we in fact minimize $\|\eps\btauz\|^2_{\omega_z}+\|\eps^2{\rm div}\btauz+\theta_{T;z} F_{T;z}\|^2_{\omega_z}$, while \eqref{F_z_T} is dropped).
Then 
Theorem~\ref{theo_main_bounds} remains valid
(in view of the results in \S\S\ref{subsec_51}--\ref{sec_sigma} and \S\ref{sec_A1_star}).
\end{remark}
\smallskip

\begin{remark}\color{blue}
It is assumed throughout this section that
any $z\in{\mathcal N}_{\rm ani}$ with $h_z\lesssim \eps$ also satisfies $\omega_z\subset {\mathcal T}^*$.
This is consistent with the definition \eqref{Tcal_star} of ${\mathcal T}^*$ as long as the somewhat imprecise
condition $h_T\lesssim\eps$ used in \eqref{Tcal_star} always follows from 
$h_z\lesssim\eps$.
\end{remark}
\smallskip

\subsection{Proof of \eqref{theo_div_bound} and \eqref{tau_z_main} in Theorem~\ref{theo_main_bounds}(i) for $h_z\lesssim \eps$}\label{subsec_51}
Our findings in this section are presented as two lemmas.
\smallskip

\begin{lemma}\label{cor_case1}
Let $\{\btauz\}$ satisfy \eqref{tau_z},\,\eqref{F_z_T} for any $ z\in\mathcal N$ with $h_z\lesssim \eps$.
Then for any $ z\in{\mathcal N}$, there is a function $g_z $ with support on $\omega_z$ that satisfies \eqref{theo_div_bound} and \eqref{tau_z_main}.%
\end{lemma}%
\smallskip

\begin{proof}
In any $T\subset\omega_z$, let $g_z:=\eps^2{\rm div}\btauz+\theta_{T;z} F_{T;z}$ if $h_z\lesssim \eps$, and $g_z:=\theta_{T;z} F_{T;z}$ otherwise.
In view of \eqref{Q_T},\,\eqref{FEM}, $\sum_{{\mathcal N}\cap z\in\pt T } \theta_{T;z} F_{T;z}=\widebar f_{h;T}=\widebar f_h$ in any $T\in\mathcal T$, so the first assertion \eqref{theo_div_bound} follows.
Next, if $h_z \gtrsim \eps$, i.e. $\min\{1,\,h_z\eps^{-1}\}\sim 1$, one immediately gets $\|g_z\|_{\omega_z}\lesssim \min\{1,\,h_z\eps^{-1}\}\|f_h^I\|_{\omega_z}$.
Otherwise, i.e. if $h_z\lesssim \eps$,
{\color{blue}a version of}
\eqref{tauz_small_h_prop_a} implies
{\color{blue} $g_z=\theta_{T;z} F_{T;z}-\eps^2\alpha_i d_i^{-1}=\theta_{T;z} (F_{T;z}-\widetilde F_{T;z})$, and so}
$|g_z|\le |F_{T;z}-\widetilde F_{T;z}|$
for any $T\subset\omega_z$.
By \eqref{F_z_T}, unless $g_z=0$, one has $|g_z|\le |F_{T;z}-\widebar F_z|\le {\rm osc}(f_h^I;\omega_z)+\color{blue}\mathbbm{1}_{z\in{\mathcal N}^*_{\pt\Omega}}f_h(z)$, where we also used \eqref{f_bar_z}.
At the same time,
$H_z\gtrsim \eps$ implies $1\sim \min\{1,\,H_z \eps^{-1}\}\sim \lambda_T$ (as $z\in{\mathcal N}_{\rm ani}$ so $H_z\sim H_T$).
Combining these observations with $|T|\sim|\omega_z|$ for any $T\subset\omega_z$ yields
$\|g_z\|_{\omega_z}\lesssim {}$
{\color{blue}the second line in \eqref{tau_z_main}.}
\end{proof}
\medskip

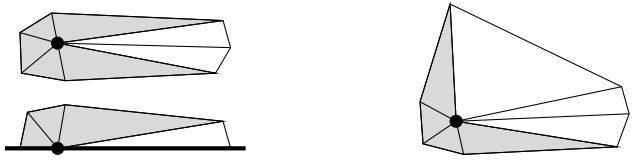
\begin{figure}[t!]
~\hfill
\begin{tikzpicture}[scale=0.20]
\path[draw, fill=gray!30!white] (20,5)--(10,4.5)--(7.1,5)--(7,7.7)--(9.1,9)--(20.5,8.5)--(9.5,7)--cycle;
\draw[fill] (9.5,7) circle [radius=0.4];
\path[draw]  (10,4.5)--(20,5)--(21,6.7)--(20.5,8.5)--(9.1,9)--(7,7.7)--(7.1,5)--cycle;
\path[draw] (10,4.5)--(9.5,7)--(20,5);
\path[draw] (21,6.7)--(9.5,7)--(20.5,8.5);
\path[draw] (9.1,9)--(9.5,7)--(7,7.7);
\path[draw] (7.1,5)--(9.5,7);
%
\path[draw, fill=gray!30!white] (7,0)--(7.5,2.4)--(10,2.9)--(20.5,1.8)--(9.5,0)--cycle;
\draw[fill] (9.5,0) circle [radius=0.4];
\path[draw]  (7,0)--(7.5,2.4)--(10,2.9)--(20.5,1.8)--(21,0)--cycle;
\path[draw] (7.5,2.4)--(9.5,0)--(20.5,1.8);
\path[draw] (10,2.9)--(9.5,0);
\path[draw,ultra thick
] (6,0)--(22,0);
\end{tikzpicture}
\hfill
\begin{tikzpicture}[scale=0.20]
\path[draw, fill=gray!30!white] (10,0)--(20.2,0.5)--(9.5,2.2)--(9.1,10)--(7.1,3.5)--(7.3,0.7)--cycle;
\draw[fill] (9.5,2.2) circle [radius=0.4];
\path[draw]  (10,0)--(20.2,0.5)--(21,2.7)--(20.5,4.5)--(9.1,10)--(7.1,3.5)--(7.3,0.7)--cycle;
\path[draw] (10,0)--(9.5,2.2)--(20.2,0.5);
\path[draw] (21,2.7)--(9.5,2.2)--(20.5,4.5);
\path[draw] (9.1,10)--(9.5,2.2)--(7.1,3.5);
\path[draw] (7.3,0.7)--(9.5,2.2);
\end{tikzpicture}
\hfill~
\caption{Examples of nodes that satisfy ${\mathcal A}1_{\rm mix}$
with the set $\{T_i\}_{i=1}^n$ (used in the proof of Lemma~\ref{lem_beta}) highlighted by the grey color:
$n=5$ (top left), $n=3$ (bottom left), $n=4$ (right).
}
\label{fig_node_typesA_mixed}
\end{figure}

\begin{lemma}\label{lem_beta}
Under condition ${\mathcal A}1$, for any $z\in\mathcal N$ with $h_z\lesssim \eps$, there is a solution $\{\beta_i\}_{i=1}^{N_z}$ of \eqref{tau_n_hat}
such that $\btauz$ defined by \eqref{tauz_h_small} satisfies \eqref{tau_z_main}.
\end{lemma}%
\smallskip%

\begin{proof}
Our task is to show that $\|\eps\btauz\|_{\omega_z}$ satisfies \eqref{tau_z_main}, in which the right-hand side involves
$\color{blue}\min\{1,\,\eps h_z^{-1}\}\sim 1$ and
$\min\{1,\,h_z\eps^{-1}\}\sim h_z\eps^{-1}$.
So it suffices to prove 
\beq
\|\eps\btauz\|_{\omega_z}\lesssim \|\eps J_z\|_{\omega_z}\!+h_z \eps^{-1}\|f^I_h\|_{\omega_z}+\!\!\sum_{T\subset\omega_z}\!\!\lambda_T\|{\rm osc}(f_h^I;T)\|_{T}
+\mathbbm{1}_{z\in{\mathcal N}^*_{\pt\Omega}}\,\|\lambda_T f_h(z)\|_{\omega_z}.
\label{eq_lem_beta}
\eeq

An inspection of the proof of Lemma~\ref{lem_tau_case1} reveals that if $\widetilde F_{T;z}=F_{T;z}$ in \eqref{F_z_T},~then
$$
\sum_{T_i\subset\omega_z}\|\eps\alpha_i\|^2_{\omega_z} \lesssim
|\omega_z|\,|h_z\eps^{-1} f_h(z)|^2 +|\omega_z|\!\sum
_{\textstyle{T_i\in{\mathcal T}^*:{}\above 0pt d_i\sim h_{T_i}}}
\!\!|d_i\eps^{-1}\widebar f_{h;T}|^2
\lesssim
\sum_{T\subset\omega_z}\|h_z\eps^{-1}f^I_h\|^2_{T},
$$
where
for the first relation we used \eqref{FEM} combined with \eqref{Q_T},\,\eqref{Tcal_star}, and for the second,
 $|\omega_z|\sim h_zH_z$ and $h_{T_i} H_z\sim |T_i|$ for any $T_i\in{\mathcal T}^*$.
One gets a similar conclusion for
the case $\widetilde F_{T;z}=\widebar F_{z}$ in \eqref{F_z_T} (in fact, the latter case is more straightforward as then $z\in{\mathcal N}_{\rm ani}$
so $|T_i|\sim |\omega_z|$ for any $T_i\subset\omega_z$).
Now, in view of \eqref{tauz_small_h_prop_b},
 to get the desired assertion
 {\eqref{eq_lem_beta}}, it suffices to show that
\beq\label{case1_beta_desired}
|\beta_i d_i^{-1}|\lesssim |J_z|+\max_{j=1,\ldots,N_z}|\alpha_j|+|\hat\sigma_z|,
\eeq
with some $\hat\sigma_z$ such that $\|\eps\hat\sigma_z\|_{\omega_z}$ satisfies a version of {\eqref{eq_lem_beta}}.

For \eqref{case1_beta_desired}, we start with a straightforward observation that follows from~\eqref{tau_n_hat}:
\beq\label{beta_aux}
\mbox{if~~}|S_i^-|\sim d_i \gtrsim d_{i-1}
\quad\Rightarrow\quad
|\beta_id_i^{-1}|\lesssim |\beta_{i-1}d_{i-1}^{-1}|+|\alpha_{i-1}|+|\alpha_i|+|J_z|.
\eeq
Consider three cases (a), (b) and (c).

(a) Suppose that $z$ satisfies ${\mathcal A}1_{\rm mix}$.
Then the $N_z$ triangles in $\omega_z$ can be numbered counterclockwise so that the set
$\{T_i\}_{i=1}^{n}\neq\emptyset$, with some $n=n_z\le N_z$,
is formed by all triangles having at least one edge in $\mathring{S}_z$ (see Fig.\,\ref{fig_node_typesA_mixed}).
To be more precise, this set will include
all triangles from $\mathring{\omega}_z$, and, possibly, one or two anisotropic triangles  that  either share an edge with $\mathring{\omega}_z\neq\emptyset$
or, if $\mathring{\omega}_z=\emptyset$ and so $\mathring{S}_z$ includes a single edge, touch this edge.
Note that then $d_i\sim \mathring{h}_z$ for $i=1,\ldots, n_z$ and $d_i\sim H_z$ for $i>n_z$, while $|S_i^-|\sim \mathring{h}_z$ for $i=2,\ldots, n_z$.
%
%
So setting $\beta_1:=0$ and applying \eqref{beta_aux} for $i>1$, we arrive at
\eqref{case1_beta_desired} with $\hat\sigma_z:=0$.
\smallskip

(b) Next, consider $z\in {\mathcal N}_{\rm ani}\backslash\pt\Omega$ that satisfies~${\mathcal A}1_{\rm ani}$ (and so not~${\mathcal A}1_{\rm mix}$).
Then $\mathring{\mathcal S}_z$ includes exactly two edges of length${}\sim {\mathring h}_z\sim h_z$.
Let the triangles $\{T_i\}_{i=1}^N$ forming the patch $\omega_z$
be numbered counterclockwise so that $\mathring{\mathcal S}_z=\{\pt T_{i-1}\cap \pt T_i\}_{i=1,m+1}$, for some $m=m_z\le N_z-2$
(see Fig.\,\ref{fig_local_kappa} (left)).

Note that $d_i\sim h_z$ only for $i=0,1,m,m+1$ and $d_i\sim H_z$ otherwise, while
$|S_i^-|\sim h_z$ for $i=1,m$ and $\sim H_z$ otherwise.
Hence,
one can employ \eqref{beta_aux} for $i\neq 0,m$. So it remains to get the desired bound \eqref{case1_beta_desired} only
 for $i= 0,m$.
For this, let
\beq\label{sigma_def}
\widetilde\sigma_z:=
\sum_{i=2}^{m} |S_{i}^-|\bigl[\partial_\bnu u_h\bigr]_{\pt T_{i-1}\cap \pt T_i}+2\eps^{-2}\sum_{i=1}^m\theta_{T_i;z}|T_i|\,\widetilde F_{T_i;z}
\eeq
(compare with \eqref{FEM}).
Now, an application of $\sum_{i=1}^m$ to \eqref{tau_n_hat} (and also noting that $\bnu_i\cdot\bigl(|S_i^+|\bnu^+_i +|S_{i}^-|\bnu^-_{i}\bigr)+2|T_i|d_i^{-1}=0$) yields
\beq\label{beta_0_m}
\beta_{0}-\beta_{m}
+\alpha_{0}\bnu_{0}\cdot|S_{0}^+|\bnu^+_{0}
-\alpha_{m}\bnu_{m}\cdot|S_{m}^+|\bnu^+_{m}
=
|S_{1}^-|\bigl[\partial_\bnu u_h\bigr]_{\pt T_{0}\cap \pt T_1}+\widetilde\sigma_z.
\eeq
So,
 for example, one can set $\beta_0:=0$ and compute and then estimate $\beta_m$ from \eqref{beta_0_m}. Or, one can choose
 $\beta_0$ and $\beta_m$, in agreement with \eqref{beta_0_m}, but in a more balanced way.
 Importantly, one can ensure for $i=0,m$ that $|\beta_i d_i^{-1}|\lesssim |\alpha_i|+|J_z|+ h_z^{-1}|\widetilde\sigma_z|$.
Consequently,
we get \eqref{case1_beta_desired} for all $i$ with $\hat\sigma_z:=h_z^{-1}\widetilde\sigma_z$.

Finally, similarly to \eqref{f_bar_z}, define a version of \eqref{sigma_def}:
\beq\label{sigma_def__}
\sigma_z:=
\sum_{i=2}^{m} |S_{i}^-|\bigl[\partial_\bnu u_h\bigr]_{\pt T_{i-1}\cap \pt T_i}+2\eps^{-2}\widebar F_{z}\sum_{i=1}^m\theta_{T_i;z}|T_i|.
\eeq
By \eqref{F_z_T}, unless $\widetilde\sigma_z=\sigma_z$,
one has
$\widetilde F_{T_i;z}\neq\widebar F_{z}$ and so $H_z\eps^{-1}\sim \min\{1,\,H_z\eps^{-1}\}\sim \lambda_T$ (the latter is also because $z\in{\mathcal N}_{\rm ani}$),
so $\eps h_z^{-1}|\widetilde\sigma_z-\sigma_z|\lesssim \sum_{T\subset\omega_z}\lambda_T\,{\rm osc}(f_h^I;T)$.
Combining this with a technical result
\eqref{bound_lem_jump_A1} (obtained below in \S\ref{sec_sigma}), one arrives at
\beq\label{hat_sigma}
|\eps h_z^{-1}\widetilde\sigma_z|=
|\eps\hat\sigma_z|\lesssim |\eps J_z|+h_z\eps^{-1}|\widebar F_z|
+\sum_{T\subset\omega_z}\lambda_T\,{\rm osc}(f_h^I;T).
\eeq
As $\|\widebar F_z\|_{\omega_z}\lesssim \|f_h^I\|_{\omega_z}$ (by \eqref{FEM},\,\eqref{f_bar_z}), so we have again obtained \eqref{case1_beta_desired} with
 $\|\eps\hat\sigma_z\|_{\omega_z}$ now satisfying a version of {\eqref{eq_lem_beta}}.
 This completes the proof of {\eqref{eq_lem_beta}} for this case.

\newpage
(c) It remains to consider $z\in{\mathcal N}^*_{\pt\Omega}$, which satisfies~${\mathcal A}1_{\rm ani}$ but not~${\mathcal A}1_{\rm mix}$.
This case is similar to case~(b), with
a version of \eqref{beta_0_m} becoming
$$
\beta_{1}-\beta_{m}-\alpha_{1}\bnu_{1}\cdot|S_1^-|\bnu^-_{1}
-\alpha_{m}\bnu_{m}\cdot|S_{m}^+|\bnu^+_{m}
=\widetilde\sigma_z.
$$
Again, using \eqref{bound_lem_jump_A1}, we get a version of \eqref{hat_sigma} with an additional term
$H_z\eps^{-1}|\widebar F_z|$
in the right-hand side.
As, by \eqref{f_bar_z},
{\color{blue}unless $\widebar F_z=0$, one has
$\widebar F_z=f_h(z)$
and $H_z\eps^{-1}\sim\min\{1,\,H_z\eps^{-1}\}\sim \lambda_T$
for $z\in{\mathcal N}^*_{\pt\Omega}$,
we again get {\eqref{eq_lem_beta}}.}
%
%
%
\end{proof}
\medskip

\subsection{Estimation of $\sigma_z$}\label{sec_sigma}
Here we give one technical result on $\sigma_z$. 
Throughout this section, we use the notation from the proof of Lemma~\ref{lem_beta}.
\smallskip

\begin{lemma}\label{lem_jump}
(i) If $z\in {\mathcal N}_{\rm ani}\backslash\pt\Omega$, with $h_z\lesssim \eps$,  satisfies~${\mathcal A}1_{\rm ani}^*$,
then for $\sigma_z$ of~\eqref{sigma_def__} one has
\beq\label{bound_lem_jump}
\eps h_z^{-1}\bigg|\sigma_z-
\!\!\!\sum_{i=1,m+1}
\bmu_{i-1}^+\cdot{\mathbf i}_\xi\,\,
\frac{H_{T_{i-1}}H_{T_i}}{H_{T_{i-1}}+H_{T_i}} \,
[\pt_\bnu u_h]_{\pt T_{i-1}\cap \pt T_i}
\bigg|\lesssim|\eps J_z|+h_z\eps^{-1}|\widebar F_{z}|,
\eeq
where $\mathbf{i}_\xi$ is the unit vector that points from $z$ in the direction of any edge from $\{S_i^-\}_{i=2}^m$.%
\smallskip

\noindent
(ii) If $z\in {\mathcal N}_{\rm ani}\backslash\{\mbox{corners of~}\Omega\}$, with $h_z\lesssim \eps$,  satisfies ${\mathcal A}1_{\rm ani}$ but not~${\mathcal A}1_{\rm mix}$,
then
\beq\label{bound_lem_jump_A1}
\eps h_z^{-1}|\sigma_z|\lesssim |\eps J_z|+h_z\eps^{-1}|\widebar F_z|
+\mathbbm{1}_{z\in{\mathcal N}^*_{\pt\Omega}}\,\,H_z\eps^{-1}|\widebar F_z|.
\eeq
\end{lemma}

\begin{proof}
(i)
For any scalar $w$, let 
$\llbracket w\rrbracket_{\pt T_{i-1}\cap\pt T_i}:=w\bigr|_{\pt T_i}-w\bigr|_{\pt T_{i-1}}$.
Furthermore, for fixed $z\in\mathcal N$, introduce 
the local cartesian coordinates $(\xi,\eta)$ 
such that $z=(0,0)$, and ${\mathbf i}_\xi$ points in the $\xi$ direction
(see Fig.\,\ref{fig_local_kappa} (left)).
In these coordinates, let $(\xi_i,\eta_i)$ be the endpoint of the edge $S_i^-=\pt T_{i-1}\cap\pt T_i$ on $\pt\omega_z$.
%

Now, a calculation shows that
$|S_{i}^-|\bigl[\partial_\bnu u_h\bigr]_{\pt T_{i-1}\cap \pt T_i}=\eta_i \llbracket\pt_\xi  u_h\rrbracket-\xi_i \llbracket\pt_\eta  u_h\rrbracket$,
where, by ${\mathcal A}^*_{\rm ani}$ and the maximum angle condition, any $|\eta_i|\lesssim h_z$, while $\displaystyle H_z^+:=\min_{i=2,\ldots m}\xi_i\sim H_z$ and $0\le \xi_i-H_z^+\lesssim h_z$,
so
\beq\label{a}
\sigma_z=
-\sum_{i=2}^{m} H_z^+ \llbracket\pt_\eta  u_h\rrbracket_{\pt T_{i-1}\cap \pt T_i}+{\mathcal O}(h_z|J_z|)+\eps^{-2}\widebar F_{z}\!\sum_{i=1,m}\!\!\!|T_i|.
\eeq
Here we also used $\theta_{T_i;z}=0$ for $i=2,\ldots,m-1$ and $\theta_{T_i;z}=\frac12$ for $i=1,m$
in view of $T_i\subset{\mathcal T}^*$ for any $T_i\subset\omega_z$ (see also \eqref{Q_T},\,\eqref{Tcal_star}).

Next,
multiplying \eqref{FEM} combined with \eqref{f_bar_z} by $2\eps^{-2}$, then subtracting $\sigma_z$ and applying a similar argument,
one gets
\beq\label{b}
-\sigma_z=
\sum_{i=m+2}^{N} H_z^- \llbracket\pt_\eta  u_h\rrbracket_{\pt T_{i-1}\cap \pt T_i}+{\mathcal O}(h_z|J_z|)+\eps^{-2}\widebar F_{z}\!\!\!\!\sum_{i=m+1,N}\!\!\!\!|T_i|.
\eeq
Here  $\displaystyle H_z^-:=\min_{i=m+2,\ldots N}|\xi_i|\sim H_z$,
and we also used $|S_i^-|\sim h_z$ for $i=1,m+1$.

Finally, using $|T_1|=\frac12 |\xi_2\eta_1-\xi_1\eta_2|=\frac12 \eta_1H_z^+ +{\mathcal O}(h_z^2)$
(where $|\xi_1|+|\eta_2|\lesssim h_z$) and similar observations for the other triangle areas,
we arrive at
$$
\sum_{i=1,m}\!\!H_z^-|T_i|-\sum_{i=m+1,N}\!\!\! H_z^+|T_i|={\mathcal O}(h_z^2 H_z).
$$
Combining this with \eqref{a},\,\eqref{b} 
and $\sum_{i=1}^N\llbracket\pt_\eta  u_h\rrbracket_{\pt T_{i-1}\cap \pt T_i}=0$ yields
$$
(H_z^-+H_z^+)\sigma_z=\!\! \sum_{i=1,m+1}\!\!H_z^-H_z^+ \,\llbracket\pt_\eta  u_h\rrbracket_{\pt T_{i-1}\cap \pt T_i}+ h_z H_z\,{\mathcal O}\bigl(|J_z|+h_z\eps^{-2}|\widebar F_{z}|\bigr).
$$
The desired assertion \eqref{bound_lem_jump} follows in view of
$H_z^+=H_{T_i}+{\mathcal O}(h_z)$ for $i=1,m$ and a similar relation with $H_z^-$ for $i=m+1, N$,
as well as
$$
\llbracket\pt_\eta  u_h\rrbracket_{\pt T_{i-1}\cap \pt T_i}
=\bmu_{i-1}^+\cdot{\mathbf i}_\xi\,\,[\pt_\bnu u_h]_{\pt T_{i-1}\cap \pt T_i}\,.
$$
The latter follows from $\pt_\eta={\mathbf i}_\eta\cdot \nabla$ combined with $\llbracket\nabla  u_h\rrbracket_{\pt T_{i-1}\cap \pt T_i}=\bnu_i^-[\pt_\bnu u_h]_{\pt T_{i-1}\cap \pt T_i}$
and $\bnu_i^-\cdot{\mathbf i}_\eta=\bmu_{i-1}^+\cdot{\mathbf i}_\xi\,$.
\smallskip

(ii) 
If $z\in {\mathcal N}_{\rm ani}\backslash\pt\Omega$, then ${\mathcal A}1$ implies $|\bmu_{i-1}^+\cdot{\mathbf i}_\xi|\lesssim h_z H_z^{-1}$,
so \eqref{bound_lem_jump} implies \eqref{bound_lem_jump_A1}.
It remains to consider $z\in{\mathcal N}^*_{\pt\Omega}$. In view of ${\mathcal A}1$ and \eqref{N_boundary}, one may choose  the  unit vector ${\mathbf i}_\xi$ in part~(i) of this proof to be normal to $\pt\Omega$.
As $u_h=\pt_\eta  u_h=0$ on $\pt\Omega$ so $\sum_{i=2}^{m}  \llbracket\pt_\eta  u_h\rrbracket_{\pt T_{i-1}\cap \pt T_i}=0$, so \eqref{a} again yields the desired assertion
\eqref{bound_lem_jump_A1}.
\end{proof}
\medskip

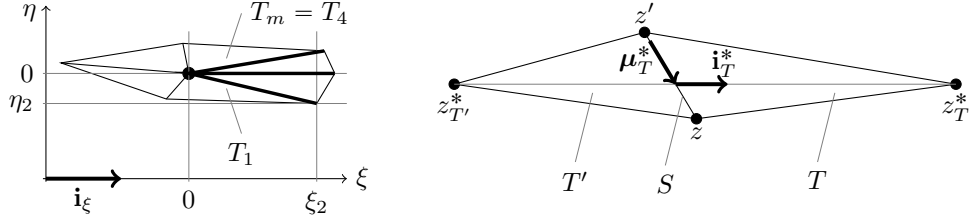
\begin{figure}[!t]
\begin{tikzpicture}[scale=0.20]
  \draw[->] (-0.2,0) -- (20,0) node[right] {$\xi$};
  \draw[line width=1.5pt, ->] (0,0) -- (5,0) ; \node[below] at (2.5,0) {${\mathbf i}_{\xi}$};
  \draw[->] (0,-0.2) -- (0,11.5) ;
  \node[left] at (0,11) {$\eta$};
\draw[fill] (9.5,7) circle [radius=0.4];
\path[draw]  (8,5.3)--(18,5)--(19.2,7)--(18.5,8.5)--(9.1,9)--(1,7.7)--cycle;
\path[draw] (8,5.3)--(9.5,7)--(18,5);
\path[draw] (19.2,7)--(9.5,7)--(18.5,8.5);
\path[draw] (9.1,9)--(9.5,7)--(1,7.7);
%
%
\path[draw,help lines]  (18,-0.2)node[below,black] {$\xi_2$}--(18,9.8);
\path[draw,help lines]  (9.5,-0.2)node[below,black] {$0$}--(9.5,9.8);
\path[draw,help lines]  (-0.2,7)node[left,black] {$0$}--(20,7);
\path[draw,help lines]  (-0.2,5)node[left,black] {$\eta_2$}--(20,5);
\path[draw,help lines](13,3)node[below,black] {$T_1$}--(12,6);
\path[draw,help lines](13,11)node[right,black] {$T_m=T_4$}--(12,8);
\path[draw,line width=1.3 pt] (18,5)--(9.5,7)--(19.2,7);
\path[draw,line width=1.3 pt] (9.5,7)--(18.5,8.5);
\end{tikzpicture}
\hfill
\begin{tikzpicture}[scale=0.23]
\draw[fill] (0,-2) circle [radius=0.3];
\draw[fill] (-3,3) circle [radius=0.3];
\draw[fill] (15,0) circle [radius=0.3];
\draw[fill] (-14,0) circle [radius=0.3];
\path[draw]  (0,-2)node[below]{$z$}--(15,0)node[below]{$z^*_{T}$}--(-3,3)node[above]{$z'$}--(-14,0)node[below]{$z^*_{T'}$}--cycle;
\path[draw]  (0,-2)--(-3,3);
\path[draw, help lines](15,0)--(-14,0);
\draw[line width=1.5pt, <-] (-1.2,0) -- (-2.7,2.5); \node[left] at (-1.9,1.4) {${\bmu}_T^*$};
\draw[line width=1.5pt, ->] (-1.2,0) -- (1.8,0) ; \node[above] at (1.6,-0.1) {${\mathbf i}_T^*$};
\path[draw,help lines](7,-4.5)node[below,black] {$T$}--(8,-0.5);
\path[draw,help lines](-7,-4.5)node[below,black] {$T'$}--(-6,-0.5);
\path[draw,help lines](-1.8,-4.7)node[below,black] {$S$}--(-0.8,-0.7);
\path[draw,white] (0,-8)--(0,-8);
\end{tikzpicture}
\caption{Notation used in 
Lemma~\ref{lem_jump} (left):
the edges $\{S_i^-\}_{i=2}^m$ (here $m=4$) are highlighted,  the unit vector ${\mathbf i}_{\xi}$ in the direction of any of these edges (here $S_3^-$),
local coordinates $(\xi,\eta)$, the endpoint $(\xi_2,\eta_2)$ of the edge $S_2^-$.
Notation used in the definition \eqref{btauS} of $\btauS^J$ (right).}
\label{fig_local_kappa}
\end{figure}

\section{Construction of  $\btau_z$ for $ h_z \lesssim \eps$ under weaker condition ${\mathcal A}1^*$.
Proof of Theorem~\ref{theo_main_bounds}(ii)}\label{sec_A1_star}
This section deals with a weaker version ${\mathcal A}1^*$  of ${\mathcal A}1$.
For this we need to address the terms  subtracted from $\sigma_z$ in \eqref{bound_lem_jump} (which are $\lesssim h_z |J_z|$ under assumption ${\mathcal A}1$, but not under ${\mathcal A}1^*$).

Let ${\mathcal S}^*\color{blue}\subset{\mathcal S}$ in the definition \eqref{tau_def} of $\btau$ be
\beq\label{S_star}\color{blue}
{\mathcal S}^*:=\!\bigl\{ S\mbox{~is shortest edge in $T$ and $T'$} :\,
h_T\ll H_T,\; h_{T'}\ll H_{T'},\;
h_T\sim h_{T'}\lesssim \eps \bigr\}.
\eeq
Roughly speaking, ${\mathcal S}^*\subset\mathcal S$
is the set of short edges shared by pairs of anisotropic triangles.
Now, we include a non-trivial component $\sum_{S\in{\mathcal S}^*} \btauS^J$ in $\btau$, where
$\btauS^J$ has support on $T\cup T'$
for any $S=\pt T\cap\pt T'\in {\mathcal S}^*$, and
\beq\label{btauS}
\btauS^J:={\color{blue}\kappa_S}
(d^*_T)^{-1}(\phi_z+\phi_{z'})\bmu^*_T\;\;\mbox{in~}T,
\qquad
{\color{blue}\kappa_S}
:=\bmu_T^*\cdot{\mathbf i}_T^*\,\,\frac{H_T H_{T'}}{H_T+H_{T'}}\,[\pt_\bnu u_h]_{S}\,.
\eeq
Here
$d_T^*:=2|T||S|^{-1}$,
$\bmu_T^*$ is the tangential unit vector along $S$ in the counterclockwise direction in $T$,
the edge $S$ joins the nodes $z'$ and $z$,
with  $\bmu_T^*$ pointing from $z'$ to $z$,
and ${\mathbf i}^*_T$ is the unit vector in the direction from $z^*_{T'}$ to $z^*_T$, the latter being the vertices opposite to $S$
in $T'$ and $T$ respectively;
see Fig.\,\ref{fig_local_kappa} (right).
Note that $\bmu^*_T=-\bmu^*_{T'}$ and
${\mathbf i}^*_T=-{\mathbf i}^*_{T'}$ so
{\color{blue}the definition of $\kappa_S$ is consistent for $T$ and $T'$.}
\medskip

\begin{lemma}\label{lem_tauS}
For any $S=\pt T\cap\pt T'\in {\mathcal S}^*$ with the notation \eqref{S_star},\,\eqref{btauS}, one has~\eqref{bound_theo_S} and
\beq
\label{taus_jumps}
\btauS^J\cdot \bnu=0\;\;\mbox{on~}S,\qquad
|S''|\btauS^J\cdot \bnu={\color{blue}\kappa_S}(\phi_z-\phi_{z'})\;\;\mbox{on~any~edge~}S''\subset \pt T\backslash S.
\eeq
\end{lemma}\vspace{-0.15cm}
\begin{proof}
For ${\rm div} \btauS^J=0$ in \eqref{bound_theo_S}, as well as \eqref{taus_jumps},
imitate the proof of Lemma~\ref{lem_tau_case1}.
For the {\color{blue}first} bound on $\eps\btauS^J$ in \eqref{bound_theo_S}, note that
$d_T^*\sim d_{T'}^*\sim H_T$ implies
$|\btauS^J|\lesssim H_T^{-1}|{\color{blue}\kappa_S}|\lesssim | [\pt_\bnu u_h]_{S}|$.
{\color{blue}For the final assertion in \eqref{bound_theo_S}, note that
$|T\cup T'||\omega_z|^{-1}\lesssim (h_T+h_{T'})h_z^{-1}\lesssim \min\{1,\eps h_z^{-1}\}$.}
\end{proof}
\medskip

Now that $\sum_{S\in{\mathcal S}^*} \btauS^J$ is included in $\btau$, we need to ensure that $\btau$ still satisfies~\eqref{tau_jump}.
For this, the definition of $\btauz$  should be updated to take into account the possibly non-trivial jumps $[\btauS^J\cdot \bnu]$ across $\gamma_z$.
{\color{blue}For a possible modification of $\btauz$ in the case $h_z\gtrsim \eps$, see Remark~\ref{rem_kappa_case2}.}

For $h_z\lesssim \eps$, the definition \eqref{tauz_h_small} of $\btauz$ is tweaked as follows.
Relations \eqref{tau_z} and \eqref{F_z_T} remain unchanged, while in \eqref{tau_n_hat} we replace
 $\{\beta_i\}_{i=1}^{N_z}$ by $\{\beta_i^*\}_{i=1}^{N_z}$,
 where
 \beq\label{beta_star}\color{blue}
 \beta_i:=\beta_i^*+\kappa_{i}+\kappa_{i+1},
 \qquad
\kappa_i:=\left\{\begin{array}{cl}
\kappa_{\pt T_{i-1}\cap\pt T_i}& \mbox{if~~} \pt T_{i-1}\cap\pt T_i\subset {\mathcal S}^*,\\[3pt]
0&\mbox{otherwise.}
\end{array}\right.
\eeq

\begin{remark}[Anisotropic flux equilibration]
The observations of Remark~\ref{rem_beta} remain valid
(although a version of system \eqref{tau_n_hat}
is now solved for $\{\beta_i^*\}$ ).
One can simply use \eqref{beta_choice_case1}
(or choose $\{\beta_i\}$ as in the proof of Lemma~\ref{lem_beta_star}).
{\color{blue}Similarly, one can minimize \eqref{optima}, only subject to \eqref{tau_n_hat} with $\{\beta_i\}$ replaced by $\{\beta_i^*\}$ in the latter.}
\end{remark}%
\medskip

\begin{lemma}\label{lem_beta_star}
Under condition ${\mathcal A}1^*$, for any $z\in\mathcal N$ with $h_z\lesssim \eps$,
let $\btauz$ be defined by \eqref{tau_z} and \eqref{F_z_T} and use the notation \eqref{S_star},\,\eqref{btauS}.
Then there is a solution of the system \eqref{tau_n_hat},
in which $\{\beta_i\}_{i=1}^{N_z}$ is replaced by $\{\beta_i^*\}_{i=1}^{N_z}$ of \eqref{beta_star},
such that  $\btau$ of \eqref{tau_def} satisfies \eqref{tau_jump} and
the results of Lemmas~\ref{lem_tau_case1}, \ref{cor_case1} and~\ref{lem_beta} remain true.
\end{lemma}
\smallskip

\begin{proof}
One can easily check that, indeed, the results of Lemmas~\ref{lem_tau_case1} and \ref{cor_case1} remain true,
while for Lemma~\ref{lem_beta}, it suffices to obtain~\eqref{case1_beta_desired}.
That the normal jumps in $\btau$ satisfy
\eqref{tau_jump} can be checked by a direct calculation using \eqref{taus_jumps}
and taking into account that if
$(\pt{T_{i-1}}\cap\pt T_i)\subset(\gamma_z\cap{\mathcal S}^*)$, then
$\beta_i-\beta_{i-1}=\beta_i^*-\beta_{i-1}^*$ (in view of $\color{blue}\kappa_{{i-1}}=\kappa_{i+1}=0$).
Note that in the latter case \eqref{beta_aux} is still true.

Otherwise, i.e. for $(\pt{T_{i-1}}\cap\pt T_i)\not\subset(\gamma_z\cap{\mathcal S}^*)$,
 a version of \eqref{beta_aux} will be employed:
\beq\label{beta_aux_new}
\mbox{if~}|S_i^-|\sim d_i \gtrsim d_{i-1}
\;\Rightarrow\;
|\beta_id_i^{-1}|\lesssim |\beta_{i-1}d_{i-1}^{-1}|+\max_j|\alpha_{j}|+
{\color{blue}\max_j|\kappa_{j}|}
d_i^{-1}+|J_z|.
\eeq

Next,  consider two cases (a) and (b), as in the proof of Lemma~\ref{lem_beta}, to get the bound \eqref{case1_beta_desired} for $\beta_i d_i^{-1}$ and thus complete the proof.
\smallskip

(a) Suppose that $z$ satisfies ${\mathcal A}1_{\rm mix}$.
Unless $\gamma_z\cap {\mathcal S}^*=\emptyset$ (and so the results of Lemma~\ref{lem_beta} apply),
$\mathring{\omega}_z=\emptyset$ and $\mathring{\mathcal S}_z=\gamma_z\cap {\mathcal S}^*$ contains exactly one edge $\pt T_1\cap\pt T_2$.
{\color{blue}Then note that $\kappa_i=0$ unless $i=2$.}
Set $\beta_1:=0$ and use \eqref{beta_aux} with $i=2$ as in the proof of Lemma~\ref{lem_beta}.
For $i>2$, use \eqref{beta_aux_new}, where $d_i\sim H_z$, so the additional term
$\color{blue}\max|\kappa_j|d_i^{-1}=|\kappa_{2}|H_z^{-1}\lesssim |J_z|$.
So we get~\eqref{case1_beta_desired} with $\hat\sigma_z:=0$.
\smallskip

(b) It remains to consider
$z\in{\mathcal N}_{\rm ani}\backslash\pt\Omega$ under condition ${\mathcal A}1_{\rm ani}^*$
(as $z\in{\mathcal N}_{\rm ani}\cap\pt\Omega$ satisfies either ${\mathcal A}1_{\rm ani}$ or ${\mathcal A}1_{\rm mixed}$, so have been considered in part (a) or in Lemma~\ref{lem_beta}).
We shall imitate part~(b) from the proof of Lemma~\ref{lem_beta}.
Note that
$(\pt{T_{i-1}}\cap\pt T_i)\subset(\gamma_z\cap{\mathcal S}^*)$
{\color{blue}(and so $\kappa_i\neq 0$)}
only for $i=1,m+1$.
Hence, we employ \eqref{beta_aux} for $i=1,m+1$, and
\eqref{beta_aux_new} with $d_i\sim H_z$ for $i\neq 1,m, m+1, N$,
and it remains to bound $\beta_i d_i^{-1}$ for $i=0,m$.
For the latter, combining \eqref{beta_0_m} with \eqref{beta_star} yields
$$
[\beta_0-{\color{blue}\kappa_{1}}]-[\beta_m-{\color{blue}\kappa_{m+1}}]
=
\widetilde\sigma_z
+{\mathcal O}\bigl(h_z|J_z|\bigr).
$$
From this bound, one gets \eqref{case1_beta_desired} only now
$\hat\sigma_z:=h_z^{-1}(\widetilde\sigma_z+\color{blue}\kappa_{1}-\kappa_{m+1})$. 
As $|\widetilde\sigma_z-\sigma_z|$ was bounded in the proof of Lemma~\ref{lem_beta}, to complete the proof, it 
remains to show 
%
$$
\eps h_z^{-1}|\sigma_z+{\color{blue}\kappa_{1}-\kappa_{m+1}}|\lesssim |\eps J_z|.
$$
The above follows from \eqref{bound_lem_jump} using the following observations.
For $\color{blue}\kappa_{1}=\kappa_{\pt{T_{0}}\cap\pt T_{1}}$,
we use 
{\color{blue}the triangle $T_0$} with
$\bmu^*_{T_0}=\bmu_0^+$
and
$|{\mathbf i}^*_{T_0}+{\mathbf i}_\xi|\lesssim h_z H_z^{-1}$.
Similarly, for $\color{blue}\kappa_{m+1}=\kappa_{\pt{T_{m}}\cap\pt T_{m+1}}$, we use
{\color{blue}the triangle $T_m$ with}
$\bmu^*_{T_m}=\bmu_m^+$
and
$|{\mathbf i}^*_{T_m}-{\mathbf i}_\xi|\lesssim h_z H_z^{-1}$.
\end{proof}
\smallskip

\section{Construction of  $\btau_z$ for $h_z\gtrsim \eps$ under condition ${\mathcal A}2$}\label{sec_A2}
Throughout this section, for any $T\subset\omega_z$, we use
$S_T$, $\bnu_T$ and $\bmu_T$, as well as $S_T^\pm$, $\bnu_T^\pm$ and $\bmu_T^\pm$,
defined as in \S\ref{sec_A1} (see Fig.\,\ref{fig_tau_z_case1} (right)) only with subscript $T$ in place of $i$ when dealing with element $T$.
We start with {\color{blue}two useful technical results.}
\medskip

\begin{figure}[t!]
~\hfill
\begin{tikzpicture}[scale=0.21]
\path[draw, fill=gray!30!white] (-1.153014337, 4.626351703)--(0,6.5)--(12,0)--cycle;
\draw[fill] (0,6.5) circle [radius=0.4];
\path[draw]  (0,6.5)node[left] {$z$}--(-4,0)--(12,0)--cycle;
\node[left] at (-1.8,3.25) {$S_T^-$};
\node[right] at (6.8,3.25) {$S_T^+$};
%
%
\draw[help lines] (0, 5.25)--(4,7) node[right, black] {$T^+$};
\end{tikzpicture}
\hfill
\begin{tikzpicture}[scale=0.21]
\path[draw,fill=gray!30!white] (1.934442327, 5.452177073)--(-4,0)--(0,6.5)--cycle;
\draw[fill] (0,6.5) circle [radius=0.4];
\path[draw]  (0,6.5)node[left] {$z$}--(-4,0)--(12,0)--cycle;
\node[left] at (-1.8,3.25) {$S_T^-$};
\node[right] at (6.8,3.25) {$S_T^+$};
\draw[help lines] (0, 5.25)--(4,7) node[right, black] {$T^-$};
%
%
\end{tikzpicture}
\hfill~
\caption{Notation used in 
Lemma~\ref{lem_case2}: each triangle $T^\pm\!\!\subset T$ shares the edge $S_T^\pm$ with $T$
and has another edge of length $\eps'\sim \eps$ along $S_T^\mp$.
}
\label{fig_T_pm_notation}
\end{figure}
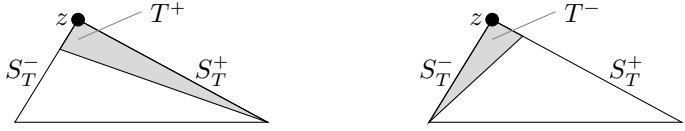

\begin{lemma}\label{lem_case2}
Let $h_T\ge\sqrt{6}\eps$.
For any triangle $T$ with a vertex $z$, there exist two functions
$\btau_{z;T}^+$ and $\btau_{z;T}^-$ in $T$ such that
\begin{subequations}\label{tau_z_case2_bounds}
\begin{eqnarray}
\label{tau_z_case2_bounds_a}
&\btau_{z;T}^\pm\cdot\bnu=0\;\;\mbox{on~}\pt T\backslash S_T^\pm,\qquad
\btau_{z;T}^\pm\cdot\bnu=\phi_z\;\;\mbox{on~}S_T^\pm,
\\[0.2cm]
\label{tau_z_case2_bounds_b}
&\|\eps\,{\rm div}\btau_{z;T}^\pm\|^2_{T}=\|\btau_{z;T}^\pm\|^2_{T}
={\textstyle \frac1{2\sqrt{6}}}\eps|S_T^\pm|\varsigma_{T;z}^{-1},
\quad\;\;
\varsigma_{z;T}:=
\sin\angle(S_T^-,S_T^+).
\end{eqnarray}
\end{subequations}
\end{lemma}

\begin{proof}
Let $\eps':=\sqrt{6}\eps$ and, skipping the subscripts when there is no ambiguity, set
$$
\btau^+:=-\varsigma^{-1} \varphi^+_z\bmu^-,\qquad
\btau^-:=\varsigma^{-1} \varphi^-_z\bmu^+.
$$
Here
$\varphi^+_z$ is a barycentric coordinate in the triangle $T^+$ formed by the edge $S^+_T=S^+$  and the point $z+\eps'\bmu^-$ such that
$\varphi^+_z\bigr|_z=1$; see Fig.\,\ref{fig_T_pm_notation}.
Similarly, $\varphi^-_z$ is a barycentric coordinate in the triangle $T^-$ formed by the edge $S^-_T=S^-$ and the point $z-\eps'\bmu^+$.

The boundary properties \eqref{tau_z_case2_bounds_a} are satisfied as $\bmu^\pm\cdot\bnu^\pm=0$ and $\varsigma=-\bmu^-\cdot\bnu^+=\bmu^+\cdot\bnu^-$.
Next,
using $|T^+|=\frac12\eps'|S^+|\varsigma$,
$$
\|\btau^+\|^2_{T}
={\textstyle \frac16}\varsigma^{-2}|T^+|
={\textstyle \frac1{2\sqrt{6}}}\eps\varsigma^{-1}|S^+|,
$$
while ${\rm div}\btau^+=-\varsigma^{-1}\pt_{\bmu^-} \varphi^+_z=\varsigma^{-1}\eps'^{-1}$ so
$\|\eps\,{\rm div}\btau^+\|^2_{T}=(\eps/\eps')^2\varsigma^{-2}|T^+|=\|\btau^+\|^2_{T}$.
\end{proof}
\medskip

\begin{remark}[Version of ${\mathcal A}2$]\label{rem_sqrt6}
In ${\mathcal A}2$, one can impose that
each $z\in\mathcal N$ with $h_z\gtrsim \eps$
 satisfies $\mathring{h}_z\ge c'\eps$ for any fixed positive constant $c'$ (rather than $\mathring{h}_z\ge \sqrt{6}\eps$).
 For this case, one can employ a version of the above lemma  under the condition $h_T\ge c'\eps$.
 Choosing $\eps':=c'\eps$ in the proof, one, indeed, arrives at the following version of
 \eqref{tau_z_case2_bounds_b}:
 $$
\|\eps\,{\rm div}\btau_{z;T}^\pm\|^2_{T}=6c'^{-2}\|\btau_{z;T}^\pm\|^2_{T}
=(2c')^{-1}\eps|S_T^\pm|\varsigma_{T;z}^{-1}\,.
$$
\end{remark}

\newpage
\begin{lemma}\label{lem_case22}\color{blue}
For any triangle $T$ with a vertex $z$ and its opposite edge $S_T$ satisfying $|S_T|>4\sqrt{6}\eps$, there exists a function $\varphi_{z;T}\in C(\bar T)$ such that
\beq\label{varphi_zT}
\varphi_{z;T}=\phi_z\;\;\mbox{on~}\pt T, \qquad
\|\eps\,{\rm div}(\varphi_{z;T}\,\bmu_T)\|^2_{T}\le\|\varphi_{z;T}\|^2_{T}
={\textstyle \frac2{\sqrt{6}}}\eps|T||S_T|^{-1}\lesssim \eps H_T.
\eeq
\end{lemma}

\begin{proof}\color{blue}
Introduce the two triangles $T^-,T^+\subset T$, each $T^\pm$ formed by the edge $S_T^\pm$ and  a common vertex lying on the median of $T$ originating at $z$,
subject to $|T^\pm|=\sqrt{6}\eps|T||S_T|^{-1}<\frac14|T|$ (see Fig.\,\ref{fig_T_pm_notation_2}).
Now, define a unique $\varphi_{z;T}\in C(\bar T)$ that (i) satisfies $\varphi_{z;T}=\phi_z$ on $\pt T$; (ii) has support in $T^-\cup T^+$;
(ii) is linear in each $T^\pm$.
%
Clearly, $\|\varphi_{z;T}\|^2_{T}=\frac16(|T^-|+|T^+|)$. Furthermore, a calculation shows that
$|{\rm div}(\varphi_{z;T}\,\bmu_T)|=|\pt_{\bmu_T}\varphi_{z;T}|\le |S_T|^{-1}|T|/|T^{\pm}|$
so
$|\eps\,{\rm div}(\varphi_z\bmu_T)|\le \frac1{\sqrt{6}}$ in $T^-\cup T^+$.
Combining these observations, one gets \eqref{varphi_zT}
(with the final relation in \eqref{varphi_zT} easily following from $|T|\sim h_T H_T\lesssim |S_T| H_T$).
\end{proof}

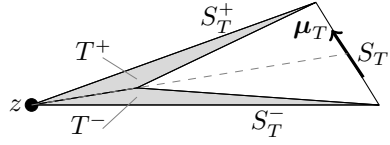
\begin{figure}[t!]
~\hfill
\begin{tikzpicture}[scale=0.21]
\path[draw, fill=gray!30!white] (0,0)--(18,6.5)--(6.67,1.08)--(22,0)--cycle;
\draw[fill] (0,0) circle [radius=0.4];
\path[draw]  (0,0)node[left] {$z$}--(18,6.5)--(22,0)--cycle;
\path[draw]  (0,0)--(6.67,1.08)--(22,0);
\path[draw]  (0,0)--(6.67,1.08)--(18,6.5);

\draw[very thick, ->] (21,1.75) -- (19,4.75) ;
\node[right] at (20,3.25) {$S_T$};
\node[below] at (15,0.5) {$S_T^-$};
\node[above] at (11.8,4) {$S_T^+$};
\node[left] at (19.4,4.75) {${\bmu}_T$};

\node[left] at (5.4,-1.2) {$T^-$};
\node[left] at (5.7,3.5) {$T^+$};

\draw[help lines,dashed] (0,0)--(20,3.25);

\draw[help lines] (6.67,0.5)--(4.5,-1.5);
\draw[help lines] (6.67,1.64)--(4.5,3);
\end{tikzpicture}
\hfill~
\caption{\color{blue}Notation used in 
Lemma~\ref{lem_case22}: each triangle $T^\pm\!\!\subset T$ is formed by the edge $S_T^\pm$
and a common vertex lying on the median of $T$ originating at $z$.}
\label{fig_T_pm_notation_2}
\end{figure}

\begin{figure}[b!]
~\hfill
\begin{tikzpicture}[scale=0.22]
\path[draw, fill=gray!30!white] (9.5,7)--(18,5)--(19,6.7)--(18.5,8.5)--cycle;
\path[draw, fill=gray!30!white](1,6.7)--(9.5,7)--(1.4,8.7)--cycle;
\draw[fill] (9.5,7) circle [radius=0.4];
\path[draw]  (8,5.1)--(18,5)--(19,6.7)--(18.5,8.5)--(9.1,9)--(1.4,8.7)--(1,6.7)--cycle;
\path[draw] (8,5.1)--(9.5,7)--(18,5);
\path[draw] (19,6.7)--(9.5,7)--(18.5,8.5);
\path[draw] (9.1,9)--(9.5,7)--(1,6.7)--(9.5,7)--(1.4,8.7);
%
\path[draw, fill=gray!30!white] (4,2.4)--(16,0)--(17,1.7)--(15.8,3.5)--cycle;
\draw[fill] (4,2.4) circle [radius=0.4];
\path[draw]  (4,0.7)--(16,0)--(17,1.7)--(15.8,3.5)--(4,4.5)--(4,2.4)--cycle;
\path[draw] (4,2.4)--(16,0);
\path[draw] (17,1.7)--(4,2.4)--(15.8,3.5);
\path[draw,ultra thick
] (4,0)--(4,5);
\end{tikzpicture}
\hfill
\begin{tikzpicture}[scale=0.22]
\path[draw, fill=gray!30!white] (9.5,7)--(20,5)--(21,6.7)--(20.5,8.5)--cycle;
\draw[fill] (9.5,7) circle [radius=0.4];
\path[draw]  (10,4.5)--(20,5)--(21,6.7)--(20.5,8.5)--(9.1,9)--(7,7.7)--(7.1,5)--cycle;
\path[draw] (10,4.5)--(9.5,7)--(20,5);
\path[draw] (21,6.7)--(9.5,7)--(20.5,8.5);
\path[draw] (9.1,9)--(9.5,7)--(7,7.7);
\path[draw] (7.1,5)--(9.5,7);
%
\path[draw, fill=gray!30!white] (9.5,0)--(20.5,1.8)--(21,0)--cycle;
\draw[fill] (9.5,0) circle [radius=0.4];
\path[draw]  (7,0)--(7.5,2.4)--(10,2.9)--(20.5,1.8)--(21,0)--cycle;
\path[draw] (7.5,2.4)--(9.5,0)--(20.5,1.8);
\path[draw] (10,2.9)--(9.5,0);
\path[draw,ultra thick
] (6,0)--(22,0);
\end{tikzpicture}
\hfill
\begin{tikzpicture}[scale=0.21]
\path[draw, fill=gray!30!white] (9.5,2.2)--(20.2,0.5)--(21,2.7)--(20.5,4.5)--cycle;
\draw[fill] (9.5,2.2) circle [radius=0.4];
\path[draw]  (10,0)--(20.2,0.5)--(21,2.7)--(20.5,4.5)--(9.1,10)--(7.1,3.5)--(7.3,0.7)--cycle;
\path[draw] (10,0)--(9.5,2.2)--(20.2,0.5);
\path[draw] (21,2.7)--(9.5,2.2)--(20.5,4.5);
\path[draw] (9.1,10)--(9.5,2.2)--(7.1,3.5);
\path[draw] (7.3,0.7)--(9.5,2.2);
\end{tikzpicture}
\hfill~
\caption{For various nodes, the set $\omega_z^*$ (defined in \eqref{omega_star_case2}) is highlighted by the grey color. 
}
\label{fig_node_types_omega_z_star}
\end{figure}
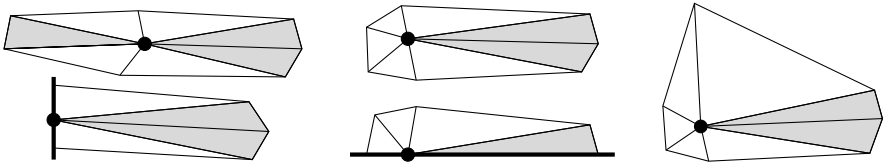

\subsection{Definition of $\btauz$ for $h_z\gtrsim \eps$}\label{ssec_case2}
%
Introduce a subset $\omega_z^*$ of $\omega_z$
{\color{blue}and, using $\varphi_{z;T}$ from Lemma~\ref{lem_case22}, a related function $\phi_z^*$:}
\beq\label{omega_star_case2}
\omega_z^*:=
\Bigl\{T\subset\omega_z:
\color{blue} |S_T|\sim h_T\ll H_T
%
%
\Bigr\},
\quad
\phi_z^*:=
\left\{\begin{array}{cl}
\varphi_{z;T}
&\mbox{on~~}T\subset \omega_z^* : |S_T|>4\sqrt{6}\eps,\\
\phi_z &\mbox{otherwise}.
\end{array}\right.
\eeq
Thus, 
$\omega_z^*$ includes only triangles with extremely small angles at $z$, so $\omega_z\backslash\omega_z^*\neq\emptyset$
(see Fig.\,\ref{fig_node_types_omega_z_star}).
{\color{blue}Note also that $\phi_z^*=\phi_z$ on $\gamma_z$.}
Now,
using  $\btau_{z;T}^\pm$ from Lemma~\ref{lem_case2},
set
\beq\label{tau_z_case2}
\btauz:=\left\{\begin{array}{ll}
{\textstyle\frac12}\Bigl({\mathcal J}_z\bigr|_{S^-_T}\btau_{z;T}^-+{\mathcal J}_z\bigr|_{S^+_T} \btau_{z;T}^+\Bigr)
&\mbox{for~~}T\subset \omega_z\backslash\omega_z^*,\\[0.3cm]
(\beta_T d_T^{-1})\,{\color{blue}\phi_z^*}\,\bmu_T,\;\;\;
d_T:=2|T|\,|S_T|^{-1}
&\mbox{for~~}T\subset \omega_z^*,
\end{array}\right.
\eeq
where, with the convention $[\partial_\bnu u_h]_{\pt\Omega}:=0$,
\beq\label{cal_J}
{\mathcal J}_z\bigr|_{S^\pm_T}:=\left\{\begin{array}{cl}
\bigl[\partial_\bnu u_h\bigr]_{S_T^\pm}&\mbox{if~~}S_T^\pm\not\subset\pt\omega_z^*,\\[0.3cm]
\displaystyle
|S_T^\pm|^{-1}\sum_{S\in\gamma_{z;T}^\pm} |S|\,\bigl[\partial_\bnu u_h\bigr]_{S}
&\mbox{otherwise}.
\end{array}\right.
\eeq
To define $\gamma_{z;T}^\pm$ in \eqref{cal_J}, it is convenient to assume that $\omega_z^*$ includes triangles with their boundaries.
Now, let $\omega_{z;T}^\pm$ be the maximal connected subset of $\omega_z^*\backslash\{z\}$  that shares the edge $S^\pm_T$ with $T$.
The set of all edges originating  at $z$ that are contained in this subset $\omega_{z;T}^\pm$ (including $S^\pm_T$) is denoted $\gamma_{z;T}^\pm$.

The unique set of values $\{\beta_T\}$ for $T\subset\omega_z^*$ in \eqref{tau_z_case2} is chosen to satisfy~\eqref{tau_n}.
For example,
consider a bundle of $m$ triangles $\omega_{z;T}^+=\omega_{z;T'}^-=\{T_i\}_{i=1}^m $, numbered counterclockwise,
that touches $T,T'\subset \omega_z\backslash\omega_z^*$.
Now, \eqref{tau_n} is equivalent to a version of \eqref{tau_n_hat}:
\begin{subequations}\label{tau_n_hat_version_ab}
\beq\label{tau_n_hat_version}
\beta_{i-1}-\beta_{i}
=|S_{i}^-| [\partial_\bnu u_h]_{\pt T_{i-1}\cap \pt T_i},
\quad i=1,\ldots, m+1,
\eeq
where
the notation $\beta_i:=\beta_{T_i}$ is used for $i=1,\ldots,m$, while
\beq\label{beta_0m}
\beta_0=-\beta_{m+1}:={\textstyle\frac12} \sum_{i=1}^{m+1}|S_{i}^-| [\partial_\bnu u_h]_{\pt T_{i-1}\cap \pt T_i}
={\textstyle\frac12}|S^+_{T}|\,{\mathcal J}_z\bigr|_{S^+_{T}}
={\textstyle\frac12}|S^-_{T'}|\,{\mathcal J}_z\bigr|_{S^-_{T'}}.
\eeq
\end{subequations}
Note that the above system involves $m+1$ equations for $\{\beta_i\}_{i=1}^m$, but is consistent and has a unique solution.
This becomes clear on application of $\sum_{i=1}^{m+1}$ to \eqref{tau_n_hat_version} which yields a relation for
$\beta_{0}-\beta_{m+1}$ consistent with \eqref{beta_0m}.

If for a bundle of $m$ triangles $\cup_{i=1}^m T_i=\omega_{z;T}^+$, numbered counterclockwise,  one has $S^+_{m}\subset\pt\Omega$,
then we use \eqref{tau_n_hat_version} with $i\neq m+1$, and $\beta_0$ from \eqref{beta_0m} (while $\beta_{m+1}$ remains undefined).
Similarly, if $S^-_{1}\subset\pt\Omega$, then use \eqref{tau_n_hat_version} with $i\neq 1$ combined with the definition of $\beta_{m+1}$ from
\eqref{beta_0m} (and $\beta_{0}$ remaining undefined).
\smallskip

\begin{remark}[$\gamma_z\cap{\mathcal S}^*\neq\emptyset$]\label{rem_kappa_case2}
\color{blue}Note that ${\mathcal S}^*$, defined by \eqref{S_star}, can be chosen so that
$\gamma_z\cap{\mathcal S}^*=\emptyset$ whenever $h_z\gtrsim \eps$ under condition ${\mathcal A}2$
(as then $h_T \gtrsim \eps$).
If, however, $\gamma_z\cap{\mathcal S}^*\neq\emptyset$, then the non-trivial jumps $[\btauS^J\cdot \bnu]$ across $\gamma_z$
are easily taken into account by replacing
$[\partial_\bnu u_h]_{\pt T_{i-1}\cap \pt T_i}$
with
$[\partial_\bnu u_h]_{\pt T_{i-1}\cap \pt T_i}+|S_{i}^-|^{-1}(\kappa_{i-1}-\kappa_{i+1})$
in \eqref{cal_J} and \eqref{tau_n_hat_version_ab}
(where $\kappa_i$ is defined in \eqref{beta_star}).
With this modification, Lemma~\ref{lem_theo_bound_case2} below remains valid as $|\kappa_i|\lesssim H_z|J_z|\;\forall i$
(the latter follows from \eqref{btauS}).
\end{remark}

\subsection{Proof of \eqref{tau_z_main} in Theorem~\ref{theo_main_bounds} for $h_z\gtrsim \eps$}

It suffices to prove the following.
\smallskip
\begin{lemma}\label{lem_theo_bound_case2}
Under condition ${\mathcal A}2$, for any $z\in{\mathcal N}$ with $h_z\gtrsim \eps$, the function
$\btauz$ defined by \eqref{tau_z_case2_bounds},\,{\color{blue}\eqref{varphi_zT}},\,\eqref{omega_star_case2}\,\eqref{tau_z_case2},\,\eqref{cal_J},\,\eqref{tau_n_hat_version_ab} satisfies \eqref{tau_n} and
\beq\label{theo_bound}
\|\eps^2\,{\rm div}\btauz\|^2_{\color{blue}\omega_z}+\|\eps\btauz\|^2_{\color{blue}\omega_z}
%
\color{blue}
\lesssim
\sum_{S\in \gamma_z}\!\!\eps |S|\,\,\bigl(\eps[\pt_\bnu u_h]_S\bigr)^2
\lesssim\eps h_z^{-1}\,\|\eps J_z\|^2_{\omega_z}
\,,
\vspace{-0.2cm}
\eeq
{\color{blue}where $\eps |S|\sim\min\{\eps |S|,\,|\omega_z|\}$ for any $S\in \gamma_z$.}
\end{lemma}
\smallskip

\begin{proof}
Condition \eqref{tau_n} is satisfied by the construction of $\btauz$ in \S\ref{ssec_case2},
so it remains to establish \eqref{theo_bound}.

\color{blue}
First, for each fixed $S\in \gamma_z\cap\mathring{\mathcal S}_z$ (i.e. $ |S|\sim \mathring{h}_z$), we shall trace the contribution  of
$[\pt_\bnu u_h]_S$ to $\btauz$ of \eqref{tau_z_case2}. 
In this case, $[\pt_\bnu u_h]_S$ is involved in $\btauz$
only on the triangles adjacent to $S$ (such triangles are not in $\omega_z^*$)
in the form of  the terms ${\mathcal J}_z\bigr|_{S}=[\pt_\bnu u_h]_S$. 
Hence, the contribution of the considered $[\pt_\bnu u_h]_S$
to the left-hand side of \eqref{theo_bound}
is indeed bounded by $\eps |S|(\eps[\pt_\bnu u_h]_S)^2$, as can be shown by an application of \eqref{tau_z_case2_bounds_b} with $|S_T^\pm|=|S|$ and $\varsigma_{z;T}^{-1}\sim 1$.
Furthermore,
 $|S|\lesssim H_z$ implies $\eps |S|\lesssim \eps h_z^{-1}|\omega_z|$ and so
$\eps |S|(\eps[\pt_\bnu u_h]_S)^2
\lesssim\,\eps h_z^{-1}\|\eps[\pt_\bnu u_h]_S\|^2_{\omega_z}$.

It remains to bound the contribution to the left-hand side of \eqref{theo_bound} of
$[\pt_\bnu u_h]_S$ for the edges $S\not\in \mathring{\mathcal S}_z$.
In this case, $|S|\sim H_z$, so 
$\eps |S|\bigl(\eps[\pt_\bnu u_h]_S\bigr)^2
\sim\,\eps h_z^{-1}\bigl\|\eps[\pt_\bnu u_h]_S\bigr\|^2_{\omega_z}$.
This observation implies that it now suffices to prove only the second relation in \eqref{theo_bound}, or, equivalently, show that
$\|\eps\,{\rm div}\btauz\|^2_{T}+\|\btauz\|^2_{T}\lesssim\eps H_z| J_z|^2$ for any $T\subset\omega_z$, to which we proceed.

Suppose $T\subset\omega_z^*$.
By \eqref{omega_star_case2},
$d_T\sim H_T\sim H_z$, and, by \eqref{tau_n_hat_version_ab}, $|\beta_T|\lesssim H_z|J_z|$, so
$|d_T^{-1}\beta_T|\lesssim |J_z|$.
Now, if $\phi_z^*=\varphi_{z;T}$ in $T$, then the desired bound on $\|\eps\,{\rm div}\btauz\|^2_{T}+\|\btauz\|^2_{T}$ follows from \eqref{varphi_zT}
combined with $\eps H_T\le\eps H_z$.
Otherwise, $\phi_z^*=\phi_z$ in $T$ implies  ${\rm div}\btauz=0$ (see the proof of Lemma~\ref{lem_tau_case1}), and also $h_T\lesssim \eps$ (in view of the definition of $\phi_z^*$ in \eqref{omega_star_case2}),
so again $\|\btauz\|^2_{T}\lesssim |T||J_z|^2\lesssim \eps H_z|J_z|^2$.
Finally, suppose $T\subset\omega_z\backslash\omega_z^*$.
Note that $|{\mathcal J}_z|\lesssim|J_z|$, which follows from \eqref{cal_J} as $|S|\sim |H_T|$ for all edges
$S\in \gamma_{z;T}^\pm $ including $S_T^\pm$.
Now,  the desired bound on $\|\eps\,{\rm div}\btauz\|^2_{T}+\|\btauz\|^2_{T}$ follows from \eqref{tau_z_case2_bounds_b} with $|S_T^\pm|\lesssim H_z$ and $\varsigma_{z;T}^{-1}\sim 1$.
\end{proof}\vspace{-0.05cm}

\section{Numerical results}
\label{ssec_numer}

\begin{figure}[!b]
~\hfill
\includegraphics[width=0.4\textwidth]{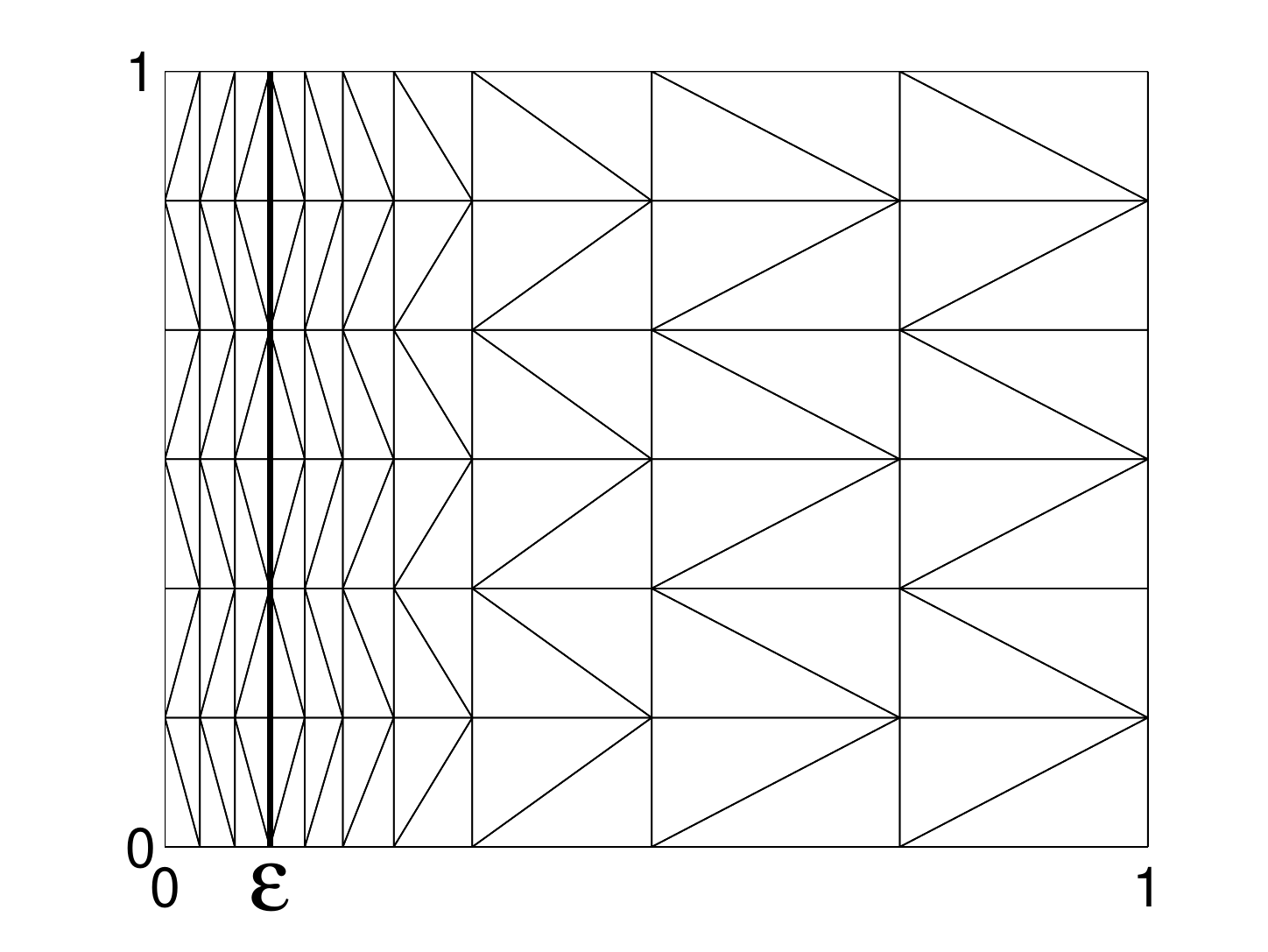}
\includegraphics[width=0.4\textwidth]{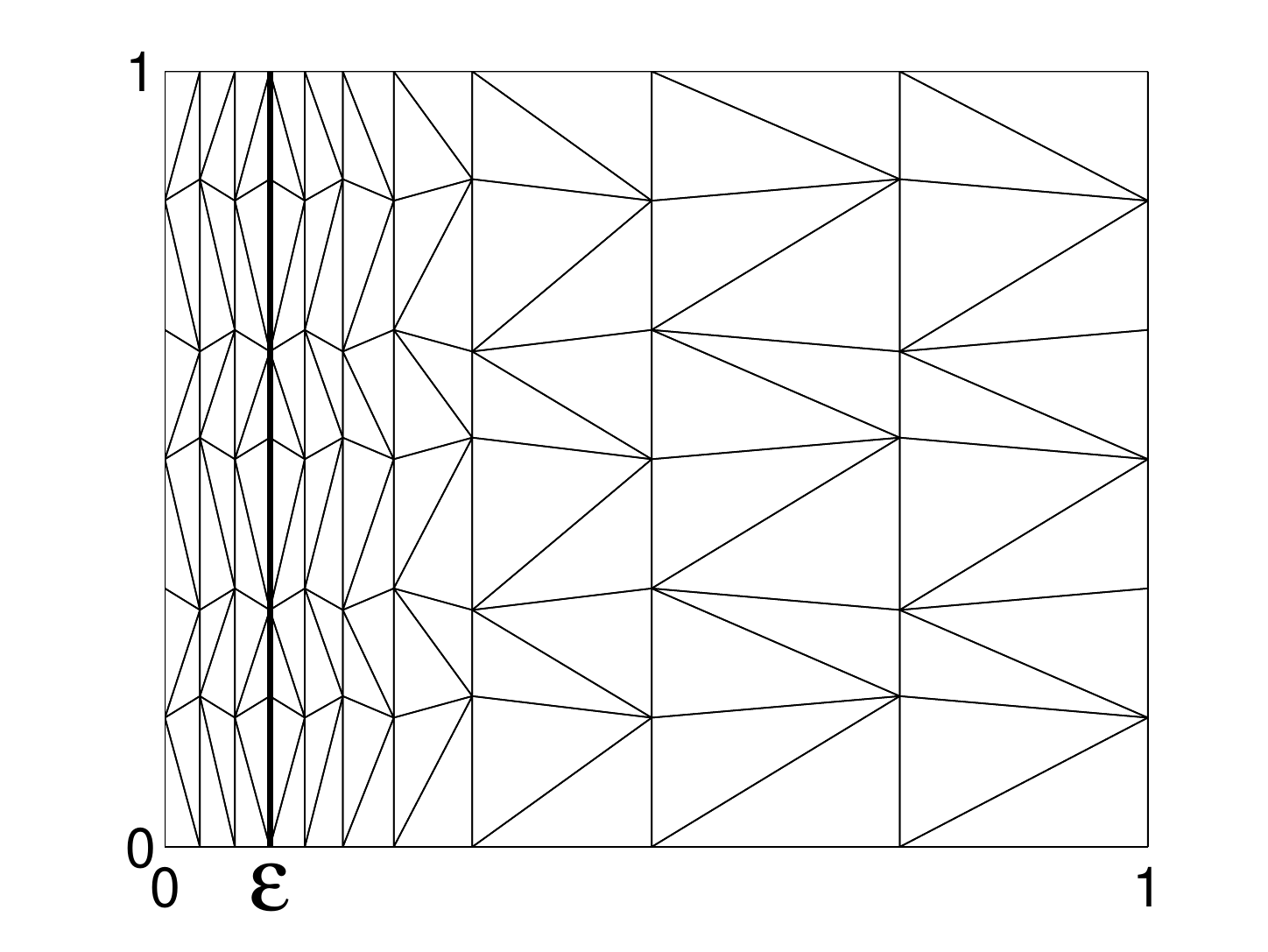}
\hfill~
\vspace{-0.3cm}
\caption{\color{blue}Non-obtuse triangulation used in \S\ref{ssec_numer} (left);
its version with obtuse triangles (right).
}
\label{fig_partial_new}
\end{figure}

Our estimator is tested
using a simple version of (\ref{eq1-1}) with $\Omega=(0,1)^2$ and
$f=u-F(x,y)$, where $F$ is such that the unique exact solution
$u=4y\,(1-y)\,[C_u\cos(\pi x/2)- (e^{-x/\eps}-e^{-1/\eps})/(1-e^{-x/\eps})]$
(the latter exhibits a sharp boundary layer at $x=0$);
{\color{blue}the constant parameter $C_u$ in $u$ will take values $1$ and $0$.}
We consider {\color{blue}an a-priori-chosen layer-adapted non-obtuse triangulation, as on Fig.\,\ref{fig_partial_new} (left)},
which is obtained by drawing diagonals from the tensor product of
the Bakhvalov grid $\{\chi(\frac{i}{N})\}_{i=1}^N$ in the $x$-direction \cite{Bak}
and a uniform grid $\{\frac{j}{M}\}_{j=0}^M$ in the $y$-direction with $M=\frac12 N$.
The continuous mesh-generating function $\chi(t)=t$ if $\eps > \frac16$;  otherwise,
$\chi(t)=3\eps \ln\frac1{1 - 2t}$ for $t\in (0, \frac12-3\eps)$ and is linear elsewhere subject to $\chi(1)=1$.
{\color{blue}Furthermore, to test our estimator on a mesh with obtuse triangles and, in particular,
the role of the estimator components $\btauS^J$ in \eqref{tau_def} for $S\in{\mathcal S}^*$,
we distort the initial non-obtuse triangulation by moving some of the nodes upwards/downwards
by $\min\{\mathring{h}_z,\,\frac18 H_z\}$; see Fig.\,\ref{fig_partial_new} (right).}

{
\begin{table}[tbp]
\caption{
 %
 %
 \color{blue}Test problem with $C_u=1$,
 non-obtuse triangulation (see Fig.\,\ref{fig_partial_new}, left).
 %
}\label{table1_new}
\tabcolsep=6pt
\vspace{-0.2cm}
{\small\hfill
\begin{tabular}{r|rrrrrrr}
\hline
\strut\rule{0pt}{9pt}$ N$&
$\varepsilon=1$& $\varepsilon=2^{-5}$&$\varepsilon=2^{-10}$&
$\varepsilon=2^{-15}$& $\varepsilon=2^{-20}$&$\varepsilon=2^{-25}$&$\varepsilon=2^{-30}$\\
\hline
\strut&
\multicolumn{7}{l}{\rule{0pt}{8pt}%
Errors $\vvvert u_h-u \vvvert_{\eps\,;\Omega}$ 
}
\\
\strut\rule{0pt}{11pt}%
%
%
 64&3.203e-2	&5.204e-3	&1.065e-3	&6.734e-4	&6.576e-4	&6.571e-4	&6.571e-4	\\
128&1.602e-2	&2.594e-3	&4.534e-4	&1.797e-4	&1.641e-4	&1.636e-4	&1.636e-4	\\
256&8.011e-3	&1.296e-3	&2.157e-4	&5.533e-5	&4.133e-5	&4.081e-5	&4.080e-5	\\
512&4.006e-3	&6.479e-4	&1.062e-4	&2.130e-5	&1.071e-5	&1.020e-5	&1.019e-5	\\
\hline
\strut&
\multicolumn{7}{l}{\rule{0pt}{8pt}%
\color{blue}${\mathcal S}^*=\emptyset$ in \eqref{tau_def}:~
Estimators
(odd rows) ~\&~
  Effectivity Indices (even rows)
  }
\\
\strut\rule{0pt}{11pt}%
%
%
 64&3.301e-2	&6.994e-3	&1.325e-3	&6.878e-4	&6.581e-4	&6.572e-4	&6.571e-4	\\
&1.031	&1.344	&1.244	&1.021	&1.001	&1.000	&1.000	\\
128&1.647e-2	&2.698e-3	&6.007e-4	&1.928e-4	&1.645e-4	&1.636e-4	&1.636e-4	\\
&1.028	&1.040	&1.325	&1.073	&1.003	&1.000	&1.000	\\
256&8.232e-3	&1.335e-3	&2.928e-4	&6.541e-5	&4.178e-5	&4.083e-5	&4.080e-5	\\
&1.028	&1.030	&1.357	&1.182	&1.011	&1.000	&1.000	\\
512&4.115e-3	&6.668e-4	&1.460e-4	&2.753e-5	&1.115e-5	&1.022e-5	&1.019e-5	\\
&1.027	&1.029	&1.375	&1.292	&1.041	&1.001	&1.000	\\
\hline
\end{tabular}\hfill}
\end{table}
}

{
\begin{table}[tbp]
\caption{
 \color{blue}Test problem with $C_u=1$,
 mesh with obtuse triangles (see Fig.\,\ref{fig_partial_new}, right).
 %
%
}\label{table2_newnew}
\tabcolsep=6pt
\vspace{-0.2cm}
{\small\hfill
\begin{tabular}{r|rrrrrrr}
\hline
\strut\rule{0pt}{9pt}$ N$&
$\varepsilon=1$& $\varepsilon=2^{-5}$&$\varepsilon=2^{-10}$&
$\varepsilon=2^{-15}$& $\varepsilon=2^{-20}$&$\varepsilon=2^{-25}$&$\varepsilon=2^{-30}$\\
\hline
\strut&
\multicolumn{7}{l}{\rule{0pt}{8pt}%
Errors $\vvvert u_h-u \vvvert_{\eps\,;\Omega}$ 
}
\\
\strut\rule{0pt}{11pt}%
%
%
 64&3.334e-2	&5.311e-3	&1.095e-3	&7.218e-4	&7.072e-4	&7.067e-4	&7.067e-4	\\
128&1.669e-2	&2.647e-3	&4.580e-4	&1.913e-4	&1.768e-4	&1.763e-4	&1.763e-4	\\
256&8.352e-3	&1.323e-3	&2.161e-4	&5.774e-5	&4.451e-5	&4.404e-5	&4.402e-5	\\
512&4.177e-3	&6.612e-4	&1.061e-4	&2.170e-5	&1.149e-5	&1.101e-5	&1.100e-5	\\
\hline
\strut&
\multicolumn{7}{l}{\rule{0pt}{8pt}%
\color{blue}${\mathcal S}^*=\emptyset$ in \eqref{tau_def}:~ Estimators
(odd rows) ~\&~
  Effectivity Indices (even rows)
  }
\\
\strut\rule{0pt}{11pt}%
%
%
 64&3.546e-2	&8.155e-3	&1.666e-3	&7.554e-4	&7.083e-4	&7.068e-4	&7.067e-4	\\
&1.064	&1.535	&1.521	&1.047	&1.002	&1.000	&1.000	\\
128&1.772e-2	&3.556e-3	&8.999e-4	&2.370e-4	&1.785e-4	&1.764e-4	&1.763e-4	\\
&1.062	&1.343	&1.965	&1.239	&1.010	&1.000	&1.000	\\
256&8.866e-3	&1.770e-3	&5.948e-4	&1.254e-4	&4.867e-5	&4.417e-5	&4.403e-5	\\
&1.061	&1.339	&2.752	&2.171	&1.093	&1.003	&1.000	\\
512&4.434e-3	&8.842e-4	&3.927e-4	&1.063e-4	&2.169e-5	&1.148e-5	&1.101e-5	\\
&1.062	&1.337	&3.703	&4.896	&1.889	&1.043	&1.001	\\
\hline
\strut&
\multicolumn{7}{l}{\rule{0pt}{8pt}%
\color{blue}${\mathcal S}^*\neq\emptyset$ in \eqref{tau_def}:~
Estimators
(odd rows) ~\&~
  Effectivity Indices (even rows)
  }
\\
\strut\rule{0pt}{11pt}%
%
%
 64&3.546e-2	&8.149e-3	&1.536e-3	&7.466e-4	&7.080e-4	&7.068e-4	&7.067e-4	\\
&1.064	&1.534	&1.402	&1.034	&1.001	&1.000	&1.000	\\
128&1.772e-2	&3.553e-3	&7.089e-4	&2.140e-4	&1.776e-4	&1.763e-4	&1.763e-4	\\
&1.062	&1.342	&1.548	&1.118	&1.005	&1.000	&1.000	\\
256&8.866e-3	&1.770e-3	&3.474e-4	&7.509e-5	&4.532e-5	&4.406e-5	&4.402e-5	\\
&1.061	&1.338	&1.607	&1.300	&1.018	&1.001	&1.000	\\
512&4.434e-3	&8.839e-4	&1.732e-4	&3.238e-5	&1.225e-5	&1.104e-5	&1.100e-5	\\
&1.062	&1.337	&1.633	&1.492	&1.066	&1.002	&1.000	\\
\hline
\end{tabular}\hfill}
\end{table}
}

{
\begin{table}[tbp]
\caption{
 \color{blue}Test problem with $C_u=0$,
 mesh with obtuse triangles (see Fig.\,\ref{fig_partial_new}, right).
%
%
}\label{table3_newnew}
\tabcolsep=6pt
\vspace{-0.2cm}
{\small\hfill
\begin{tabular}{r|rrrrrrr}
\hline
\strut\rule{0pt}{9pt}$ N$&
$\varepsilon=1$& $\varepsilon=2^{-5}$&$\varepsilon=2^{-10}$&
$\varepsilon=2^{-15}$& $\varepsilon=2^{-20}$&$\varepsilon=2^{-25}$&$\varepsilon=2^{-30}$\\
\hline
\strut&
\multicolumn{7}{l}{\rule{0pt}{8pt}%
Errors $\vvvert u_h-u \vvvert_{\eps\,;\Omega}$ 
}
\\
\strut\rule{0pt}{11pt}%
%
 64&5.329e-2	&4.862e-3	&8.425e-4	&1.489e-4	&2.633e-5	&4.655e-6	&8.228e-7	\\
128&2.680e-2	&2.438e-3	&4.222e-4	&7.469e-5	&1.320e-5	&2.334e-6	&4.126e-7	\\
256&1.344e-2	&1.220e-3	&2.110e-4	&3.740e-5	&6.611e-6	&1.169e-6	&2.066e-7	\\
512&6.727e-3	&6.104e-4	&1.051e-4	&1.871e-5	&3.308e-6	&5.848e-7	&1.034e-7	\\
\hline
\strut&
\multicolumn{7}{l}{\rule{0pt}{8pt}%
\color{blue}${\mathcal S}^*=\emptyset$ in \eqref{tau_def}:~
Estimators
(odd rows) ~\&~
  Effectivity Indices (even rows)
  }
\\
\strut\rule{0pt}{11pt}%
%
%
 64&5.729e-2	&7.763e-3	&1.487e-3	&2.632e-4	&4.652e-5	&8.224e-6	&1.454e-6	\\
&1.075	&1.597	&1.765	&1.767	&1.767	&1.767	&1.767	\\
128&2.881e-2	&3.383e-3	&8.710e-4	&1.565e-4	&2.766e-5	&4.891e-6	&8.645e-7	\\
&1.075	&1.388	&2.063	&2.095	&2.095	&2.095	&2.095	\\
256&1.444e-2	&1.689e-3	&5.883e-4	&1.167e-4	&2.063e-5	&3.648e-6	&6.448e-7	\\
&1.075	&1.384	&2.788	&3.121	&3.121	&3.121	&3.121	\\
512&7.231e-3	&8.443e-4	&3.903e-4	&1.055e-4	&1.866e-5	&3.299e-6	&5.832e-7	\\
&1.075	&1.383	&3.713	&5.638	&5.642	&5.642	&5.642	\\
\hline
\strut&
\multicolumn{7}{l}{\rule{0pt}{8pt}%
\color{blue}${\mathcal S}^*\neq\emptyset$ in \eqref{tau_def}:~ Estimators
(odd rows) ~\&~
  Effectivity Indices (even rows)
  }
\\
\strut\rule{0pt}{11pt}%
%
%
 64&5.729e-2	&7.668e-3	&1.370e-3	&2.422e-4	&4.282e-5	&7.569e-6	&1.338e-6	\\
&1.075	&1.577	&1.626	&1.626	&1.626	&1.626	&1.626	\\
128&2.881e-2	&3.330e-3	&6.870e-4	&1.215e-4	&2.148e-5	&3.797e-6	&6.712e-7	\\
&1.075	&1.366	&1.627	&1.627	&1.627	&1.627	&1.627	\\
256&1.444e-2	&1.660e-3	&3.437e-4	&6.086e-5	&1.076e-5	&1.902e-6	&3.362e-7	\\
&1.075	&1.360	&1.629	&1.627	&1.627	&1.627	&1.627	\\
512&7.231e-3	&8.301e-4	&1.715e-4	&3.046e-5	&5.385e-6	&9.519e-7	&1.683e-7	\\
&1.075	&1.360	&1.632	&1.628	&1.628	&1.628	&1.628	\\
\hline
\end{tabular}\hfill}
\end{table}
}

In our numerical experiments, we set ${\mathcal T}_0:=\emptyset$ in \eqref{Tcal_star} and replace $\lesssim $ by $\le$
in \eqref{Tcal_star},\,\eqref{tau_def},
and when dealing with the two cases $h_z\lesssim \eps$ and $h_z\not\lesssim \eps$,
{\color{blue}as well as with $h_T\lesssim \eps$ in \eqref{S_star}.}
Also, we understand $a\ll b$ as $a\le \frac15 b$ for any two quantities $a$ and $b$
(so, for example, \eqref{Tcal_star} becomes ${\mathcal T}^*:=\{T\in {\mathcal T} : h_T\le \frac15 H_T \mbox{~and~}h_{T}\le\eps\}$).

We compute the estimator $\mathcal E$ from  \eqref{error_lemma} with $C_f:=1$ and $\btau$ from \eqref{tau_def},\,\eqref{tau_f_def}.
{\color{blue}For the non-obtuse mesh of Fig.\,\ref{fig_partial_new} (left),
conditions ${\mathcal A}1$ and ${\mathcal A}2$ are satisfied,} so we set ${\mathcal S}^*=\emptyset$.
The component $\btauz$ in \eqref{tau_def} is computed by \eqref{tauz_h_small} combined with \eqref{beta_choice_case1} for $h_z\le\eps$, and, otherwise,
using
\eqref{omega_star_case2},\,\eqref{tau_z_case2},\,\eqref{cal_J} combined with \eqref{tau_n_hat_version_ab}.
Note that instead of explicitly including the components involving $\btau_{z;T}^\pm$ (from \eqref{tau_z_case2}) in $\btau$,
we use \eqref{tau_z_case2_bounds_b} (as well as Remark~\ref{rem_sqrt6}). This somewhat simplifies the computations, but yields a slightly less sharp estimator.
{\color{blue}Similarly, whenever $\phi_z^*\neq\phi_z$ in \eqref{tau_z_case2}, we employ the bounds from \eqref{varphi_zT}.
}
When computing the error and the estimator, we replace $\nabla u$ by its linear Lagrange interpolant, and
$u$ and $f_h$  by their quadratic Lagrange interpolants.%

{\color{blue}When using the mesh with obtuse triangles of Fig.\,\ref{fig_partial_new} (right),
we consider ${\mathcal S}^*$ defined by \eqref{S_star}, and also compare the latter with a simpler choice ${\mathcal S}^*=\emptyset$.
Whenever ${\mathcal S}^*\neq\emptyset$, the estimator involves $\btauS^J$ computed by \eqref{btauS},
while the computation of $\btauz$  employs \eqref{beta_star} and  Remark~\ref{rem_kappa_case2}.

For the test problem with $C_u=1$,}
the errors $\vvvert u_h-u \vvvert_{\eps\,;\Omega}$ are compared with the corresponding estimators $\mathcal E$  in Tables~\ref{table1_new}
and \ref{table2_newnew}.
One observes that the effectivity indices
(computed as the ratio of the estimator to the error) do not exceed {\color{blue}1.633,
as long as ${\mathcal S}^*\neq\emptyset$ for the mesh with obtuse triangles.
%
By contrast, ${\mathcal S}^*=\emptyset$ on the mesh with obtuse triangles 
larger and less stable effectivity indices.
%
But the superiority of the estimator with ${\mathcal S}^*\neq \emptyset$ is particularly evident for the test problem with $C_u=0$
on the mesh with obtuse triangles;
compare the effectivity indices for the two choices of ${\mathcal S}^*$ in Table \ref{table3_newnew}. 
Some additional numerical results are given in Appendix~\ref{sec_app_A}.
}

Overall, for the considered ranges of $\eps$ and $N$,
the aspect ratios of the mesh elements take values
between 2 and 3.6e+8.
Considering these variations,
our estimator ${\mathcal E}$
performs quite well and its effectivity indices do not exceed {\color{blue}1.63 and stabilize} 
 as $\eps \rightarrow 0$ and $N$ increases
{\color{blue}(as long as ${\mathcal S}^*\neq \emptyset$ is used for the mesh with obtuse triangles).
We have also observed that the inclusion of
the estimator components $\btauS^J$ in \eqref{tau_def} for $S\in{\mathcal S}^*\neq\emptyset$, in general,
yields a superior estimator.}
A more comprehensive numerical study of the proposed estimator 
certainly needs to be conducted, and
will be presented elsewhere.


{\color{blue}\section{Lower error bounds. Estimator efficiency}\label{sec_kunert}
Throughout this section, we additionally assume that
 $f(x,y; u)-f(x,y;v) \lesssim |u-v|$, and use the additional notation
 $J_S:=[\pt_\bnu u_h]_S$ and $\omega_S:=T\cup T'$ for any $S=\pt T\cap\pt T'\in{\mathcal S}$
 (with the obvious modification $\omega_S:=T$ for the case $S\subset\pt T\cap\pt\Omega$).

 %

{
\begin{table}[t!]
\caption{Lower error estimators for test problem with $u=\sin(\pi a x)$ and $\eps=1$.
}\label{table_Kunert_lb}
\tabcolsep=7pt
\vspace{-0.2cm}
{\small\hfill
\begin{tabular}{ccrrrrrrr}
\hline
&&\multicolumn{3}{l}{\rule{0pt}{8pt}$\,a=1$}&&\multicolumn{3}{l}{\rule{0pt}{8pt}$\,a=3$}\\
\strut\rule{0pt}{9pt}&&
$N=20$& $N=40$&$N=80$&&$N=20$& $N=40$&$N=80$\\
\hline
\strut&&
\multicolumn{7}{l}{\rule{0pt}{8pt}%
{Errors $\vvvert u_h-u \vvvert_{\eps\,;\Omega}$ 
(odd rows) ~\&~
 $\|h_T (f_h-f_h^I)\|_{\Omega}$ (even rows)}%
  }
\\
\strut\rule{0pt}{11pt}%
%
%
  $M=\;\;\;2N$ &&1.01e-1	&5.04e-2	&2.52e-2	&&9.26e-1	&4.56e-1	&2.27e-1\\
 &&3.87e-4	&4.84e-5	&6.05e-6	&&2.87e-2	&3.59e-3	&4.50e-4\\
 \strut\rule{0pt}{10pt}%
  $M=\;\;\;8N$ &&1.01e-1	&5.04e-2	&2.52e-2	&&9.26e-1	&4.56e-1	&2.27e-1\\
&&1.07e-4	&1.34e-5	&1.68e-6	&&7.95e-3	&9.97e-4	&1.25e-4\\
\strut\rule{0pt}{10pt}%
 $M=\; 32N$ &&1.01e-1	&5.04e-2	&2.52e-2	&&9.26e-1	&4.56e-1	&2.27e-1\\
&&2.70e-5	&3.38e-6	&4.22e-7	&&2.00e-3	&2.51e-4	&3.14e-5\\
\strut\rule{0pt}{10pt}%
$M=\!128N$ &&1.01e-1	&5.04e-2	&2.52e-2	&&9.26e-1	&4.56e-1	&2.27e-1\\
&&6.76e-6	&8.45e-7	&1.06e-7    &&5.01e-4	&6.28e-5	&7.86e-6\\	
\hline
\strut&&
\multicolumn{7}{l}{\rule{0pt}{8pt}%
$\underline{\mathcal E}$ using $\varrho_{S\rm\scriptsize(\cite{Kun01})}$
(odd rows) ~\&~
  Effectivity Indices (even rows)
  }
\\
\strut\rule{0pt}{11pt}%
%
  $M=\;\;\;2N$ &&2.89e-1	&1.45e-1	&7.24e-2    &&2.51e+0	&1.26e+0	&6.33e-1\\	
&&2.87	&2.88	&2.88	&&2.72	&2.78	&2.79\\
\strut\rule{0pt}{10pt}%
  $M=\;\;\;8N$ &&1.32e-1	&6.59e-2	&3.30e-2	&&1.17e+0	&5.86e-1	&2.93e-1\\
&&1.31	&1.31	&1.31	&&1.26	&1.29	&1.29\\
\strut\rule{0pt}{10pt}%
 $M=\; 32N$ &&6.27e-2	&3.14e-2	&1.57e-2	&&5.62e-1	&2.82e-1	&1.41e-1\\
&&0.62	&0.62	&0.62	&&0.61	&0.62	&0.62\\
\strut\rule{0pt}{10pt}%
$M=\!128N$ &&3.10e-2	&1.55e-2	&7.75e-3	&&2.79e-1	&1.39e-1	&6.97e-2\\
&&0.31	&0.31	&0.31   &&0.30	&0.31	&0.31\\
\hline
\strut&&
\multicolumn{7}{l}{\rule{0pt}{8pt}%
$\underline{\mathcal E}$ using $\varrho_{S\rm(\S\ref{ssec_new_lower})}$
(odd rows) ~\&~
  Effectivity Indices (even rows)
  }
\\
\strut\rule{0pt}{11pt}%
%
%
  $M=\;\;\;2N$ &&3.00e-1	&1.50e-1	&7.52e-2	&&2.61e+0	&1.32e+0	&6.59e-1\\
&&2.98	&2.98	&2.98	&&2.82	&2.89	&2.90\\
\strut\rule{0pt}{10pt}%
  $M=\;\;\;8N$ &&2.51e-1	&1.26e-1	&6.28e-2	&&2.25e+0	&1.13e+0	&5.64e-1\\
&&2.49	&2.49	&2.49	&&2.43	&2.47	&2.48\\
\strut\rule{0pt}{10pt}%
 $M=\; 32N$ &&2.47e-1	&1.23e-1	&6.18e-2	&&2.21e+0	&1.11e+0	&5.56e-1\\
&&2.45	&2.45	&2.45	&&2.39	&2.44	&2.45\\
\strut\rule{0pt}{10pt}%
$M=\!128N$ &&2.46e-1	&1.23e-1	&6.17e-2	&&2.21e+0	&1.11e+0	&5.55e-1\\	
&&2.44	&2.45	&2.45   &&2.39	&2.43	&2.45\\
\hline
\end{tabular}\hfill}
\end{table}
}

\subsection{Standard lower error bounds are not sharp. Numerical example}\label{ssec_non_sharp}
Consider a simple test problem (\ref{eq1-1}) with  $\eps=1$, the unique exact solution $u=\sin(\pi a x)$ (for $a=1,3$),  and
$f=u-F(x,y)$ on $\Omega=(0,1)^2$.
We employ the triangulation obtained by drawing diagonals
from the tensor product of the uniform grids $\{\frac{i}{N}\}_{i=0}^N$ and $\{\frac{j}{M}\}_{j=0}^M$ respectively in the $x$- and $y$-directions
(with all diagonals having the same orientation).
The standard lumped-mass quadrature, i.e.
${\mathcal T}_*:=\emptyset$ in \eqref{Q_T},
will be used in numerical experiments in this section (while the anisotropic quadrature with ${\mathcal T}_*:={\mathcal T}$ produces very similar results on this mesh).

For this problem, we compare two lower error estimators:
obtained using the standard bubble function approach \cite{Kun01} (see also Lemma~\ref{lem kun_lower} in \S\ref{ssec_kun_}) and  the one obtained in \S\ref{ssec_new_lower}
(combine Theorem~\ref{theo_lower} with Lemma~\ref{lem kun_lower}). They can be described by
\begin{subequations}\label{two_lower_main}
\beq\label{two_lower}
\underline{\mathcal E}:=
\Bigl\{\sum_{S\in{\mathcal S}\backslash\pt\Omega} \varrho_S\, J_S^2+\|h_T f_h^I\|_{\Omega}^2\Bigr\}^{1/2}
\lesssim
\vvvert u_h-u \vvvert_{\eps\,;\Omega}+ \|h_T (f_h-f_h^I)\|_{\Omega},
\eeq
where  the weight $\varrho_S$
for $S\in{\mathcal S}\backslash\pt\Omega$
is defined by
\beq\label{rho_def}
\varrho_S=\left\{\begin{array}{lll}
\varrho_{S\rm\scriptsize(\cite{Kun01})}&\!\!\!\!\!\!{}=
|S|\displaystyle\min_{T\subset\omega_S}\{h_{T}\},
&\mbox{\cite{Kun01} using bubble functions (also \S\ref{ssec_kun_})},
\\[0.4cm]
\varrho_{S\rm(\S\ref{ssec_new_lower})}&\!\!\!\!\!\!{}=
\frac12|\omega_S|,
& \mbox{see  Theorem~\ref{theo_lower} in \S\ref{ssec_new_lower}}.
\end{array}
\right.
\eeq
\end{subequations}
(To be more precise, when $\varrho_{S\rm(\S\ref{ssec_new_lower})}$ is used, the term $\|h_T (f_h-f_h^I)\|_{\Omega}$ in the right-hand side of \eqref{two_lower} should be replaced by a larger
$\|H_T\,{\rm osc }(f_h\,;T)\|_{\Omega}$; see \S\ref{ssec_new_lower} for details.)

To address whether the left-hand side $\underline{\mathcal E}$ in \eqref{two_lower} is sharp,
the errors $\vvvert u_h-u \vvvert_{\eps\,;\Omega}$ (as well as $\|h_T (f_h-f_h^I)\|_{\Omega}$) are compared with $\underline{\mathcal E}$
in Table~\ref{table_Kunert_lb}.
Clearly, the standard lower estimator using the weights $\varrho_{S\rm\scriptsize(\cite{Kun01})}$ is not sharp.
Not only its effectivity indices strongly depend on the ratio $M/N$, but, perhaps more alarmingly, $\underline{\mathcal E}$ converges to zero as $M/N$ increases, i.e. the mesh is anisotropically
refined in the wrong direction (while the error remains almost independent of $M/N$).
By contrast,  the estimator of \S\ref{ssec_new_lower} performs quite well, with the effectivity indices stabilizing. 

When comparing the two estimators, note that $\varrho_{S\rm\scriptsize(\cite{Kun01})}\sim \varrho_{S\rm(\S\ref{ssec_new_lower})}$ when $|S|\sim{\rm diam}\,\omega_S$, however,
$\varrho_{S\rm\scriptsize(\cite{Kun01})}\ll \varrho_{S\rm(\S\ref{ssec_new_lower})}$ when $|S|\ll {\rm diam}\,\omega_S$, i.e. for short edges.
Hence, our numerical experiments suggest that it is the short-edge jump residual terms in the standard lower estimator that are not sharp.
We shall address this theoretically in \S\ref{ssec_new_lower}.

\subsection{Lower error bounds using the standard bubble approach}\label{ssec_kun_}
Here, for completeness, we prove a version of the lower error bounds from \cite[Theorem~4.3]{Kun01} for the semilinear case
(similar, but less sharp bounds can also be found in \cite{Kunert2000,KunVer00}).
\smallskip

\begin{lemma}\label{lem kun_lower}
For  a solution $u$ of \eqref{eq1-1} and any $u_h\in S_h$, one has
\begin{subequations}\label{lower_f_J}
\begin{align}\label{lower_f}
\min\{1,\,h_T\eps^{-1}\} \|f_h^I\|_T&\lesssim \underbrace{\vvvert u_h-u \vvvert_{\eps\,;{T}}+ \min\{1,\,h_T\eps^{-1}\}\|f_h-f_h^I\|_T}_{{}=:{\mathcal Y}_T}
\hspace{-0.5cm}&\forall T\in{\mathcal T},
\\[0.2cm]\label{lower_J}
|S|^{1/2}\,\,\bigl|\eps J_S\bigr|&\lesssim\sum_{T\in\omega_S}{\mathcal Y}_T\,\min\{\eps,\,h_T\}^{-1/2}
&\forall S
\in{\mathcal S}\backslash\pt \Omega.
\end{align}
\end{subequations}
\end{lemma}

\begin{corollary}
If $|\omega_z|\sim|T|$ for any $T\subset\omega_z$, then 
$$
\min\{1,\,\eps h_z^{-1}\}^{1/2}\,\,\|\eps \widetilde J_z\|_{\omega_z}+\min\{1,\,h_z\eps^{-1}\} \|f_h^I\|_{\omega_z}\lesssim
\sum_{T\subset\omega_z}{\mathcal Y}_T\,,
\vspace{-0.4cm}
$$
\beq\label{mathring_J}
\mbox{where}\qquad\quad
\widetilde J_z :=
\underbrace{\max_{S\in\gamma_z: |S|\sim H_z} \bigl|J_S\bigr|}_{{}=:\widehat J_z}
\,\,+\,\,
\{h_z H_z^{-1}\}^{1/2}
\underbrace{\max_{S\in\gamma_z: |S|\ll H_z} \bigl|J_S\bigr|}_{{}=:\mathring{J}_z}.
\hphantom{\mbox{where}\qquad\quad}
\vspace{-0.2cm}
\eeq
\end{corollary}

\begin{remark}[Estimator efficiency under an adaptive-mesh-alignment condition]\label{rem_eff}
It appears that the above result is as sharp as one can get  using the bubble function approach,
while in \S\ref{ssec_non_sharp} we have seen that the short-edge jump residual terms are not sharp in such bounds.
On the other hand, the interpolation error bounds suggest that
a reasonably optimal and  correctly-aligned mesh may be expected to satisfy $\mathring{J}_z\lesssim \widehat J_z$.
Consequently, it appears reasonable to impose a mild version of this condition:
\beq\label{eff_cond}
\eps\mathring{J}_z\lesssim  \eps\widehat J_z+\min\{1,\,\eps h_z^{-1}\}^{-1/2}\min\{1,\,h_z\eps^{-1}\}\|f_h^I\|_{L_\infty(\omega_z)}.
\eeq
when constructing a mesh adaptively. Clearly,
if both \eqref{eff_cond} and
 the condition of the above corollary  are satisfied for all $z\in\mathcal N$, then the upper error estimator from \eqref{upper_bound}
 is efficient.
\end{remark}
\medskip

{\it Proof of Lemma~\ref{lem kun_lower}.}
(i) On any $T\in\mathcal T$, consider $w:=f_h^I\,\phi_1\phi_2\phi_3$, where $\{\phi_i\}_{i=1}^3$  are the standard hat functions associated with the three vertices of $T$.
Now, a standard calculation yields $\|f_h^I\|_T^2\sim \langle f_h^I, w\rangle$. So, using $f_h^I=-\eps^2\triangle u_h +f_h^I$ and \eqref{eq1-1} yields
$\|f_h^I\|_T^2\sim \eps^2\langle \nabla(u_h-u), \nabla w\rangle+\langle f_h^I-f(\cdot;u), w\rangle$.
Next, invoking $\|\nabla w\|_T\lesssim h_T^{-1}\|w\|_T$, one arrives at
$$
\|f_h^I\|_T^2\lesssim \Bigl(
(\eps h_T^{-1}+1)
\vvvert u_h-u \vvvert_{\eps\,;{T}}+\|f_h-f_h^I\|_T \Bigr)\,
\|w\|_T
\,.
$$
Here we also used $|f_h^I-f(\cdot;u)|\lesssim |u_h-u|+|f_h-f_h^I|$.
The desired result \eqref{lower_f} follows in view of $\|w\|_T\lesssim \|f_h^I\|_T$ and $\eps h_T^{-1}+1\sim \min\{1,\,h_T\eps^{-1}\}^{-1}$.

(ii) For each of the two triangles $T\subset\omega_S$, introduce a triangle $\widetilde T\subseteq T$ with an edge $S$ such that $|\widetilde T|\sim \min\{\eps,\,h_T\}|S|$.
Next, set $w:=J_S\, \widetilde \phi'\widetilde \phi''$, where $\widetilde \phi'$ and $\widetilde \phi''$
are the hat functions on the triangulation $\{\widetilde T\}_{T\subset\omega_S}$
associated with the two end points of~$S$ (with $w:=0$ on each $T\backslash\widetilde T$ for $T\subset\omega_S$).
A standard calculation using $\triangle u_h =0$ in $T\subset\omega_S$ and \eqref{eq1-1}, yields
$$
|S|\,(\eps J_S)^2\sim \eps^2 \int_S w\,[\pt_\bnu u_h]_S =\eps^2\langle \nabla u_h, \nabla w \rangle
=\eps^2\langle \nabla (u_h-u), \nabla w \rangle - \langle f(\cdot;u), w \rangle.
$$
Next, invoking $\|\nabla w\|_{T}\lesssim  \min\{\eps,\,h_{T}\}^{-1}\|w\|_{T}$ for any $T\subset\omega_S$, we arrive at
$$
|S|\,(\eps J_S)^2\lesssim\sum_{T\in\omega_S}
\underbrace{\Bigl( \min\{1,\,h_T\eps^{-1}\}^{-1}\vvvert u_h-u \vvvert_{\eps\,;T}+\|f_h\|_{T} \Bigr)}
_{{}\lesssim{\mathcal Y}_{T}\min\{1,\,h_T\eps^{-1}\}^{-1}\;\rm by\;\eqref{lower_f}}
\underbrace{\|w\|_{T}}_{\sim \min\{\eps,\,h_T\}^{1/2}|S|^{1/2}|J_S|}\hspace{-1cm}.\hspace{1cm}
$$
In view of
$\min\{1,\,h_T\eps^{-1}\}^{-1}\,\min\{\eps,\,h_T\}^{1/2}\,\eps^{-1}\sim \min\{\eps,\,h_T\}^{-1/2}$,
one gets \eqref{lower_J}.
\endproof\medskip

\subsection{New lower error bound with sharp short-edge jump residual terms}\label{ssec_new_lower}
Throughout this section,
we make additional restrictions on the anisotropic mesh as follows.
Let  $\Omega:=(0,1)^2$, and $\{x_i\}_{i=0}^n$ be an arbitrary mesh in the $x$ direction on the interval $(0,1)$.
Then, let each $T\in\mathcal T$, for some $i$,
(i) have the shortest edge on the line $x=x_i$;
(ii) have a vertex on the line $x=x_{i+1}$ or $x=x_{i-1}$
(see Fig.\,\ref{fig_partial}).
Also, let ${\mathcal N}={\mathcal N}_{\rm ani}$, i.e.
each $z\in\mathcal N$ be an anisotropic node in the sense of (\ref{ani_node})
and satisfy ${\mathcal A}1_{\rm ani}$.
The above conditions essentially imply that all mesh elements are anisotropic and aligned in the $x$-direction.
The main result of this section is the following.
\medskip

\begin{figure}[!b]
\hfill\color{black}
\begin{tikzpicture}[scale=0.25]
\draw[ultra thick 
] (0,0) -- (21,0) ;
\path[draw,help lines]  (19,-0.2)node[below] {$x_{i-1}$}--(19,9.5);
\path[draw,help lines]  (9.5,-0.2)node[below] {$x_i$}--(9.5,9.5);
\path[draw,help lines]  (2,-0.2)node[below] {$x_{i+1}$}--(2,9.5);
\path[draw]  (19,0)--(19,9.9);
\path[draw]  (9.5,0)--(9.5,9.9);
\path[draw]  (2,0)--(2,9.9);
\path[draw]  (2,1)--(9.5,1.8)--(19,0)--(9.5,0)--cycle;
\path[draw]  (9.5,1.8)--(19,1.9);
\path[draw]  (2,3)--(9.5,4.5)--(19,4)--(9.5,1.8)--cycle;
\path[draw]  (2,3)--(9.5,3)--(19,4);
\path[draw]  (2,5)--(9.5,6.4)--(19,7.5)--(9.5,4.5)--cycle;
\path[draw]  (2,6.7)--(9.5,6.4);
\path[draw]  (19,5.9)--(9.5,4.5);
\path[draw]  (2,8.1)--(9.5,8.2)--(19,7.5)--(9.5,6.4)--cycle;
%
\path[draw]  (2,9.5)--(9.5,8.2)--(19,9.4);
\end{tikzpicture}
\hfill~
\vspace{-0.2cm}
\caption{Partially structured anisotropic mesh considered in \S\ref{ssec_new_lower}.}
\label{fig_partial}
\end{figure}
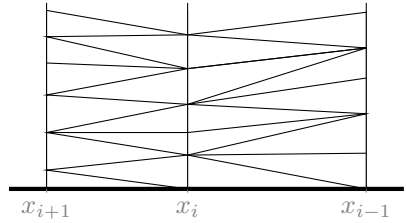

\begin{theorem}[Short-edge jump residual terms]\label{theo_lower}
Let   $u$ and $u_h$ respectively satisfy \eqref{eq1-1} and \eqref{eq1-2}, and
$\Omega_i:=( x_{i- 1}, x_{i+ 1})\times(0,1)$.
If either
no no quadrature is used in $\Omega_i$ (i.e. \eqref{eq1-2} involves $\langle f_h,v_h\rangle_h=\langle f_h,v_h\rangle$ $\forall v_h\in S_h$ with support in $\Omega_i$),
or $\langle \cdot,\cdot\rangle_h$ is defined by \eqref{Q_T}--\eqref{f_bar} with
either $\mathcal T\cap \Omega_i\subset \mathcal T^*$, or $\Omega_i\cap \mathcal T^*=\emptyset$, then
\beq\label{lower_particular}
\sum_{S\in{\mathcal S}\cap\{x=x_i\}}\!\!\!
\min\{\eps|S|,|\omega_S|\}\,
\bigl(\eps{J}_S\bigr)^2\lesssim
\underbrace{\vvvert u_h-u \vvvert^2_{\eps\,;{\Omega_i}}
+\|\lambda_T\,{\rm osc }(f_h\,;T)\|^2_{\Omega_i}}_{{}=\,\sum_{T\subset\Omega_i}{\mathcal Y}_T^2\,\,=:\,\,{\mathcal Y}_i^2}\,.
\eeq
\end{theorem}

To prove this theorem, we shall use an auxiliary result.
\medskip

\begin{lemma}\label{lem_osc_v}
(i) If $\gamma_z\cap\{x=x_i\}$ is formed by exactly two edges $S^-$ and $S^+$,
then
\beq\label{osc_v_lem}
\bigl|J_{{S}^+} - J_{{S}^-}\bigr| \lesssim h_z H_z^{-1}\!\!\sum_{S\in\gamma_z\backslash\{x=x_i\}}\!\!\! |J_S|.
\vspace{-0.1cm}
\eeq
(ii) If $\gamma_z\cap\{x=x_i\}$ is formed by a single edge $S^+\!$, then
$ J_{{S}^-}\!$ in  \eqref{osc_v_lem} is replaced by~$0$.%
\end{lemma}
\smallskip

\begin{proof}
(i) Note that in this case $z\not\in\pt\Omega$.
Using the notation $\{T_i\}$ of \S\ref{sec_A1} (see Fig.\,\ref{fig_tau_z_case1}, centre), let
$\llbracket \nabla u_h\rrbracket_{\pt T_{i-1}\cap\pt T_i}:=\nabla u_h\bigr|_{\pt T_i}-\nabla u_h\bigr|_{\pt T_{i-1}}$.
Then
$\sum_{S\in\gamma_z} \llbracket \nabla u_h\rrbracket_{S}=0$.
Multiplying this relation by the unit vector ${\mathbf i}_{x}$ in the $x$-direction, and noting that $\llbracket \nabla u_h\rrbracket_{{S}^\pm}\cdot{\mathbf i}_{x}=\pm J_{{S}^\pm}$,
one gets the desired assertion.
We also use the observation that
for $S\in\gamma_z\backslash\{x=x_i\}$, one has
$|\llbracket \nabla u_h\rrbracket_S\cdot{\mathbf i}_{x}|\sim |J_S\, \bnu_S\cdot{\mathbf i}_{x}|$, where $\bnu_S$ is a unit vector normal to $S$, where,
in view of ${\mathcal A}1_{\rm ani}$, one has
$|\bnu_S\cdot{\mathbf i}_{x}|\lesssim h_z H_z^{-1}$.

(ii) Now $z\in\pt\Omega$, so extend $u_h$ to $\R^2\backslash\Omega$ by $0$ and
imitate the above proof with the modification that now
$\sum_{S\in{\mathcal S}_z} \llbracket \nabla u_h\rrbracket_{S}=0$.
When dealing with the two edges on $\pt\Omega$,
note that for $S\in{\mathcal S}_z\cap\pt \Omega$,
one gets $\bnu_S\cdot{\mathbf i}_{x}=0$.
\end{proof}
\medskip

{\it Proof of Theorem~\ref{theo_lower}.}
%
Set $H:=x_{i+1}-x_{i-1}$,
 and $\theta:=\min\bigl\{\eps H^{-1},\,\frac12\bigr\}$, and then
$\widetilde x_{i\pm 1}:=x_i\pm \theta|x_{i\pm1}-x_i|$ and $\widetilde \Omega_i:=(\widetilde x_{i- 1},\widetilde x_{i+ 1})\times(0,1)$
(so $ \Omega_i$ is a rectangular domain, at least, twice as narrow as $\Omega_i$).
Furthermore,
 define a triangulation $\widetilde{\mathcal T}_i$ on $\widetilde \Omega_i$ by dividing each trapezoid in the partition ${\mathcal T}\cap\widetilde \Omega_i$ into two triangles.

Now, define $v\in C(\bar \Omega)$ with support in $\widetilde \Omega_i$ (so $v=0$ on $\pt\widetilde \Omega_i$) using the standard piecewise-linear interpolation on $\widetilde{\mathcal T}_i$.
Its node values in the interior of $\widetilde \Omega_i$ are defined by
$v(z):={J}_S$  for any $z\in\mathcal N$ on $\{x=x_i\}\backslash\pt\Omega$,
 where $S\in\gamma_z\cap\{x=x_i\}$ is any vertical short edge originating at $z$.
(For definiteness, let $S$ connect $z$ with the node above it.)

Also, let $v_h\in S_h$ be the piecewise-linear interpolant of $v$ on the original triangulation ${\mathcal T}$ (then $v\in C(\bar \Omega)$ has support in $\Omega_i$),
and $w:=v-\theta v_h$.
 Now, a standard calculation yields
\begin{align}\notag
\underbrace{\eps^2\langle\nabla (u_h-u),{}\nabla v\rangle+\langle \widehat f_h-f(\cdot;u),v\rangle}_{{}=:\psi_1}&{}=
\eps^2 \langle\nabla u_h,\nabla v\rangle+\langle \widehat f_h,v\rangle\,,\\
&\hspace{-1.5cm}
{}=
\eps^2 \langle \pt_x u_h ,\pt_x w\rangle+
\underbrace{\eps^2\langle \pt_y u_h ,\pt_y w\rangle
+\langle \widehat f_h,v\rangle-\theta\langle \widehat f_h,v_h\rangle_h
}_{{}=:\psi_2}.
\label{standard_calc}
\end{align}
Here we used a function $\widehat f_h\approx f_h$, which will be specified later subject to the condition
$\|\lambda_T(\widehat f_h- f_h)\|_{\Omega_i}\lesssim \|\lambda_T\,{\rm osc }(f_h\,;T)\|_{\Omega_i}\le{\mathcal Y}_i$.

With $\bnu_x:=(\bnu\cdot{\mathbf i}_x){\mathbf i}_x$ (which is the standard vector projection of the outward normal vector $\bnu$ onto ${\mathbf i}_x$),
one gets
\begin{align*}
\langle \pt_x u_h ,\pt_x w\rangle&{}=\sum_{S\subset{\mathcal S}\cap \Omega_i}\int_S[\nabla u_h\cdot \bnu_x] w
=\sum_{S\subset{\mathcal S}\cap\{x=x_i\}}\int_S[\nabla u_h\cdot \bnu_x] w,
\end{align*}
where for $S\subset{\mathcal S}\cap\Omega_i\backslash\{x=x_i\}$, we used
$\int_S w=\int_S v-\theta\int_S v_h=0$ (as each of $v$ and $v_h$ is linear on its support on $S$, and $v=v_h$ on $\{x=x_i\}$).
Next, note that for $S\subset{\mathcal S}\cap\{x=x_i\}$, one has
$|S|\sim H^{-1}|\omega_S|$, while
 $[\nabla u_h\cdot \bnu_x]={J}_S$
and $w=(1-\theta)v$ with $v\ge {J}_S-{\rm osc}(v\,;S)$ (as $v=J_S$ at one of the end points of $S$), so
$$
\langle \pt_x u_h ,\pt_x w\rangle\ge(1-\theta)H^{-1} \!\!\!
\sum_{S\subset{\mathcal S}\cap\{x=x_i\}} \!\!\!|\omega_S|\,{J}_S \bigl\{{J}_S-{\rm osc}(v\,;S)\bigr\}.
$$
Combining the latter with \eqref{standard_calc} multiplied by $\theta H$, and noting that
$1-\theta\ge\frac12$,  one now gets
\beq\label{main100}
\sum_{S\subset{\mathcal S}\cap\{x=x_i\}}\!\!\! \theta|\omega_S|\,(\eps {J}_S)^2 \lesssim\!\!\!
\sum_{S\subset{\mathcal S}\cap\{x=x_i\}}\!\!\! \theta|\omega_S|\,\bigl(\eps\,{\rm osc}(v\,;S)\bigr)^2
+(\theta H) |\psi_1-\psi_2|\,.
\eeq

We claim
that, to complete the proof, it suffices to get a somewhat similar bound: 

\beq\label{suffices10}
\sum_{S\subset{\mathcal S}\cap\{x=x_i\}}\!\!\! \theta|\omega_S|\,(\eps {J}_S)^2 \lesssim
{\mathcal Y}_i^2
+\!\!\!
\sum_{S\subset{\mathcal S}\cap\{x=x_i\}}\!\!\! \theta|\omega_S|\,\bigl(\eps H|S|^{-1}\,{\rm osc}(v\,;S)\bigr)^2.
\eeq
Indeed, this implies \eqref{lower_particular}, as here in the left-hand side, $\theta|\omega_S|\sim\min\{\eps |S|,|\omega_S|\}$. Furthermore, using Lemma~\ref{lem_osc_v}
to estimate ${\rm osc}(v\,;S)$, the sum in the right-hand side of~\eqref{suffices10} is bounded by
$
\sum_{S\subset{\mathcal S}\cap\Omega_i\backslash\{x=x_i\}} \theta|\omega_S|\,(\eps J_S)^2\lesssim \sum_{T\subset\Omega_i}{\mathcal Y}_T^2\sim{\mathcal Y}_i^2$.
The latter assertion follows from \eqref{lower_J} in view of $\theta|\omega_S|\sim \min\{\eps h_T,|\omega_S|\} \lesssim\min\{\eps |S|,h_T|S|\}$ for any $T\subset\omega_S$.
So it remains to derive \eqref{suffices10} from \eqref{main100}.

For $\psi_1$, defined in \eqref{standard_calc}, in view of $|f_h-f(\cdot; u)|\lesssim |u_h-u|$ and $\|\lambda_T(\widehat f_h- f_h)\|_{\Omega_i}\lesssim{\mathcal Y}_i$, one has
\begin{subequations}\label{psi1_bounds}
\beq
|\psi_1|
\lesssim
{\mathcal Y}_i\,
\Bigl\{ \eps\|\nabla v \|_{\widetilde\Omega_i}+\|\lambda_T^{-1} v \|_{\widetilde\Omega_i} \Bigr\}.
\eeq
Here, recalling the definition of $v$, note that
$\pt_y v=0$ in any triangle in $\widetilde{\mathcal T}_i$ with a single vertex on $\{x=x_i\}$,
while $\pt_y v=\pm |S|^{-1} {\rm osc}(v\,;S)$ and
$|\widetilde T|\sim\theta|\omega_S|$
for any triangle $\widetilde T\in\widetilde{\mathcal T}_i$ sharing an edge $S$ with $\{x=x_i\}$,
so
\beq\label{pt_y_v}
\| \eps\, \pt_y v \|^2_{\widetilde\Omega_i}\lesssim \sum_{S\subset{\mathcal S}\cap\{x=x_i\}}\!\!\!
\theta|\omega_S|\,\bigl(\eps |S|^{-1} {\rm osc}(v\,;S)\bigr)^2.
%
\eeq
Furthermore,  any triangle $\widetilde T\in\widetilde{\mathcal T}_i$ touches an edge $S\subset\{x=x_i\}$ such that
$|\eps\,\pt_x v|\lesssim \eps(\theta H)^{-1}\max_{\widetilde T}|v|=\eps(\theta H)^{-1}J_S$, while $\lambda_T^{-1}\sim \eps (\theta H)^{-1}$ implies a similar bound
for $|\lambda_T^{-1}v|$. Combining these observations with $|\widetilde T|\sim\theta|\omega_S|$ yields
\beq\label{pt_v_x}
\| \eps\, \pt_x v \|^2_{\widetilde\Omega_i}
+\|\lambda_T^{-1} v\|^2_{\widetilde\Omega_i}\lesssim \!\!\!\sum_{S\subset{\mathcal S}\cap\{x=x_i\}}\!\!\!
\theta|\omega_S|\,\bigl(\eps(\theta H)^{-1}J_S\bigr)^2
=(\theta H)^{-2}\!\!\!\!\sum_{S\subset{\mathcal S}\cap\{x=x_i\}}\!\!\!
 \theta|\omega_S|\,(\eps J_S)^2.
\eeq
\end{subequations}

To estimate $\psi_2$ (defined in \eqref{standard_calc}),
set $\widehat f_h(x,y):=f^I_h(x_i,y)$ and $\widehat v(x,y):=v(x_i,y)$ in $\Omega_i$.
Note that
\beq\label{theta_rels}
\langle \pt_y u_h ,\pt_y v\rangle=\theta\,\langle \pt_y u_h ,\pt_y v_h\rangle,
\qquad
\int_{\widetilde\Omega_i} \!\widehat f_h \widehat v\, \widetilde\varphi_i(x)=\theta\,
\int_{\Omega_i} \!\widehat f_h \widehat v\, \varphi_i(x),
\eeq
where $\widetilde\varphi_i(x)$ and $\varphi_i(x)$ are the standard one-dimensional hat functions on the intervals
$(\widetilde x_{i-1},\widetilde x_{i+1})$ and $(x_{i-1},x_{i+1})$, respectively, with $\widetilde\varphi_i(x_i)=\varphi_i(x_i)=1$.
For the first relation in \eqref{theta_rels}, we relied on the observations made on $\pt_y v$ when obtaining \eqref{pt_y_v},
as well as similar properties of $v_h$.

First, consider the case
of no quadrature used in $\Omega_i$, i.e.  $\langle \widehat f_h , v\rangle_h=\langle \widehat f_h , v\rangle$. Then
\beq\label{psi2_rel}
\psi_2 = \int_{\widetilde\Omega_i} \!\widehat f_h \bigl(v-\widehat v\, \widetilde\varphi_i(x)\bigr)
-\theta \int_{\Omega_i} \!\widehat f_h \bigl(v_h-\widehat v\, \varphi_i(x)\bigr)
.
\eeq
From this one can show (we shall comment on this below) that
\beq\label{psi2_bound}
|\psi_2|\lesssim
\underbrace{\bigl\|\min\{1,\,h_T\eps^{-1}\}\,f_h^I\bigr\|_{\Omega_i}}_{{}\lesssim{\mathcal Y}_i{\rm\;by\;\eqref{lower_f}}}
\,\Bigl\{\sum_{S\in{\mathcal S}\cap \{x=x_i\}}\!\!\! \theta|\omega_S|\, \bigl(
\underbrace{\min\{1,\,|S|\eps^{-1}\}^{-1}}_{{}\lesssim \eps|S|^{-1}\theta^{-1}}
\,{\rm osc}(v\,;S)\bigr)^2\Bigr\}^{1/2}.
\eeq
Now, combining \eqref{psi1_bounds} and \eqref{psi2_bound} with \eqref{main100}
one arrives at the desired assertion \eqref{suffices10}.

To complete the proof, we still need to show that
that \eqref{psi2_bound} follows from \eqref{psi2_rel}, as well as $\|\lambda_T(\widehat f_h- f_h)\|_{\Omega_i}\lesssim \|\lambda_T\,{\rm osc }(f_h\,;T)\|_{\Omega_i}$.
For each $T\subset\Omega_i$, introduce the minimal rectangle $R_T=(x_{i-1},x_{i+1})\times(y_T^-,y_T^+)$  containing $T$
(i.e. $(y_T^-,y_T^+)$ is the range of $y$ values within $T$).
Note that, crucially, by condition~${\mathcal A}1_{\rm ani}$
there is $K\lesssim 1$ such that $|R_T|\sim|T|$ and $R_T\subset \omega_T^{(K)}\cap\Omega_i$,
with the notation
$\omega_T^{(j+1)}$ 
for the patch of elements in/touching $\omega_T^{(j)}$ and $\omega_T^{(0)}:=T$.
Now,
$|v- \widehat v\,\widetilde \varphi_i |\le {\rm osc }(v\,;R_T\cap\{x=x_i\})$ for any $T\subset\Omega_i$,
so
\eqref{psi2_rel} implies a version of \eqref{psi2_bound} with $f_h^I$ replaced by $\widehat f_h$, and ${\rm osc }(v\,;S)$ replaced by ${\rm osc }(v\,;R_T\cap\{x=x_i\})$.
As $h_{T'}\sim h_T$ and $H_{T'}\sim H_T$ for any triangle $T'\cap R_T\neq\emptyset$, \eqref{psi2_bound} follows.
Similarly,
$|\widehat f_h- f_h|\le {\rm osc }(f_h^I\,;R_T)$ for any $T\subset\Omega_i$ implies
$\|\lambda_T(\widehat f_h- f_h)\|_{\Omega_i}\lesssim \|\lambda_T\,{\rm osc }(f_h\,;T)\|_{\Omega_i}$.

Finally note that \eqref{psi2_rel} is valid only if no quadrature is used in $\Omega_i$.
Otherwise, the estimation of $\psi_2$ needs to be slightly adjusted.
For the case
${\mathcal T}\cap\Omega_i\subset{\mathcal T}^*$,
tweak the definition of $\widehat f_h$ to $\widehat f_h(x,y):=\widebar f_i(y)$, where $\widebar f_i(y)$ is a one-dimensional piecewise-constant interpolant of $f_h$ on $\{x=x_i\}\cap\Omega$
such that it is constant on each edge $S\subset\{x=x_i\}$. With this modification,
$\int_{\Omega_i} \!\widehat f_h \widehat v\, \widetilde\varphi_i(x)=\langle \widehat f_h , v\rangle_h$,
so the second term in \eqref{psi2_rel} vanishes, while all other arguments apply.
For the case $\Omega_i\cap \mathcal T^*=\emptyset$,
one has $\langle \widehat f_h , v\rangle_h=\int_{\Omega} (\widehat f_h , v)^I$, so the bound   \eqref{psi2_bound} on $\psi_2$
 will additionally include $\|\lambda_T\,{\rm osc }(f_h^I\,;T)\|_{\Omega_i}\|\lambda_T^{-1}\,v\|_{\widetilde \Omega_i}$,
so \eqref{pt_v_x} is employed again  for this additional term.
\endproof
\medskip

\begin{remark}\label{rem_tb_justified}
Combing the lower error bounds \eqref{lower_f_J} and \eqref{lower_particular} and comparing the resulting lower bound with the upper error bound \eqref{upper_bound},
one concludes that for the estimator to be efficient, the term
$\min\{1,\,\eps h_z^{-1}\} \|\eps J_z\|^2_{\omega_z}$  should be replaced by
$\sum_{S\in\gamma_z}\min\{\eps|S|,|\omega_z|\}\,(\eps{J}_S)^2$
in \eqref{upper_bound}, and, equivalently, in \eqref{tau_z_main}.
When $h_z\gtrsim \eps$, this improvement follows from the first relation in \eqref{theo_bound}.
Otherwise,
if $H_z\lesssim  \eps$, this follows from $ |\omega_z|\lesssim \eps|S|$.
For the remaining case $h_z\lesssim\eps\lesssim H_z$, assuming $z\in{\mathcal N}_{\rm ani}$ under condition ${\mathcal A}1_{\rm ani}$, this
sharper upper bound
can be shown for a slightly more intricate version of $\btauz$, defined as follows.
Using the notation of \eqref{tau_z} and \eqref{tau_z_case2_bounds}
(see Fig.\,\ref{fig_tau_z_case1} and Fig.\,\ref{fig_local_kappa} (left); also assume that $\btau_{z;T}^\pm\in L_2(\Omega)$ has support on $T$), set
$$
\btauz:=\sum_{i=1,m_z} {\textstyle\frac12} J_{S_i^-}\bigl(\btau_{z;T_{i-1}}^+ +\btau_{z;T_{i}}^-\bigr)
+\phi_z\sum_{i=1}^{N_z}\bigl(\alpha_i \bnu_i+\beta_i d_i^{-1}\bmu_i\bigr)\,\mathbbm{1}_{(x,y)\in T_i}\,,
$$
where $\{\alpha_i\}$ and $\{\beta_i\}$
are chosen to minimize  \eqref{optima}
subject to the constraints \eqref{tau_n_hat}, in which $ \bigl[\partial_\bnu u_h\bigr]_{\pt T_{i-1}\cap \pt T_i}$ for $i=1,m_z$ are replaced by $0$;
see Appendix~\ref{sec_app_B}.

It appears, however, that in most practical situations, this modification of $\btauz$ will not improve the estimator, as the short-edge jump residual terms
in the upper error estimator are expected to be dominated by the other terms (as discussed in Remark~\ref{rem_eff}).
\end{remark}}



\appendix
\section{Additional numerical results}\label{sec_app_A}
We again consider the test problem and the meshes from \S\ref{ssec_numer}, but now look at two components of the the error in the energy norm. It is reasonable to assume that
\beq\label{er_decomp}
\vvvert u_h-u \vvvert_{\eps\,;\Omega}\sim
\bigl\{\eps^2\|\nabla u_h-(\nabla u)^I\|^2_{\Omega}+\| u_h-u^I\|^2_{\Omega}\bigr\}^{1/2}
+\| u-u^I\|_{\Omega}\,.
\eeq


{
\begin{table}[t!]
\caption{
{\color{blue}Test problem with $C_u=1$,
 non-obtuse triangulation (see Fig.\,\ref{fig_partial_new}, left):%
 }
  error component $\bigl\{\eps^2\|\nabla u_h-(\nabla u)^I\|^2_{2\,;\Omega}+\| u_h-u^I\|^2_{2\,;\Omega}\bigr\}^{1/2}$ from \eqref{er_decomp} compared with
  ${\mathcal E}\bigr|_{f_h:=f_h^I}$\,. 
}\label{table2_new}
\tabcolsep=6pt
\vspace{-0.2cm}
{\small\hfill
\begin{tabular}{r|rrrrrrr}
\hline
\strut\rule{0pt}{9pt}$ N$&
$\varepsilon=1$& $\varepsilon=2^{-5}$&$\varepsilon=2^{-10}$&
$\varepsilon=2^{-15}$& $\varepsilon=2^{-20}$&$\varepsilon=2^{-25}$&$\varepsilon=2^{-30}$\\
\hline
\strut&
\multicolumn{7}{l}{\rule{0pt}{8pt}%
{Errors 
(odd rows) ~\&~
  Computational Rates (even rows)}%
  }
\\
\strut\rule{0pt}{11pt}%
%
 64&3.203e-2	&5.172e-3	&8.450e-4	&1.489e-4	&2.632e-5	&4.653e-6	&8.225e-7	\\
&1.000	&0.997	&0.996	&0.996	&0.996	&0.996	&0.996\\
128&1.602e-2	&2.591e-3	&4.237e-4	&7.469e-5	&1.320e-5	&2.334e-6	&4.126e-7	\\
&1.000	&0.999	&0.999	&0.998	&0.998	&0.998	&0.998\\
256&8.011e-3	&1.296e-3	&2.119e-4	&3.740e-5	&6.611e-6	&1.169e-6	&2.066e-7	\\
&1.000	&1.000	&1.003	&0.999	&0.999	&0.999	&0.999\\
512&4.006e-3	&6.478e-4	&1.057e-4	&1.871e-5	&3.308e-6	&5.848e-7	&1.034e-7	\\
\hline
\strut&
\multicolumn{7}{l}{\rule{0pt}{8pt}%
\color{blue}${\mathcal S}^*=\emptyset$ in \eqref{tau_def}:~ Estimators
(odd rows) ~\&~
  Effectivity Indices (even rows)
  }
\\
\strut\rule{0pt}{11pt}%
%
 64&3.290e-2	&6.984e-3	&1.159e-3	&2.047e-4	&3.618e-5	&6.396e-6	&1.131e-6 \\	
&1.027	&1.350	&1.372	&1.374	&1.375	&1.375	&1.375  \\
128&1.646e-2	&2.695e-3	&5.802e-4	&1.023e-4	&1.808e-5	&3.196e-6	&5.651e-7 \\
&1.027	&1.040	&1.369	&1.370	&1.370	&1.370	&1.370	\\
256&8.230e-3	&1.335e-3	&2.908e-4	&5.115e-5	&9.042e-6	&1.598e-6	&2.826e-7 \\	
&1.027	&1.030	&1.372	&1.368	&1.368	&1.368	&1.368	\\
512&4.115e-3	&6.667e-4	&1.460e-4	&2.558e-5	&4.522e-6	&7.993e-7	&1.413e-7 \\	
&1.027	&1.029	&1.382	&1.367	&1.367	&1.367	&1.367  \\
\hline
\end{tabular}\hfill}
\end{table}
}

{
\begin{table}[t!]
\caption{
{\color{blue}Test problem with $C_u=1$,
 non-obtuse triangulation (see Fig.\,\ref{fig_partial_new}, left):%
 }
  error component \mbox{$\| u-u^I\|_{\Omega}$}  from \eqref{er_decomp} compared with
  the estimator component
  $\|f_h-f_h^I\|_{\Omega}$\,. 
}\label{table3_new}
\tabcolsep=6pt
\vspace{-0.2cm}
{\small\hfill
\begin{tabular}{r|rrrrrrr}
\hline
\strut\rule{0pt}{9pt}$ N$&
$\varepsilon=1$& $\varepsilon=2^{-5}$&$\varepsilon=2^{-10}$&
$\varepsilon=2^{-15}$& $\varepsilon=2^{-20}$&$\varepsilon=2^{-25}$&$\varepsilon=2^{-30}$\\
\hline
\strut&
\multicolumn{7}{l}{\rule{0pt}{8pt}%
{Errors 
(odd rows) ~\&~
  Computational Rates (even rows)}%
  }
\\
\strut\rule{0pt}{11pt}%
%
%
 64&2.242e-4	&6.120e-4	&6.496e-4	&6.567e-4	&6.571e-4	&6.571e-4	&6.571e-4	\\
&2.000	&2.004	&2.006	&2.006	&2.006	&2.006	&2.006	\\
128&5.607e-5	&1.525e-4	&1.617e-4	&1.635e-4	&1.636e-4	&1.636e-4	&1.636e-4	\\
&2.000	&2.002	&2.003	&2.003	&2.003	&2.003	&2.003	\\
256&1.402e-5	&3.807e-5	&4.036e-5	&4.077e-5	&4.079e-5	&4.080e-5	&4.080e-5	\\
&2.000	&2.001	&2.002	&2.002	&2.002	&2.002	&2.002	\\
512&3.505e-6	&9.510e-6	&1.008e-5	&1.018e-5	&1.019e-5	&1.019e-5	&1.019e-5  \\
\hline
\strut&
\multicolumn{7}{l}{\rule{0pt}{8pt}%
\color{blue}${\mathcal S}^*=\emptyset$ in \eqref{tau_def}:~
Estimators
(odd rows) ~\&~
  Effectivity Indices (even rows)
  }
\\
\strut\rule{0pt}{11pt}%
%

 64&2.661e-3	&5.762e-4	&6.484e-4	&6.567e-4	&6.571e-4	&6.571e-4	&6.571e-4	\\
&11.867	&0.941	&0.998	&1.000	&1.000	&1.000	&1.000	\\
128&6.671e-4	&1.435e-4	&1.614e-4	&1.634e-4	&1.636e-4	&1.636e-4	&1.636e-4	\\
&11.897	&0.941	&0.998	&1.000	&1.000	&1.000	&1.000	\\
256&1.670e-4	&3.579e-5	&4.029e-5	&4.077e-5	&4.079e-5	&4.080e-5	&4.080e-5	\\
&11.913	&0.940	&0.998	&1.000	&1.000	&1.000	&1.000	\\
512&4.178e-5	&8.938e-6	&1.006e-5	&1.018e-5	&1.019e-5	&1.019e-5	&1.019e-5   \\
&11.921	&0.940	&0.998	&1.000	&1.000	&1.000	&1.000	\\
\hline
\end{tabular}\hfill}
\end{table}
}

\noindent
This error decomposition is useful as  the two error components in \eqref{er_decomp}
exhibit somewhat different behaviour 
in our experiments, as${}\sim \eps^{1/2}N^{-1}$ and${}\sim N^{-2}$, respectively
(compare, for example, the upper parts of Tables\,\,\ref{table2_new} and\,\,\ref{table3_new}).
Furthermore, one can identify that ${\mathcal E}\bigr|_{f_h:=f_h^I}$
(obtained from \eqref{error_lemma} by replacing $f_h$ with its linear interpolant $f_h^I$)
essentially estimates the first error component in \eqref{er_decomp} (see Table~\ref{table2_new}), while 
$\|f_h-f_h^I\|_{\Omega}$
provides a reasonable estimator for the remaining error component  in \eqref{er_decomp}
(see Table~\ref{table3_new}).
Indeed, for the estimator components in Tables\,\,\ref{table2_new} and\,\,\ref{table3_new} (for the latter, when $\eps\le 2^{-5}$)
on the non-obtuse mesh, the effectivity indices
 do not exceed 1.382 (related results are given in Table~\ref{table1_new}).
 For the mesh with obtuse triangles, analogous results are presented in Tables\,\,\ref{table22} and\,\,\ref{table23}
 (with related results in Table\,\,\ref{table2_newnew}).


{
\begin{table}[tbp]
\caption{
{\color{blue} \color{blue}Test problem with $C_u=1$,
 mesh with obtuse triangles (see Fig.\,\ref{fig_partial_new}, right):}
  error component $\bigl\{\eps^2\|\nabla u_h-(\nabla u)^I\|^2_{2\,;\Omega}+\| u_h-u^I\|^2_{2\,;\Omega}\bigr\}^{1/2}$ from \eqref{er_decomp} compared with
  ${\mathcal E}\bigr|_{f_h:=f_h^I}$\,. 
}\label{table22}
\tabcolsep=6pt
\vspace{-0.2cm}
{\small\hfill
\begin{tabular}{r|rrrrrrr}
\hline
\strut\rule{0pt}{9pt}$ N$&
$\varepsilon=1$& $\varepsilon=2^{-5}$&$\varepsilon=2^{-10}$&
$\varepsilon=2^{-15}$& $\varepsilon=2^{-20}$&$\varepsilon=2^{-25}$&$\varepsilon=2^{-30}$\\
\hline
\strut&
\multicolumn{7}{l}{\rule{0pt}{8pt}%
{Errors 
(odd rows) ~\&~
  Computational Rates (even rows)}%
  }
\\
\strut\rule{0pt}{11pt}%
%
%
 64&3.334e-2	&5.274e-3	&8.452e-4	&1.489e-4	&2.632e-5	&4.653e-6	&8.225e-7	\\
&0.998	&0.997	&0.996	&0.996	&0.996	&0.996	&0.996	\\
128&1.669e-2	&2.642e-3	&4.236e-4	&7.469e-5	&1.320e-5	&2.334e-6	&4.126e-7	\\
&0.999	&0.999	&1.001	&0.998	&0.998	&0.998	&0.998	\\
256&8.352e-3	&1.322e-3	&2.117e-4	&3.740e-5	&6.611e-6	&1.169e-6	&2.066e-7	\\
&1.000	&1.000	&1.005	&0.999	&0.999	&0.999	&0.999	\\
512&4.177e-3	&6.611e-4	&1.055e-4	&1.871e-5	&3.308e-6	&5.848e-7	&1.034e-7   \\
\hline
\strut&
\multicolumn{7}{l}{\rule{0pt}{8pt}%
\color{blue}${\mathcal S}^*=\emptyset$ in \eqref{tau_def}:~ Estimators
(odd rows) ~\&~
  Effectivity Indices (even rows)
  }
\\
\strut\rule{0pt}{11pt}%
%
%
%
 64&3.534e-2	&8.142e-3	&1.516e-3	&2.681e-4	&4.739e-5	&8.378e-6	&1.481e-6	\\
&1.060	&1.544	&1.794	&1.800	&1.801	&1.801	&1.801	\\
128&1.771e-2	&3.552e-3	&8.840e-4	&1.586e-4	&2.803e-5	&4.955e-6	&8.759e-7	\\
&1.061	&1.344	&2.087	&2.123	&2.123	&2.123	&2.123	\\
256&8.864e-3	&1.770e-3	&5.936e-4	&1.174e-4	&2.076e-5	&3.670e-6	&6.487e-7	\\
&1.061	&1.339	&2.804	&3.139	&3.140	&3.140	&3.140	\\
512&4.434e-3	&8.842e-4	&3.927e-4	&1.057e-4	&1.870e-5	&3.305e-6	&5.843e-7	\\
&1.061	&1.337	&3.722	&5.648	&5.652	&5.652	&5.652	\\
\hline
\strut&
\multicolumn{7}{l}{\rule{0pt}{8pt}%
\color{blue}${\mathcal S}^*\neq\emptyset$ in \eqref{tau_def}:~
Estimators
(odd rows) ~\&~
  Effectivity Indices (even rows)
  }
\\
\strut\rule{0pt}{11pt}%
%
%
%
 64&3.534e-2	&8.137e-3	&1.372e-3	&2.422e-4	&4.282e-5	&7.569e-6	&1.338e-6	\\
&1.060	&1.543	&1.623	&1.627	&1.627	&1.627	&1.627	\\
128&1.771e-2	&3.550e-3	&6.886e-4	&1.215e-4	&2.148e-5	&3.797e-6	&6.712e-7	\\
&1.061	&1.343	&1.626	&1.627	&1.627	&1.627	&1.627	\\
256&8.864e-3	&1.769e-3	&3.453e-4	&6.086e-5	&1.076e-5	&1.902e-6	&3.362e-7	\\
&1.061	&1.338	&1.631	&1.627	&1.627	&1.627	&1.627	\\
512&4.434e-3	&8.839e-4	&1.732e-4	&3.046e-5	&5.385e-6	&9.519e-7	&1.683e-7	\\
&1.061	&1.337	&1.641	&1.628	&1.628	&1.628	&1.628	\\
\hline
\end{tabular}\hfill}
\end{table}
}

{
\begin{table}[tbp]
\caption{
{\color{blue} \color{blue}Test problem with $C_u=1$,
 mesh with obtuse triangles (see Fig.\,\ref{fig_partial_new}, right):}
  error component $\| u-u^I\|_{\Omega}$  from \eqref{er_decomp} compared with
  the estimator component
  $\|f_h-f_h^I\|_{\Omega}$\,. 
}\label{table23}
\tabcolsep=6pt
\vspace{-0.2cm}
{\small\hfill
\begin{tabular}{r|rrrrrrr}
\hline
\strut\rule{0pt}{9pt}$ N$&
$\varepsilon=1$& $\varepsilon=2^{-5}$&$\varepsilon=2^{-10}$&
$\varepsilon=2^{-15}$& $\varepsilon=2^{-20}$&$\varepsilon=2^{-25}$&$\varepsilon=2^{-30}$\\
\hline
\strut&
\multicolumn{7}{l}{\rule{0pt}{8pt}%
{Errors 
(odd rows) ~\&~
  Computational Rates (even rows)}%
  }
\\
\strut\rule{0pt}{11pt}%
%
%
 64&2.412e-4	&6.591e-4	&6.978e-4	&7.063e-4	&7.067e-4	&7.067e-4	&7.067e-4	\\
&1.998	&2.001	&2.001	&2.003	&2.003	&2.003	&2.003	\\
128&6.040e-5	&1.646e-4	&1.743e-4	&1.762e-4	&1.763e-4	&1.763e-4	&1.763e-4	\\
&1.999	&2.001	&2.001	&2.002	&2.002	&2.002	&2.002	\\
256&1.511e-5	&4.113e-5	&4.354e-5	&4.399e-5	&4.402e-5	&4.402e-5	&4.402e-5	\\
&1.999	&2.000	&2.001	&2.001	&2.001	&2.001	&2.001	\\
512&3.779e-6	&1.028e-5	&1.088e-5	&1.099e-5	&1.100e-5	&1.100e-5	&1.100e-5	\\
\hline
\strut&
\multicolumn{7}{l}{\rule{0pt}{8pt}%
\color{blue}${\mathcal S}^*=\emptyset$ in \eqref{tau_def}:~
Estimators
(odd rows) ~\&~
  Effectivity Indices (even rows)
  }
\\
\strut\rule{0pt}{11pt}%
%
%
 64&2.878e-3	&6.274e-4	&6.967e-4	&7.062e-4	&7.067e-4	&7.067e-4	&7.067e-4	\\
&11.931	&0.952	&0.998	&1.000	&1.000	&1.000	&1.000	\\
128&7.223e-4	&1.566e-4	&1.740e-4	&1.762e-4	&1.763e-4	&1.763e-4	&1.763e-4	\\
&11.958	&0.951	&0.998	&1.000	&1.000	&1.000	&1.000	\\
256&1.809e-4	&3.912e-5	&4.347e-5	&4.399e-5	&4.402e-5	&4.402e-5	&4.402e-5	\\
&11.972	&0.951	&0.998	&1.000	&1.000	&1.000	&1.000	\\
512&4.528e-5	&9.777e-6	&1.086e-5	&1.099e-5	&1.100e-5	&1.100e-5	&1.100e-5	\\
&11.979	&0.951	&0.998	&1.000	&1.000	&1.000	&1.000	\\
\hline
\strut&
\multicolumn{7}{l}{\rule{0pt}{8pt}%
\color{blue}${\mathcal S}^*\neq\emptyset$ in \eqref{tau_def}:~Estimators
(odd rows) ~\&~
  Effectivity Indices (even rows)
  }
\\
\strut\rule{0pt}{11pt}%
%
%
 64&2.878e-3	&6.274e-4	&6.967e-4	&7.062e-4	&7.067e-4	&7.067e-4	&7.067e-4	\\
&11.931	&0.952	&0.998	&1.000	&1.000	&1.000	&1.000	\\
128&7.223e-4	&1.566e-4	&1.740e-4	&1.762e-4	&1.763e-4	&1.763e-4	&1.763e-4	\\
&11.958	&0.951	&0.998	&1.000	&1.000	&1.000	&1.000	\\
256&1.809e-4	&3.912e-5	&4.347e-5	&4.399e-5	&4.402e-5	&4.402e-5	&4.402e-5	\\
&11.972	&0.951	&0.998	&1.000	&1.000	&1.000	&1.000	\\
512&4.528e-5	&9.777e-6	&1.086e-5	&1.099e-5	&1.100e-5	&1.100e-5	&1.100e-5	\\
&11.979	&0.951	&0.998	&1.000	&1.000	&1.000	&1.000	\\
\hline
\end{tabular}\hfill}
\end{table}
}


{\color{blue}\section{Justification of Remark~\ref{rem_tb_justified}\label{sec_app_B}}
To get a sharper version of the upper bound \eqref{upper_bound} for our estimator, with
$\min\{1,\,\eps h_z^{-1}\} \|\eps J_z\|^2_{\omega_z}$  replaced by a sharper term
$\sum_{S\in\gamma_z}\min\{\eps|S|,|\omega_z|\}\,(\eps{J}_S)^2$, we need to tweak the definition 
of $\btauz$ in \eqref{tauz_h_small}. 
To be more precise, whenever $h_z\lesssim\eps\lesssim H_z$ and $\mathring{h}_z\not\simeq H_z$, let
\begin{subequations}\label{tauz_h_small_optm0}
\begin{align}\label{tau_z_optm0}
\btauz&:=
\btauz'
+\left\{\begin{array}{cl}
\displaystyle
\sum_{i=1,m} {\textstyle\frac12} J_{S_i^-}\bigl(\btau_{z;T_{i-1}}^+ +\btau_{z;T_{i}}^-\bigr)
&\mbox{if~}z\in {\mathcal N}_{\rm ani}\backslash{\pt\Omega},
\\[0.5cm]\displaystyle
{\textstyle\frac12}\mathring{h}_z\mathring{\mathcal J}_z\,\bigl(|S_1^+|^{-1}\btau_{z;T_{1}}^+ +|S_n^-|^{-1}\btau_{z;T_{n}}^-\bigr)
&\mbox{if~$z$~satisfies~}{\mathcal A}1_{\rm mix},
\end{array}
\right.
\\\label{tau_z_pr_optm0}
\btauz' &:= \phi_z\bigl(\alpha_i \bnu_i+\beta_i d_i^{-1}\bmu_i\bigr)\quad\forall\ T_i\subset\omega_z,
\qquad\qquad\quad
\mathring{\mathcal J}_z:=\mathring{h}_z^{-1}\sum_{i=2}^{n}|S_i^-|J_{S_i^-}\,.
\end{align}
\end{subequations}
Here we use the notation of \eqref{tau_z} and \eqref{tau_z_case2_bounds} 
(see Fig.\,\ref{fig_tau_z_case1} and Fig.\,\ref{fig_local_kappa} (left)), assuming that $\btau_{z;T}^\pm\in L_2(\Omega)$ has support on $T$, while
 $m=m_z$, $n=n_z$ are defined in the proof of Lemma~\ref{lem_beta}.
 If $z\in{\mathcal N}^*_{\pt\Omega}$, or $z\in\pt\Omega$ satisfies ${\mathcal A}1_{\rm mix}$ and $n_z=1$, the definition \eqref{tauz_h_small} of $\btauz$ remains unchanged.
\medskip

\begin{lemma}
Let $\btauz$ be modified to \eqref{tauz_h_small_optm0}
whenever $h_z\lesssim\eps\lesssim H_z$ and $\mathring{h}_z\not\simeq H_z$.
Set $\{\alpha_i,\beta_i\}_{i=1}^{N_z}$  in  \eqref{tauz_h_small_optm0} to minimize  \eqref{optima}
subject to the constraint $[\btauz\cdot \bnu]
=\phi_z[\partial_\bnu u_h]$ on $\gamma_z$.
Then Theorem~\ref{theo_main_bounds}(i) is valid with
the term
$\min\{1,\,\eps h_z^{-1}\}^{1/2} \|\eps J_z\|_{\omega_z}$  in \eqref{tau_z_main} replaced by a sharper
$\bigl\{\sum_{S\in\gamma_z}\min\{\eps|S|,|\omega_z|\}\,(\eps{J}_S)^2\bigr\}^{1/2}$. 
\end{lemma}%
\medskip

\begin{proof}
For the case $h_z\gtrsim \eps$, the sharper version of \eqref{tau_z_main} follows from the first relation in \eqref{theo_bound}.
Otherwise,
if $h_z\le H_z\lesssim  \eps$, this follows from $ |\omega_z|\lesssim \eps|S|$.

For the remaining case $h_z\lesssim\eps\lesssim H_z$,
using the notation $\mathring{J}_z$ and $\widehat J_z$ of \eqref{mathring_J},
we need to show \eqref{tau_z_main} with $\min\{1,\,\eps h_z^{-1}\}^{1/2} \|\eps J_z\|_{\omega_z}$
replaced by $\{\eps \mathring{h}_z\}^{1/2}\,\eps\mathring{J}_z+\|\eps \widehat J_z\|_{\omega_z}$.
For $\btauz-\btauz'$, we employ Lemma~\ref{lem_case2};
in particular, \eqref{tau_z_case2_bounds_b} implies
$\|\eps^2{\rm div}(\btauz-\btauz')\|^2_{\omega_z}\sim\|\eps(\btauz-\btauz')\|^2_{\omega_z}
\sim\eps \mathring{h}_z(\eps\mathring{J}_z)^2$.
Recalling that
$g_z=\eps^2{\rm div}\btauz+\theta_{T;z} F_{T;z}$ for $h_z\lesssim \eps$, it suffices to prove the desired version of \eqref{tau_z_main}
for $\|\eps\btauz'\|_{\omega_z}+\|\eps^2{\rm div}\btauz'+\theta_{T;z} F_{T;z}\|_{\omega_z}$.
In fact, it suffices for the latter to be established for one specific set $\{\alpha_i, \beta_i\}_{i=1}^{N_z}$
subject to $[\btauz\cdot \bnu]
=\phi_z[\partial_\bnu u_h]$ on $\gamma_z$.
Here the constraint is equivalent to a version of \eqref{tau_n_hat} taking into account the possibly non-trivial jumps $[(\btauz-\btauz')\cdot \bnu]$ across $\gamma_z$.

As in the proof of Lemma~\ref{lem_beta}, consider three cases (a),\,(b) and (c).

(a) Suppose that $z$ satisfies ${\mathcal A}1_{\rm mix}$ with $n_z\ge 2$.
Now, let
\beq\label{A_z_def}
\alpha_i:=\eps^{-2}d_i\theta_{T_i;z} (\widetilde F_{T_i;z}-A_z)\quad
:
\quad
A_z\sum_{i=1}^N 2\eps^{-2}\theta_{T_i;z}|T_i| :=-\mathring{h}_z\mathring{\mathcal J}_z.
\eeq
Now, the constraint $[\btauz\cdot \bnu]
=\phi_z[\partial_\bnu u_h]$ on $\gamma_z$ yields a version of \eqref{tau_n_hat}, in which 
$\frac12\mathring{h}_z\mathring{\mathcal J}_z(\mathbbm{1}_{i=2}+\mathbbm{1}_{i=n})$
is subtracted from the right-hand side. 
Note that the described version of \eqref{tau_n_hat}
gives
 a consistent system for $\{\beta_i\}_{i=1}^N$ with infinitely many solutions, which is shown as in the proof of Lemma~\ref{lem_tau_case1}.
In particular, if $z\not\in\pt\Omega$, the consistency of this system can be shown by adding all $N$ equations in this system 
(and also using $\bnu_i\cdot\bigl(|S_i^+|\bnu^+_i +|S_{i}^-|\bnu^-_{i}\bigr)+2|T_i|d_i^{-1}=0$), which yields the second relation in \eqref{A_z_def}.
Note that the latter uniquely defines $A_z$ and implies $|A_z|\lesssim \eps^2 |\omega_z|^{-1}\mathring{h}_z\mathring{J}_z$.

Next, set $\beta_1:=0$, and imitate the proof of Lemma~\ref{lem_beta}.
Now an application of $\sum_{i=2}^n$ to the current version of \eqref{tau_n_hat} yields
$|\beta_n d_n^{-1}|\sim |\beta_n \mathring{h}_z^{-1}|\lesssim \max_{ j=1,\ldots, n}|\alpha_j|$.
Consequently, a version of \eqref{beta_aux} implies
$|\beta_i d_i^{-1}|\lesssim \mathring{J}_z+\max_{ j=1,\ldots, i}|\alpha_j|$ for $i=2,\ldots, n-1$, and
$|\beta_i d_i^{-1}|\lesssim \widehat J_z+\max_{ j=1,\ldots, i}|\alpha_j|$ for $i=n+1,\ldots, N$.
Note that for $i=2,\ldots, n-1$, one has $|T_i|\sim \mathring{h}_z^2\lesssim \eps\mathring{h}_z$ so $\|\eps \mathring{J}_z\|_{T_i}\lesssim \{\eps\mathring{h}_z\}^{1/2}\,\eps \mathring{J}_z$.
Note also that 
$\|\eps \widehat{J}_z\|_{T_i}\le\|\eps \widehat{J}_z\|_{\omega_z}$
for $i=n+1,\ldots, N$.
Comparing these observations with the desired version of \eqref{tau_z_main} implies that to bound $\|\eps\btauz'\|_{\omega_z}$,
 it remains to estimate $\max_{i=1,\ldots, N}\|\eps\alpha_i\|_{\omega_z}$.

For the latter, recall that it was shown in the proof of Lemma~\ref{lem_tau_case1} that if one sets $A_z:=0$ in the current definition of $\alpha_i$, then
$\max\|\eps\alpha_i\|_{\omega_z}\lesssim \|h_z\eps^{-1}f_h^I\|_{\omega_z}$.
For the remaining component $\eps^{-2}d_i\theta_{T_i;z} A_z$ of $\alpha_i$,
recall that $d_i\theta_{T_i;z}\sim h_{T_i}\lesssim h_z\lesssim \eps$, so $\eps|\eps^{-2}d_i\theta_{T_i;z} A_z|\lesssim |A_z|$.
On the other hand,
 $|\eps^2{\rm div}\btauz'+\theta_{T;z} F_{T;z}|\le |A_z|$ (it is a current version of \eqref{tauz_small_h_prop_a}).
 Hence, to complete the estimation of $\|\eps\btauz'\|_{\omega_z}$, as well as to 
 bound
$\|\eps^2{\rm div}\btauz'+\theta_{T;z} F_{T;z}\|_{\omega_z}$, we proceed to the bound
$\|A_z\|_{\omega_z}\lesssim |\omega_z|^{1/2}(\eps^2 |\omega_z|^{-1}\mathring{h}_z\mathring{J}_z)
\lesssim\{\eps\mathring{h}_z\}^{1/2}  \eps\mathring{J}_z$
(in view of $\{\eps\mathring{h}_z\}^{1/2}\lesssim |\omega_z|^{1/2}$).
This observation completes the proof of the desired version of \eqref{tau_z_main} in case (a).

(b) Next, consider $z\in {\mathcal N}_{\rm ani}\backslash\pt\Omega$ that satisfies~${\mathcal A}1_{\rm ani}$ (and so not~${\mathcal A}1_{\rm mix}$).
Let
\beq
\label{A_iz_def}
\alpha_i:=\eps^{-2}d_i\theta_{T_i;z} (\widetilde F_{T_i;z}-a_i)\quad
:
\quad
\sum_{i=1}^N 2\eps^{-2}\theta_{T_i;z}|T_i|a_i =-\sum_{i= 1,m}|S_{i}^-|J_{S_i^-}\,.
\eeq
Now, the constraint $[\btauz\cdot \bnu]
=\phi_z[\partial_\bnu u_h]$ on $\gamma_z$ yields a version of \eqref{tau_n_hat}, in which
$\mathbbm{1}_{i\in\{ 1,m\}} |S_{i}^-|J_{S_{i}^-}$
is subtracted from the right-hand side. 
Note that the described version of \eqref{tau_n_hat}
gives
 a consistent system for $\{\beta_i\}_{i=1}^N$ with infinitely many solutions, which is shown as in the proof of Lemma~\ref{lem_tau_case1}.
To be more precise, adding all $N$ equations in this system
(and also using $\bnu_i\cdot\bigl(|S_i^+|\bnu^+_i +|S_{i}^-|\bnu^-_{i}\bigr)+2|T_i|d_i^{-1}=0$) yields the second relation in \eqref{A_iz_def}.

Set
$$
a_i:=\left\{\begin{array}{cl}
A_z^+& \mbox{for}\;\;i=1,\ldots,m,\\
A_z^-& \mbox{for}\;\;i=m+1,\ldots,N,
\end{array}\right.
\qquad
A_z^+\sum_{i=1}^m 2\eps^{-2}\theta_{T_i;z}|T_i|:=\widetilde\sigma_z\,,
$$
where $\widetilde\sigma_z$ is from \eqref{sigma_def}, while $A_z^-$ is now uniquely defined by the second relation in \eqref{A_iz_def}.
Then, 
$|A_z^+|\lesssim \eps H_z^{-1}|\eps h_z^{-1}\widetilde\sigma_z|$
 and $|A_z^-|\lesssim |A_z^+|+\eps H_z^{-1}|\eps {J}_z|$.
 Combining these two observations with
 \eqref{hat_sigma} (which was obtained under assumption ${\mathcal A}1_{\rm ani}$), one gets
\beq\label{app_opt_Az_bound}
A_z:=\max_{i=1,\ldots,N}|a_i|\lesssim \eps H_z^{-1}|\eps J_z|+h_z\eps^{-1}|\widebar F_z|
+\sum_{T\subset\omega_z}\lambda_T\,{\rm osc}(f_h^I;T)
\eeq
(where we also used $\eps H_z^{-1}\lesssim 1$ for the final two terms).

With these definitions, 
one gets a version of Lemma~\ref{lem_tau_case1}:
\begin{subequations}\label{tauz_small_h_prop_opt}
\begin{align}\label{tauz_small_h_prop_a_opt}
|\eps^2{\rm div}\btauz'+\theta_{T;z}\widetilde F_{T;z}|&\le A_z\qquad \forall\ T\subset\omega_z,
\\\label{tauz_small_h_prop_b_opt}
\|\eps \btauz'\|_{\omega_z}&\lesssim
\|
h_z\eps^{-1}(|f^I_h|+A_z)
\|_{\omega_z}+\eps\,\Bigl\{\sum_{i=1}^{N_z} \beta^2_i d_i^{-2} |T_i|\Bigr\}^{1/2}.
\end{align}
\end{subequations}
Furthermore, the current version of \eqref{tau_n_hat} implies a version of
\eqref{beta_0_m} with $0$ in the right-hand side:
$\beta_{0}-\beta_{m}
+\alpha_{0}\bnu_{0}\cdot|S_{0}^+|\bnu^+_{0}
-\alpha_{m}\bnu_{m}\cdot|S_{m}^+|\bnu^+_{m}=0$.
So for $i=0,m$, set
$\beta_i:=-\alpha_{i}\bnu_{i}\cdot|S_{i}^+|\bnu^+_{i}$.
For the remaining $\{\beta_i\}$, a version of
\eqref{beta_aux} in which now $J_z$ is replaced by $\widehat J_z$, yields
$$
|\beta_id_i^{-1}|\lesssim \max_{j=1,\ldots, N}|\alpha_j|+|\widehat J_z|,
$$
(this is also true for $i=0,m$).
Now, imitating the proof of Lemma~\ref{lem_tau_case1},
the second term in the right-hand side of \eqref{tauz_small_h_prop_b_opt} is bounded by the first term${}+\|\eps\widehat J_z\|_{\omega_z}$.
As $h_z\eps^{-1}\lesssim 1$, 
so $\|h_z\eps^{-1}A_z\|_{\omega_z}\lesssim \|A_z\|_{\omega_z}$, 
so it now remains to estimate $\|A_z\|_{\omega_z}$.
In fact, the terms in the right-hand side of the bound \eqref{app_opt_Az_bound} for $A_z$ were estimated in the proof of Lemma~\ref{lem_beta}, except for the component
$\eps H_z^{-1}|\eps J_z|$. For the latter,
$\|\eps H_z^{-1}(\eps J_z)\|_{\omega_z}=|\omega_z|^{1/2}\eps H_z^{-1}|\eps J_z|\lesssim \{\eps h_z\}^{1/2}|\eps J_z|$
(in view of $ \{|\omega_z|\eps H_z^{-1}\}^{1/2}\sim \{\eps h_z\}^{1/2}$ combined with $\eps H_z^{-1}\lesssim 1$).
Combining this with $h_z\sim \mathring{h}_z$ (as $z\in{\mathcal N}_{\rm ani}$), and $\eps h_z\lesssim |\omega_z|$, one gets
$\|\eps H_z^{-1}(\eps J_z)\|_{\omega_z}\lesssim \{\eps \mathring{h}_z\}^{1/2}\,\eps\mathring{J}_z+\|\eps \widehat J_z\|_{\omega_z}$.
This completes the proof for case (b).

(c) If either $z\in{\mathcal N}^*_{\pt\Omega}$ satisfies~${\mathcal A}1_{\rm ani}$ but not~${\mathcal A}1_{\rm mix}$,
or $z\in\pt\Omega$ satisfies ${\mathcal A}1_{\rm mix}$ and $n_z=1$,
then 
$\mathring{\gamma}_z=\emptyset$, so
$J_z$ does not involve $\mathring{J}_z$ (i.e. $J_z=\widehat J_z$), so the original version of \eqref{tau_z_main} is equivalent to the desired version of this bound.
\end{proof}}

\end{document}